\documentclass[letterpaper,10pt]{amsart}
\usepackage[utf8]{inputenc}
\usepackage[english]{babel}
\usepackage{mathtools,amsthm,amssymb}
\usepackage{tikz-cd}
\usepackage{url}

\let\oldtocsubsection=\tocsubsection
\renewcommand{\tocsubsection}[2]{\hspace{.75cm}\oldtocsubsection{#1}{#2}}
\DeclareRobustCommand{\SkipTocEntry}[5]{}

\usepackage{enumitem}
\setenumerate[1]{label={\alph*)}}

\numberwithin{equation}{section}

\theoremstyle{definition}
\newtheorem{defn}{Definition}[section]
\newtheorem{rem}[defn]{Remark}
\theoremstyle{plain}
\newtheorem{prop}[defn]{Proposition}
\newtheorem{lemma}[defn]{Lemma}
\newtheorem{thm}[defn]{Theorem}
\newtheorem{cor}[defn]{Corollary}

\newcommand{\dee}{{\mathrm{d}}}

\DeclareMathOperator{\id}{id}

\DeclareMathOperator{\Ad}{Ad}

\DeclareMathOperator{\tr}{tr}
\DeclareMathOperator{\im}{im}
\DeclareMathOperator{\ev}{ev}

\DeclareMathOperator{\Hom}{Hom}
\DeclareMathOperator{\Ext}{Ext}

\DeclareMathOperator{\Aff}{Aff}

\newcommand{\totalK}{\underline{K}}
\newcommand{\Kalg}{\overline{K}_1^{\mathrm{alg}}}
\newcommand{\Mult}[1]{{\mathcal{M}(#1)}}
\newcommand{\Q}[1]{{\mathcal{Q}(#1)}}
\newcommand{\slim}{\mathop{\operatorname{s^*-lim}}_{n\to\infty}}

\newcommand{\ca}{$C^*$-algebra}
\newcommand{\sh}{$*$-homomorphism}
\newcommand{\Hb}{{\mathcal{H}}}
\newcommand{\K}{{\mathcal{K}}}
\newcommand{\Z}{{\mathcal{Z}}}
\newcommand{\Oinf}{{\mathcal{O}_\infty}}
\newcommand{\Otwo}{{\mathcal{O}_2}}

\title{A unified approach for classifying simple nuclear $C^\ast$-algebras}
\author{Ben Bouwen}
\address{Department of Mathematics and Computer Science, University of Southern Denmark, 5230 Odense, Denmark}
        \email{bouwen@imada.sdu.dk}
\author{James Gabe}
\address{Department of Mathematics and Computer Science, University of Southern Denmark, 5230 Odense, Denmark}
        \email{gabe@imada.sdu.dk}
\date{}
\thanks{This was supported by DFF grants 1054-00094B and 1026-00371B}

\begin{document}

\begin{abstract}
We provide a new proof of the Kirchberg--Phillips theorem by adapting the framework laid out by Carrión--Gabe--Schafhauser--Tikuisis--White for classifying separable simple unital nuclear stably finite $\mathcal Z$-stable $C^\ast$-algebras satisfying the UCT. Not only does this give a unified approach to classifying stably finite and purely infinite $C^\ast$-algebras, in contrast to the other proofs of the Kirchberg--Phillips theorem, our proof does not rely on Kirchberg's Geneva Theorems, but instead implies them as corollaries (for nuclear $C^\ast$-algebras).  
\end{abstract}

\maketitle

\renewcommand*{\thedefn}{\Alph{defn}}

\section{Introduction}

The Kirchberg–Phillips theorem, proved independently by Kirchberg (\cite{Kirchberg_TheClassificationofPurelyInfiniteCalgebrasUsingKasparovsTheory}) and Phillips (\cite{Phillips_AClassificationTheoremforNuclearPurelyInfiniteSimpleCalgebras}) in the 90s, marked a major advancement in the Elliott classification program. It provided a complete classification by $K$-theory of separable, nuclear, simple, purely infinite $C^*$-algebras satisfying the universal coefficient theorem (UCT). In contrast, its stably finite counterpart took the joint effort of numerous researchers (\cite{GongLinNiu_AClassificationOfFiniteSimpleAmenableZstableCalgebrasI}, \cite{GongLinNiu_AClassificationOfFiniteSimpleAmenableZstableCalgebrasII}, \cite{ElliottGongLinNiu_OnTheClassificationOfSimpleAmenableCalgebrasWithFiniteDecompositionRank}, \cite{TikuisisWhiteWinter_QuasidiagonalityOfNuclearCalgebras} and the vast body of work leading up to these papers) over the past three decades to complete. Combining the classification theorems on both sides gives a comprehensive classification of all separable, nuclear, unital, simple, $\Z$-stable \ca s satisfying the UCT, a landmark result in classification of \ca s; we will refer to \ca s satisfying these conditions as classifiable.
In recent years, a new, more abstract approach to classification was developed by Carrión, the author JG, Schafhauser, Tikuisis and White in \cite{CarrionGabeSchafhauserTikuisisWhite_ClassifyingHomomorphismsIUnitalSimpleNuclearCalgebras}. This approach is carried out in the stably finite setting, while for the purely infinite case, the original Kirchberg–Phillips theorem is used. This paper adapts the new abstract framework and techniques to the purely infinite setting to prove the Kirchberg–Phillips theorem, paving the way to a unified approach for classifying \ca s. 
One of the major upsides of the new approach is that Kirchberg's Geneva theorems in the nuclear setting (Corollary \ref{intro:Genevatheorems}) are obtained as straightforward consequences of classification. These were crucial and highly non-trivial ingredients in the original proof, R\o rdam's summary \cite{Rordam_ClassificationofNuclearCalgebras} and Gabe's alternative proof \cite{Gabe_ClassificationofOinftystableCalgebras}. With the new approach, we no longer require these results as prerequisites.

\subsection*{Purely infinite \ca s}
The distinction between stably finite and purely infinite \ca s is in complete analogy with the dichotomy between semifinite and infinite von Neumann algebras. There exist simple, separable and nuclear \ca s which fall between these two classes (\cite{Rordam_ASimpleCalgebraWithAFiniteAndAnInfiniteProjection}), but due to a result by Kirchberg (\cite[Theorem 4.1.10]{Rordam_ClassificationofNuclearCalgebras}), this behavior cannot occur for classifiable \ca s. Moreover, the dichotomy is witnessed by the presence of traces: a classifiable \ca\ is stably finite if and only if it has a tracial state, and is purely infinite otherwise.

Purely infinite, simple \ca s were introduced by Cuntz as \ca ic analogues of type III factors in \cite{Cuntz_SimpleCalgebrasGeneratedbyIsometries}, where he defined the Cuntz algebras $\mathcal{O}_n$ and showed that they are examples of such \ca s. Simple and purely infinite \ca s arise naturally as $C^*$-algebraic constructions from various dynamical systems, examples being Cuntz--Krieger algebras (derived from symbolic dynamics) and crossed products (\cite{AnantharamanDelaroche_PurelyInfiniteCalgebrasArisingFromDynamicalSystems}). Building on deep work in tensor products and exactness, Kirchberg announced breakthrough results at the 1994 ICM satellite conference in Geneva: the $\Otwo$-embedding theorem, stating that every separable, exact \ca\ embeds into $\Otwo$ (\cite{Kirchberg_EmbeddingOfExactCalgebrasInTheCuntzAlgebraOtwo}), and the $\Oinf$-absorption theorem, which says that a separable, simple, nuclear \ca\ is purely infinite if and only if it absorbs $\Oinf$ tensorially (\cite{Kirchberg_ExactCalgebrasTensorProductsAndTheClassificationOfPurelyInfiniteAlgebras}). These results known commonly as \emph{Kirchberg's Geneva Theorems} were key ingredients in the proofs of the Kirchberg--Phillips theorem.

Following this theorem, more general classification results in the purely infinite setting were obtained in the subsequent decades. Kirchberg outlined a strategy to classify all separable, nuclear, $\Oinf$-stable \ca s up to stable isomorphism by ideal-related $KK$-theory. This result was proved (in a very different way) in \cite{Gabe_ClassificationofOinftystableCalgebras} by the author JG.

\subsection*{An abstract approach to classification}
On the stably finite side, progress was slower and more involved. The prevailing strategy here consisted of a very careful analysis of the internal structure of these \ca s. For a more in-depth survey, see for example \cite[\S 1.2]{CarrionGabeSchafhauserTikuisisWhite_ClassifyingHomomorphismsIUnitalSimpleNuclearCalgebras}. More recently, a distinct, more abstract strategy to classify stably finite \ca s was laid out by Carrión, Gabe, Schafhauser, Tikuisis and White (\cite{CarrionGabeSchafhauserTikuisisWhite_ClassifyingHomomorphismsIUnitalSimpleNuclearCalgebras}). 

Ever since Elliott's classification of $A\mathbb T$-algebras (\cite{Elliott_OnTheClassificationOfCalgebrasOfRealRankZero}) and R\o rdam's classification of certain crossed products (\cite{Rordam_ClassificationofCertainInfiniteSimpleCAlgebras}), one of the guiding principles in the classification program in general has been to first classify \sh s between \ca s, and then deduce a classification of the \ca s themselves by use of an intertwining argument. In the unital setting, we then aim to classify unital \sh s between unital \ca s $A$ and $B$. In practice, it turns out to be much more tractable to first classify approximate \sh s, sequences of maps between $A$ and $B$ which become \sh s in the limit. These can be encoded more elegantly as honest \sh s from $A$ into the \emph{sequence algebra} $B_\infty := \ell^\infty(B)/c_0(B)$. While $B$ is assumed to be simple, $B_\infty$ is not, so we need to put a simplicity condition on the allowed maps. A \sh\ $\varphi: A \to B$ is said to be \emph{full} if $\varphi(a)$ is full for each non-zero $a\in A$, i.e.\ $\varphi(a)$ generates $B$ as an ideal. The equivalence relation up to which one aims to classify maps is approximate unitary equivalence, as this is what is necessary in order to apply the intertwining argument. In the sequence algebra, this corresponds to unitary equivalence, so in other words, the main objective becomes classifying unital and full \sh s $A\to B_\infty$ up to unitary equivalence.

In the new approach, the main object of study is the so-called \emph{trace-kernel extension} of $B$, a short exact sequence given by
\begin{equation}
\begin{tikzcd}
    0 \ar[r] & J_B \ar[r,"j_B"] & B_\infty \ar[r,"q_B"] & B^\infty \ar[r] & 0,
\end{tikzcd}
\end{equation}
where
\begin{equation}
    J_B := \{ (x_n)_n \in B_\infty \mid \textstyle{\sup_{\tau\in T(B)}}\tau(x_n^*x_n) \xrightarrow{n\to\infty} 0\},
\end{equation}
$T(B)$ is the set of tracial states on $B$ and $B^\infty := B_\infty/J_B$. This construction becomes especially nice in (and was partially inspired by) the case where $B$ has a unique trace $\tau$. In this setting (and after replacing the sequence algebra with the ultrapower $B_\omega$), the quotient is isomorphic with the von Neumann algebra $(\pi_\tau(B)'',\tau)^\omega$, the tracial von Neumann ultrapower of $\pi_\tau(B)''$, with $\pi_\tau$ the GNS representation associated to $\tau$.

The main idea is now to first use von Neumann algebraic classification to classify maps into $B^\infty$. In general, $B^\infty$ will not be a von Neumann algebra, but it still retains enough of this flavor to allow for von Neumann algebraic classification techniques to be used; see \cite[Theorem A]{CastillejosEvingtonTikuisisWhite_ClassifyingMapsIntoUniformTracialSequenceAlgebras}. Then, one classifies lifts of those maps to $B_\infty$ (\cite[Theorem 1.2]{CarrionGabeSchafhauserTikuisisWhite_ClassifyingHomomorphismsIUnitalSimpleNuclearCalgebras}), resulting in the desired classification of maps into $B_\infty$ (\cite[Theorem 1.1]{CarrionGabeSchafhauserTikuisisWhite_ClassifyingHomomorphismsIUnitalSimpleNuclearCalgebras}). The classification of lifts is the major novelty; it makes use of extension theory and its connection with $KK$-theory to determine when a unital and full \sh\ into the quotient lifts in terms of $KK$-theoretic data.

The invariant in the new approach is $\totalK T_u$, the collection of the total $K$-theory (denoted $\totalK$), the Hausdorffized unitary algebraic $K_1$-group $\Kalg$, the system of affine functions on the trace space $\Aff T$ (which naturally corresponds to the trace space itself by Kadison duality), and all the natural maps between these objects (see \cite[\S 2]{CarrionGabeSchafhauserTikuisisWhite_ClassifyingHomomorphismsIUnitalSimpleNuclearCalgebras}).

\subsection*{The main results}
In the purely infinite setting, there are some pivotal differences. The invariant $\totalK T_u$ reduces to total $K$-theory, as there are no traces and $\Kalg$ is naturally isomorphic to $K_1$ itself. Due to the absence of traces, it is also clear that the same construction does not work. A natural candidate to replace traces would be states, but the direct analogue of $J_B$, taking the supremum over all states (or even all pure states), would not work, as the norm of any positive element is always attained by some (pure) state. An alternative would be to fix a single state $\rho$ on $B$ and define\footnote{The definition contains the symmetrizing term $\rho(x_nx_n^*)$ in order to make the resulting algebra closed under adjoints. It does not appear in the definition of $J_B$ as it is superfluous for traces due to the trace identity.}
\begin{equation}
    J_{B,\rho} := \{ (x_n)_n \in B_\infty \mid \rho(x_n^*x_n) + \rho(x_nx_n^*) \to 0\},
\end{equation}
which we call the \emph{state-kernel} of $B$ associated with $\rho$.

The main issue with working with states instead of traces is that $J_{B,\rho}$ is no longer an ideal in $B_\infty$ in general, so one needs to restrict $B_\infty$ to a $C^*$-subalgebra in which it is. To define a suitable $C^*$-subalgebra, we will assume $\rho$ is faithful on $B$ and\footnote{Faithfulness on $\pi_\rho(B)''$ is not implied by faithfulness on $B$ in general; see Remark \ref{rem:faithfulvsstrongly}.} on $\pi_\rho(B)''$, the von Neumann completion of the GNS representation with respect to $\rho$. This property we call \emph{strongly faithful} (see Definition \ref{def:stronglyfaithful}). We then define a $C^*$-subalgebra $S_{B,\rho}$ of $B_\infty$ such that the quotient $S_{B,\rho}/J_{B,\rho}$ is $\pi_\rho(B)''$ (see Corollary \ref{cor:ske_maptoquotient}), resulting in the short exact sequence
\begin{equation}
\begin{tikzcd}
    0 \ar[r] & J_{B,\rho} \ar[r,"j_B"] & S_{B,\rho} \ar[r,"q_B"] & \pi_\rho(B)'' \ar[r] & 0,
\end{tikzcd}
\end{equation}
which we call the \emph{reduced state-kernel extension}. These objects have been studied in the past; see for example \cite[Proposition 9.2.7]{BrownOzawa_CalgebrasandFiniteDimensionalApproximations}. Comparing with the trace-kernel extension in the situation where $B$ has a unique trace, the restriction to $S_{B,\rho}$ amounts to removing the sequence algebra/ultrapower from the quotient $\pi_\rho(B)''$.

This choice of extension allows us to pick $\rho$ appropriately such that the quotient $\pi_\rho(B)''$ is isomorphic to $B(\Hb)$ (see Corollary \ref{cor:stronglyfaithfulstateexists}). The corresponding von Neumann algebraic result we then use is classification of maps into type-I von Neumann algebras, i.e.\ Voiculescu's theorem. In other words, we do not need to rely on the (highly technical and deep) classification results for type-III factors, the von Neumann algebraic counterpart of simple and purely infinite \ca s. This gives a moral explanation for why purely infinite classification was completed before its stably finite counterpart in the $C^*$-algebraic setting, while for von Neumann algebras the opposite is true.

The classification of lifts we obtain (Theorem \ref{thm:ClassifyingLifts}) is the following:

\begin{thm}[Classification of lifts] \label{intro:ClassifyingLifts}
    Let $A$ be a separable, nuclear \ca, $B$ a unital, simple and purely infinite \ca\ and $\rho$ a strongly faithful state on $B$. Suppose $\theta: A \to \pi_\rho(B)''$ is a unitizably full \sh,  i.e.\ its forced unitization $\theta^\dagger: A^\dagger \to \pi_\rho(B)''$ is full.
    \begin{enumerate}
        \item Existence: for any $\kappa \in KK(A,S_{B,\rho})$ satisfying $KK(A,q_B)(\kappa) = [\theta]_{KK(A,\pi_\rho(B)'')}$, there exists a \sh\ $\varphi: A \to S_{B,\rho}$ that lifts $\theta$ and such that $[\varphi]_{KK(A,S_{B,\rho})}=\kappa$;
        \item Uniqueness: if $\psi_1, \psi_2: A\to S_{B,\rho}$ are two \sh s lifting $\theta$ such that $[\psi_1,\psi_2]_{KL(A,J_{B,\rho})} = 0$, then they are unitarily equivalent as maps into $B_\omega$.
\end{enumerate}
\end{thm}

This is the purely infinite analogue of the lifting results from  and \cite[Theorem 7.1]{CarrionGabeSchafhauserTikuisisWhite_ClassifyingHomomorphismsIUnitalSimpleNuclearCalgebras}. Note that we do not need to assume $\pi_\rho(B)''\cong B(\Hb)$ for this theorem, and that we obtain uniqueness up to unitary equivalence by a unitary in $B_\omega$. The latter is a consequence of the fact that we restricted to $S_{B,\rho}$, which does not behave the same as the entire sequence algebra/ultrapower.

We then combine the classification of maps into the quotient with Theorem \ref{intro:ClassifyingLifts} to classify full and unital maps into $S_{B,\rho}$. We do not get a classification of maps into all of $B_\infty$, but our result suffices because we are only interested in maps into $B$, which factor through $S_{B,\rho}$ as it contains $B$. We thus obtain a classification of unital embeddings from $A$ to $B$ by KK-theoretic data (Theorem \ref{thm:ClassifyingEmbeddings}), resulting in the original classification by $K$-theory under the UCT assumption (Theorem \ref{thm:ClassifyingEmbeddingsbyK} and Theorem \ref{thm:KPKtheory}):

\begin{thm}\label{intro:KP}
    Let $A$ be a unital, separable, nuclear \ca\ satisfying the UCT, and $B$ a unital, simple, separable and purely infinite \ca. Then each element of $\Hom_\Lambda(\totalK(A), \totalK(B))$ preserving the $K_0$-class of the unit is induced by a unital, injective \sh\ from $A$ to $B$, and uniquely so up to approximate unitary equivalence.\\
    In particular, two unital Kirchberg algebras $A$ and $B$ satisfying the UCT are isomorphic if and only if $(K_*(A),[1_A]_0) \cong (K_*(B),[1_B]_0)$.
\end{thm}

As mentioned earlier, we recover Kirchberg's celebrated Geneva theorems in the nuclear case as straightforward corollaries of classification (Corolarries \ref{cor:Otwoembedding}, \ref{cor:Otwoabsorption} and \ref{cor:Oinfabsorption}):

\begin{cor}[Kirchberg's Geneva theorems] \label{intro:Genevatheorems}
    Let $A$ be a unital, separable and nuclear \ca. Then
    \begin{enumerate}
        \item there exists a unital embedding $A\hookrightarrow \Otwo$;
        \item if $A$ is also simple, then $A\otimes \Otwo \cong \Otwo$;
        \item if $A$ is also simple and purely infinite, then $A\otimes \Oinf \cong \Oinf$.
    \end{enumerate}
\end{cor}

The main ingredients are the KK-existence theorem, the Elliott-Kucerovsky theorem (\cite{ElliottKucerovsky_AnAbstractVoiculescuBrownDouglasFillmoreAbsorptionTheorem}), which gives a concrete characterization for absorbing maps, and Voiculescu's theorem, as mentioned earlier. Outside of these results and core properties of $KK$-theory and extension theory, the proof is essentially self-contained.

\subsection*{Acknowledgements.} The authors would like to thank José R. Carrión, Christopher Schaf\-hauser, Aaron Tikuisis and Stuart White for allowing us to include Proposition \ref{prop:elku_purelylargeequiv} and its proof, which is to appear in the follow-up paper to \cite{CarrionGabeSchafhauserTikuisisWhite_ClassifyingHomomorphismsIUnitalSimpleNuclearCalgebras}.

\numberwithin{defn}{section}
\section{Preliminaries}

\subsection{Extension theory}

\begin{defn}
    Let $A$ and $I$ be \ca s. An \emph{extension} of $A$ by $I$ is a short exact sequence
    \begin{equation}
    \begin{tikzcd}
        \mathfrak e: 0 \ar[r] & I \ar[r,"j"] & E \ar[r,"q"] & A \ar[r] & 0.
    \end{tikzcd}
    \end{equation}
    An extension is said to be \emph{unital} if $E$ (and therefore also $A$ and $q$) is unital, and \emph{separable} if $E$ (and therefore also $I$ and $A$) is separable.
\end{defn}

The following definitions cover some concepts we will use in \S \ref{sec:KK} and onwards.

\begin{defn}
    Given an extension
\begin{equation}
    \begin{tikzcd}
    \mathfrak e: 0 \arrow[r] & I \arrow[r,"j"] & E \arrow[r,"q"] & D \arrow[r] & 0
    \end{tikzcd}
\end{equation}
of \ca s and a \sh\ $\varphi: A\rightarrow D$ from a \ca\ $A$, we define the \emph{pull-back extension} $\mathfrak e \circ \varphi$ by the commutative diagram
\begin{equation} \label{eq:sep_pullback}
\begin{tikzcd}
    \mathfrak e \circ \varphi: 0 \arrow[r] & I \arrow[r,"\iota"] \ar[d,equals] & E_\varphi \arrow[r,"\pi_A"] \ar[d,"\pi_E"] & A \arrow[r] \ar[d,"\varphi"] & 0\\
    \mathfrak e: 0 \arrow[r] & I \arrow[r,"j"] & E \arrow[r,"q"] & D \arrow[r] & 0
\end{tikzcd}
\end{equation}
where $E_\varphi$ is the pull-back of $D$ by $q$ and $\varphi$, i.e.
\begin{equation}
    E_\varphi := \{ (x,a) \in E\oplus A \mid q(x) = \varphi(a) \},
\end{equation}
$\pi_E$ and $\pi_A$ are the respective projection maps and $\iota$ is the corestriction of $j\oplus 0$ to $E_\varphi$. A straightforward computation shows that $\mathfrak e\circ\varphi$ indeed forms a short exact sequence which makes diagram above commute.
\end{defn}

\begin{defn}
    Let
    \begin{equation}
    \begin{tikzcd}
        \mathfrak e: 0 \arrow[r] & I \arrow[r,"\iota"] & E \arrow[r,"\pi"] & D \arrow[r] & 0
    \end{tikzcd}
    \end{equation}
    be an extension of \ca s. A \emph{subextension} of $\mathfrak e$ is an extension of the form
    \begin{equation}
    \begin{tikzcd}
        \mathfrak e_0: 0 \arrow[r] & I_0 \arrow[r,"\iota|_{I_0}"] & E_0 \arrow[r,"\pi|_{E_0}"] & D_0 \arrow[r] & 0,
    \end{tikzcd}
    \end{equation}
    where $I_0$, $E_0$ and $D_0$ are $C^*$-subalgebras of $I$, $E$ and $D$ respectively. It fits in the commutative diagram
    \begin{equation}
    \begin{tikzcd}
        \mathfrak e_0: 0 \arrow[r] & I_0 \arrow[r] \arrow[d,hook] & E_0 \arrow[r] \arrow[d,hook] & D_0 \arrow[r] \arrow[d,hook] & 0\\
        \mathfrak e: 0 \arrow[r] & I \arrow[r,"\iota"] & E \arrow[r,"\pi"] & D \arrow[r] & 0
    \end{tikzcd}
    \end{equation}
    where the downward arrows are inclusion maps. We will denote this with $\mathfrak e_0 \subseteq \mathfrak e$.\\
    If $\mathfrak e$ is unital, we say $\mathfrak e_0$ is a \emph{unital} subextension if $E_0$ contains the unit of $E$.
\end{defn}

\begin{defn}
    Suppose
    \begin{equation}
    \begin{tikzcd}
        \mathfrak e: 0 \arrow[r] & I \arrow[r,"\iota"] & E \arrow[r,"\pi"] & D \arrow[r] & 0
    \end{tikzcd}
    \end{equation}
    is an extension of \ca s, and for each $n\in\mathbb N$
    \begin{equation}
    \begin{tikzcd}
        \mathfrak e_n: 0 \arrow[r] & I_n \arrow[r,"\iota_n"] & E_n \arrow[r,"\pi_n"] & D_n \arrow[r] & 0
    \end{tikzcd}
    \end{equation}
    is a subextension of $\mathfrak e$ such that $\mathfrak e_{n-1} \subseteq \mathfrak e_n$. Then
    \begin{equation}
    \begin{tikzcd}
        \varinjlim \mathfrak e_n: 0 \arrow[r] & \overline{\bigcup_n I_n} \arrow[r,"\overline\iota"] & \overline{\bigcup_n E_n} \arrow[r,"\overline\pi"] & \overline{\bigcup_n D_n} \arrow[r] & 0,
    \end{tikzcd}
    \end{equation}
    is again a subextension of $\mathfrak e$, where $\overline{\iota}$ and $\overline{\pi}$ are the appropriate restrictions of $\iota$ resp.\ $\pi$. In fact, $\varinjlim \mathfrak e_n$ is the smallest subextension of $\mathfrak e$ containing each of the $\mathfrak e_n$. We will call $\varinjlim \mathfrak e_n$ the \emph{inductive limit extension} of $(\mathfrak e_n)_n$.
\end{defn}

For an overview of further extension theory preliminaries, see for example \cite[\S 2]{GabeRuiz_TheUnitalExtGroupsAndClassificationOfCalgebras}. We will also adopt their notation, and establish the following fact using the UCT for $\Ext_{us}$ (\cite[Theorem 14.4]{GabeRuiz_TheUnitalExtGroupsAndClassificationOfCalgebras}).

\begin{lemma} \label{lem:ExtOinftytrivial}
    For any separable \ca\ $J$, the group $\Ext_{us}(\Oinf,J)$ is trivial.
\end{lemma}
\begin{proof}
First note that nuclearity of $\Oinf$ gives $\Ext_{us}^{-1}(\Oinf,J) = \Ext_{us}(\Oinf,J)$ by the Choi-Effros lifting theorem \cite{ChoiEffros_TheCompletelyPositiveLiftingProblemForCalgebras}. By \cite[Theorem 4.14]{GabeRuiz_TheUnitalExtGroupsAndClassificationOfCalgebras}, we have a short exact sequence of abelian groups
\begin{equation} \label{eq:extusses}
\begin{tikzcd}
    0 \arrow[r] & E \arrow[r] & \Ext_{us}(\Oinf,J) \arrow[r] & H \arrow[r] & 0,
\end{tikzcd}
\end{equation}
where
\begin{equation}
\begin{aligned}
    E &:= \Ext((K_0(\Oinf),[1]_0), K_0(J)) \oplus \Ext(K_1(\Oinf), K_1(J))\\
    H &:= \{\alpha\in\Hom(K_0(\Oinf), K_1(J)) \mid \alpha([1]_0) = 0\} \oplus \Hom(K_1(\Oinf), K_0(J)),
\end{aligned}
\end{equation}
and $\Ext((K_0(\Oinf),[1]_0), K_0(J))$ are extensions of abelian groups of the form
\begin{equation}
\begin{tikzcd}
    \mathfrak e: 0 \arrow[r] & K_0(J) \arrow[r] & (G,g) \arrow[r,"q"] & (K_0(\Oinf),[1]_0) \arrow[r] & 0
\end{tikzcd}
\end{equation}
with $g\in G$ and $q(g)=[1]_0$. Now, we know that
\begin{equation}
    (K_0(\Oinf),[1]_0, K_1(\Oinf)) \cong (\mathbb Z, 1, 0),
\end{equation}
so we immediately find $H = 0$ and $\Ext(K_1(\Oinf), K_1(J)) = 0$. Moreover, observe that any extension in $\Ext((\mathbb Z,1), K_0(J))$ splits: the homomorphism defined by mapping $1\in\mathbb Z$ to the distinguished element in the extension group provides a splitting. Hence also $\Ext((K_0(\Oinf),[1]_0), K_0(J)) = 0$, showing that both $H$ and $E$ are trivial. By exactness of \eqref{eq:extusses}, we can then conclude that $\Ext_{us}(\Oinf,J) = 0$.
\end{proof}

\subsection{Fullness of maps}
By an ideal in a \ca, we will always mean a two-sided, norm-closed ideal.

\begin{defn}
    An element $a$ of a \ca\ $A$ is said to be \emph{full} if $a$ generates $A$ as an ideal.\\
    A \sh\ $\varphi: A\to B$ between \ca s is said to be \emph{full} if $\varphi(a)$ is full for each non-zero $a\in A$. If $B$ is unital, we say $\varphi$ is \emph{unitizably full} if the forced unitization $\varphi^\dagger: A^\dagger\to B$ is full.
\end{defn}

\begin{rem}
Note that a \sh\ $\varphi: A\rightarrow B$ between \ca s is full if and only if its image has trivial intersection with all non-trivial ideals of $B$.
\end{rem}

We will need the following notion to prove the subsequent lemma.

\begin{defn}
    Let $I$ be a \ca. We say $I$ is \emph{separably stable} if it satisfies the so-called \emph{Hjelmborg-R\o rdam criterion}: for each $x\in I_+$ and $\varepsilon> 0$, there exists a $y\in I$ such that $x \approx_\varepsilon yy^*$ and $(yy^*)(y^*y) \approx_\varepsilon 0$.
\end{defn}

A priori, the above definition seems to have no connection with stability. However, Hjelmborg and R\o rdam proved in \cite{HjelmborgRordam_OnStabilityofCalgebras} that the two are equivalent if $I$ is $\sigma$-unital. The extra specifier ``separably'' relates to the notion of separably inheritable properties introduced by Blackadar; we will discuss this topic in full detail in \S\ref{sec:separabilization}.

We now record the following lemma, showing that fullness and a projection being full and properly infinite lift from the quotient if the ideal is stable (and separable). The former implies that arbitrary $*$-homomorphic lifts of (unitizably) full maps will be unitizably full. The latter will be used in the proof of existence, in order to obtain Murray-von Neumann equivalence of two projections based on equality of $K_0$-classes. 

\begin{lemma} \label{lem:propertiesliftingfromquotient}
    Let
    \begin{equation}
    \begin{tikzcd}
        \mathfrak e: 0 \arrow[r] & I \arrow[r,"j"] & E \arrow[r,"q"] & D \arrow[r] & 0
    \end{tikzcd}
    \end{equation}
    be a unital extension of \ca s with $I$ stable.
    \begin{enumerate}
        \item An arbitrary element $x\in E$ is full if and only if $q(x)$ is full in $D$. Consequently, any \sh\ $\varphi$ into $E$ is (unitizably) full if and only if $q \circ \varphi$ is (unitizably) full. \label{lem:propertiesliftingfromquotient.full}
        \item If $I$ is also separable, then an arbitrary projection $p\in E$ is full and properly infinite\footnote{A projection $p\in E$ is said to be \emph{properly infinite} if it has two mutually orthogonal subprojections which are both Murray-von Neumann equivalent to $p$. Equivalently, $pEp$ contains a unital copy of $\Oinf$.} if and only if $q(p)$ is full and properly infinite. \label{lem:propertiesliftingfromquotient.propinf}
    \end{enumerate}
\end{lemma}
\begin{proof}
a) If $x\in E$ is full, then $1_{E}$ is in $\overline{E x E}$, so also
\begin{equation}
    1_{D} = q(1_{E}) \in q\Big(\overline{E x E}\Big) \subseteq \overline{D q(x) D},
\end{equation}
i.e.\ $q(x)$ is full.

Conversely, suppose $q(x)$ is full. Then $1_{D}\in \overline{D q(x) D}$, so lifting this fact to $E$, there exist $y_1,\dots , y_n \in E$ such that
\begin{equation}
    c := 1_{E} - \sum_{i=1}^n y_i^\ast x^\ast x y_i \in I.
\end{equation}
Let $\sigma \colon E \to \mathcal M(I)$ be the canonical $\ast$-homomorphism. As $I$ is stable, there exists an isometry $v\in \mathcal M(I)$ such that $\|v^*cv\| < 1$. This implies that
\begin{equation}
    \| 1_{\mathcal M(I)} - \sum_{i=1}^n v^*\sigma(y_i^\ast x^\ast x y_i)v \| = \| v^\ast c v \| < 1,
\end{equation}
so the element $z:=\sum_{i=1}^n v^*\sigma(y_i^\ast x^\ast x y_i)v\in \mathcal M(I)$ is positive and invertible. Factorize $c= c_1^\ast c_2$ with $c_1,c_2\in I$. Then $vz^{-1/2}c_j \in I$ and therefore
\begin{equation}
c = c_1^\ast c_2 = \sum_{i=1}^n (vz^{-1/2} c_1)^\ast y_i^\ast x^\ast x y_i (vz^{-1/2} c_2) \in \overline{ExE}.    
\end{equation}
Thus
\begin{equation}
    1_E = c + \sum_{i=1}^n y_i^\ast x^\ast x y_i \in \overline{ExE}
\end{equation}
which implies that $x$ is full.

The consequence in the statement follows almost immediately from the above; for the unitizably full statement, note that $(q\circ\varphi)^\dagger = q\circ\varphi^\dagger$, as $q$ is unital.

b) If $p$ is full and properly infinite, then immediately the same follows for $q(p)$. Conversely, suppose $q(p)$ is full and properly infinite. 
{We will first argue that $pIp$ is stable. Let $x\in pIp$ be positive and $\epsilon>0$. Let $f\in C_0((0,\| x\|)$ be a positive contraction such that $f(t) = 1$ for $t\geq \epsilon$. As $q(p-f(x)) = q(p)$ is full and properly infinite, there is a $y\in M(I)$ such that $1_{M(I)}-y^\ast (p-f(x))y \in I$. As $I$ is stable, arguing as in part a), we may assume that $y^\ast(p-f(x)) y = 1_{M(I)}$. Let $z = (p-f(x))y(x-\epsilon)_+^{1/2} \in pIp$. Then $z^\ast z = (x-\epsilon)_+$ and since $(p-f(x)) (x-\epsilon)_+ = 0$ one gets $(zz^\ast) (z^\ast z) = 0$. Hence $pIp$ is stable by the Hjelmborg--Rørdam criterion (\cite{HjelmborgRordam_OnStabilityofCalgebras}).}

Now, consider the unital extension
\begin{equation}
\begin{tikzcd}
    0 \arrow[r] & pIp \arrow[r] & pEp \arrow[r,"q"] & q(p)Dq(p) \arrow[r] & 0.
\end{tikzcd}
\end{equation}
As $q(p)$ is properly infinite, we can embed $\Oinf$ unitally into $q(p)Dq(p)$, inducing a pull-back
\begin{equation}
\begin{tikzcd}
    0 \arrow[r] & pIp \arrow[r] & P \arrow[r] & \Oinf \arrow[r] & 0
\end{tikzcd}
\end{equation}
with $P \subseteq pEp$ unitally. As $pIp$ is separable and stable, and $\Ext_{us}(\Oinf, pIp) =0$ by Lemma \ref{lem:ExtOinftytrivial}, it then follows that this extension splits unitally. In other words, there exists a unital \sh\ $\Oinf \to P \subseteq pEp$, so $pEp$ contains a unital copy of $\Oinf$.
\end{proof}

\section{The reduced state-kernel extension}
\subsection{Construction}
The main object of study in the new approach to classification in the stably finite setting is the so-called \emph{trace-kernel extension}
\begin{equation}
\begin{tikzcd}
    0 \ar[r] & J_B \ar[r,"j_B"] & B_\infty \ar[r,"q_B"] & B^\infty \ar[r] & 0,
\end{tikzcd}
\end{equation}
where
\begin{equation}
    J_B := \{ (x_n)_n\in B_\infty \mid \lim_{n\to\infty} \sup_{\tau\in T(B)} \tau(x_n^*x_n) = 0\}
\end{equation}
and $B^\infty := B_\infty/J_B$. In the purely infinite setting, there are no traces, so we will need to alter the definition. A natural candidate to replace traces would be states.

\begin{defn}
    Let $\rho$ be a state on a \ca\ $B$. We define the \emph{$\rho$-norm} as
    \begin{equation}
        \|b\|_\rho := \sqrt{\rho(b^*b)} \quad (b\in B).
    \end{equation}
    Note that this is only a seminorm in general, but it is a norm if (and only if) $\rho$ is faithful. We also define the \emph{$(\rho,\#)$-norm} as
    \begin{equation}
        \|b\|_{\rho,\#} := \frac{1}{2}\left(\|b\|_\rho + \|b^*\|_\rho\right) \quad (b\in B).
    \end{equation}
\end{defn}

\begin{rem}
The $\rho$- and $(\rho,\#)$-norms are not submultiplicative in general. However, we do have the following identity
\begin{equation} \label{eq:ske_phinormofproduct}
    \|xy\|_\rho = \sqrt{\rho(y^*x^*xy)} \leq \|x\| \sqrt{\rho(y^*y)} = \|x\| \|y\|_\rho
\end{equation}
for all $x,y\in B$ by positivity of $\rho$, and subsequently also
\begin{equation} \label{eq:ske_phihashnormofproduct}
    2\|xy\|_{\rho,\#} \leq \|x\| \|y\|_\rho + \|y^*\| \|x^*\|_\rho \leq \|x\| \|y\|_{\rho,\#} + \|x\|_{\rho,\#} \|y\| \quad (x,y\in B).
\end{equation}
Moreover, we have
\begin{equation} \label{eq:ske_normdominatesphinorm}
    2\|b\|_{\rho,\#} = \sqrt{\rho(b^*b)} + \sqrt{\rho(bb^*)} \leq \sqrt{\|b^*b\|} + \sqrt{\|bb^*\|} = 2\|b\|
\end{equation}
for all $b\in B$, showing that the $(\rho,\#)$-norm (and by a similar argument, also the $\rho$-norm) is dominated by the norm on $B$.
\end{rem}

One of the main conceptual reasons for including all traces in the stably finite setting, is because this tracial data can be used to access comparison, due to $\Z$-stability of $B$ and the Toms-Winter conjecture. In the purely infinite setting, however, comparison comes for free. This allows us to limit ourselves to a single state to construct an analogue of the trace-kernel extension, as long as it satisfies certain properties. The main difference when working with states instead of traces, is that the analogue of the trace-kernel $J_B$ will now no longer be an ideal in $B_\infty$. To obtain a short exact sequence, we will therefore need to restrict $B_\infty$ to a suitable $C^*$-subalgebra in which this analogue is an ideal.

\begin{defn}
    Let $\rho$ be a state on a \ca\ $B$. We define the \emph{state-kernel} of $B$ with respect to\ $\rho$ as the set
    \begin{equation}
        J_{B,\rho} := \{ (x_n)_n\in B_\infty \mid \lim_{n\to\infty} \|x_n\|_{\rho,\#} = 0\}.
    \end{equation}
    We also define
    \begin{equation}
        S_{B,\rho} := \{ (x_n)_n\in B_\infty \mid (x_n)_n \text{ is $(\rho,\#)$-Cauchy}\}.
    \end{equation}
\end{defn}

Observe that these are well-defined subsets of $B_\infty$, as sequences which tend to zero in norm, also tend to zero in $(\rho,\#)$-norm by \eqref{eq:ske_normdominatesphinorm}, and in particular are $(\rho,\#)$-Cauchy. As mentioned earlier, the state-kernel $J_{B,\rho}$ is not an ideal in general, but it is a hereditary $C^*$-subalgebra of $B_\infty$. This we will now proceed to show.

\begin{prop} \label{prop:ske_statekernelhereditary}
    Let $\rho$ be a state on a \ca\ $B$. Then $J_{B,\rho}$ is a well-defined, hereditary $C^*$-subalgebra of $B_\infty$.
\end{prop}
\begin{proof}
As $\|\cdot\|_{\rho,\#}$ satisfies the triangle equality and \eqref{eq:ske_phihashnormofproduct}, it follows immediately that it is closed under sums and products. It is clear from the definition of $\|\cdot\|_{\rho,\#}$ that $J_{B,\rho}$ is also closed under adjoints, so we only need to show that it is closed. To this end, let $(x^{(i)})_{i\in I}$ be a net in $J_{B,\rho}$ converging to some $x\in B_\infty$. Fixing representative sequences $(x^{(i)}_n)_n$ ($i\in I$) resp.\ $(x_n)_n$ in $\ell^\infty(B)$, we then have
\begin{equation}
    \lim_{n\to\infty} \|x^{(i)}_n\|_{\rho,\#} = 0
\end{equation}
for all $i\in I$ and
\begin{equation}
    0 = \lim_{i\to\infty} \|x-x^{(i)}\| = \lim_{i\to\infty} \limsup_{n\to\infty} \|x_n-x^{(i)}_n\|.
\end{equation}
It follows that
\begin{equation}
\begin{aligned}
    \limsup_{n\to\infty} \|x_n\|_{\rho,\#}
        &\leq \limsup_{n\to\infty} \|x_n^{(i)}\|_{\rho,\#} + \limsup_{n\to\infty} \|x_n - x_n^{(i)}\|_{\rho,\#}\\
        &\leq \limsup_{n\to\infty} \|x_n^{(i)} - x_n\|
\end{aligned}
\end{equation}
for all $i\in I$, where for the second inequality we used \eqref{eq:ske_normdominatesphinorm}. As the right-hand side tends to zero for $i\to \infty$, we find that $\limsup_{n\to\infty} \|x_n\|_{\rho,\#} = 0$. Hence, $x$ is indeed in $J_{B,\rho}$, proving that $J_{B,\rho}$ is closed.

To show that it is hereditary in $B_\infty$, let $x\in J_{B,\rho}$ and $y\in B_\infty$ be two positive elements satisfying $y\leq x$, which implies that $y\in \overline{xB_\infty x}$. As we now know that $J_{B,\rho}$ is closed, we can assume without loss of generality that $y=xzx$ for some $z\in (B_\infty)_+$. Fixing positive representative sequences $(x_n)_n$ and $(z_n)_n$ in $\ell^\infty(B)$ of $x$ resp.\ $z$, it follows that
\begin{equation}
    \limsup_{n\to \infty} \|x_nz_nx_n\|_{\rho,\#} = \limsup_{n\to \infty} \|x_nz_nx_n\|_\rho \leq \limsup_{n\to \infty} \|x_nz_n\| \|x_n\|_\rho =0,
\end{equation}
where we used \eqref{eq:ske_phinormofproduct} for the inequality. Hence, we can conclude that $y= xzx \in J_{B,\rho}$, showing that $J_{B,\rho}$ is closed.
\end{proof}

A priori, there is no reason to believe that $J_{B,\rho}$ is indeed an ideal in $S_{B,\rho}$, nor that $S_{B,\rho}$ is even a $C^*$-algebra, due to the lack of submultiplicativity of the state norm. To ensure that these facts are indeed true, we need the state to be strongly faithful, a property we will now define.

\begin{defn} \label{def:stronglyfaithful}
    Let $B$ be a \ca, $\rho$ a state on $B$ and $\pi_\rho: B\to B(\Hb)$ the associated GNS representation with cyclic vector $\xi_\rho \in \Hb$. Let
    \begin{equation} \label{eq:uvn_phiextenstiontoGNS}
    \overline{\rho}(x) = \langle x\xi_\rho, \xi_\rho\rangle \quad (x\in B(\Hb))
    \end{equation}
    be the canonical extension of $\rho$ to $B(\Hb)$. We say that $\rho$ is \emph{strongly faithful} if it is faithful on $B$ and $\overline{\rho}$ is faithful on $\pi_\rho(B)''$.
\end{defn}
\begin{rem} \label{rem:faithfulvsstrongly}
Faithfulness of the state and faithfulness of its extension to $\pi_\rho(B)''$ are not related in general. On the one hand, $\overline{\rho}$ being faithful does not force $\rho$ to be faithful, as the latter is on the level of $B$, while the former is on the level of $B(\Hb)$; $\pi_\rho$ can have a non-trivial kernel. On the other hand, faithful states are not necessarily strongly faithful. Consider\footnote{This counterexample was given by Ozawa in\\ \url{https://mathoverflow.net/questions/93295/separating-vectors-for-c-algebras}.} for example the \ca\ $B := C([0,1], M_2(\mathbb C))$ with the state $\rho$ given by
\begin{equation} \label{eq:faithfulvsstrongly}
    \rho(f) := \int_C f(x)_{11} \dee x + \int_{[0,1]\setminus C} \tr f(x)\dee x \quad (f\in B),
\end{equation}
where $C$ is a (fixed) closed nowhere dense subset of $[0,1]$ with non-zero measure. Then $\rho$ is faithful, but its extension is not faithful on the strong completion of $B$, which is $L^\infty([0,1],M_2(\mathbb C))$. 

Moreover, pure states can never be strongly faithful (unless $B=\mathbb C$): for any strongly faithful state $\rho$, the GNS vector $\xi_\rho \in \Hb_\rho$ will be separating for $\pi_\rho(B)''$, which implies that it is cyclic for $\pi_\rho(B)'$. If $\rho$ were pure, then $\pi_\rho(B)' = \mathbb C 1$, so $\xi_\rho$ being cyclic would give $\Hb_\rho = \mathbb C \xi_\rho$.
\end{rem}

The main technical resource strong faithfulness brings is summarized in the following proposition. The core result is elementary and well-known; the main purpose of the proposition is simply to rephrase in a way which is convenient for our purposes.

\begin{prop}
\label{prop:ske_strongequivphinorm}
Let $\rho$ be a strongly faithful state on a \ca\ $B$, $\pi_\rho: B\to B(\Hb)$ the associated GNS representation, and $\rho''$ the canonical extension of $\rho$ to $\pi_\rho(B)''$ (as in \eqref{eq:uvn_phiextenstiontoGNS}). Then $\|\cdot\|_{\rho''}$ metrizes the strong topology on $\pi_\rho(B)''$ on bounded sets: for any $(x_n)_n\in \ell^\infty(\pi_\rho(B)'')$, we have
\begin{enumerate}
    \item $(x_n)_n$ is strongly Cauchy if and only if it is $\rho''$-Cauchy,
    \item $(x_n)_n$ converges strongly to some $x\in \pi_\rho(B)''$ if and only if it converges in $\rho''$-norm to $x$.
\end{enumerate}
The same results hold when replacing ``strong'' with ``strong-$*$'' and ``$\rho''$'' with ``$(\rho'',\#)$''. In particular, we have
\begin{equation}
\begin{aligned}
    J_{B,\rho} &= \{ (x_n)_n\in B_\infty \mid \slim \pi_\rho(x_n) = 0\},\\
    S_{B,\rho} &= \{ (x_n)_n\in B_\infty \mid \slim \pi_\rho(x_n) \text{ exists} \}.
\end{aligned}
\end{equation}
\end{prop}
\begin{proof}
Strong faithfulness of $\rho$ implies that $\rho''$ is a separating normal state on $\pi_\rho(B)''$, so the first part of the statement, including (a) and (b), follow immediately from \cite[\S III.2.2.17]{Blackadar_OperatorAlgebrasTheoryofCalgebrasandvonNeumannAlgebras}. The equivalent definitions of $J_{B,\rho}$ and $S_{B,\rho}$ are an immediate consequence of this fact, after observing that
\begin{equation}
    \| b \|_\rho = \sqrt{\langle \pi_\rho(b^*b)\xi_\rho, \xi_\rho \rangle} = \|\pi_\rho(b)\|_{\rho''} \quad (b\in B),
\end{equation}
and that $\pi_\rho(B)''$ is $*$-strongly complete.
\end{proof}

The above proposition suggests a clear candidate for the quotient in the reduced state-kernel extension: the von Neumann algebra $\pi_\rho(B)''$. After ensuring that $S_{B,\rho}$ is indeed a $C^*$-subalgebra of $B_\infty$, we will check that the obvious map $S_{B,\rho}\to \pi_\rho(B)''$ is a well-defined, surjective \sh, allowing us to subsequently define the reduced state-kernel extension. We will also show that the state on $S_{B,\rho}$ induced by the state $\rho''$ on $\pi_\rho(B)''$ is in fact the limit state on $S_{B,\rho}$ induced by $\rho$.

\begin{prop}
    Let $B$ be a \ca\ with a strongly faithful state $\rho$. Then $S_{B,\rho}$ is a $C^*$-subalgebra of $B_\infty$ containing $J_{B,\rho}$.
\end{prop}
\begin{proof}
As the $(\rho,\#)$-norm satisfies the triangle inequality and is invariant under taking the adjoint, it follows immediately that $S_{B,\rho}$ is closed under sums and adjoints. As multiplication is $*$-strongly continuous on bounded sets, it follows immediately from Proposition \ref{prop:ske_strongequivphinorm} that $S_{B,\rho}$ is also closed under multiplication.

To show that $S_{B,\rho}$ is also norm-closed, let $(x^{(i)})_{i\in I}$ be a net in $S_{B,\rho}$ converging to some $x\in B_\infty$, and fix representative sequences $(x^{(i)}_n)_n$ ($i\in I$) resp.\ $(x_n)_n$ in $\ell^\infty(B)$. To show that $(x_n)_n$ is $(\rho,\#)$-Cauchy, pick $\varepsilon >0$ and let $i\in I$ such that
\begin{equation}
    \frac{\varepsilon}{4} > \|x-x^{(i)}\| = \limsup_{n\to \infty} \|x_n-x^{(i)}_n\|.
\end{equation}
Hence, there exists an $N_1\in\mathbb N$ such that $\|x_n - x^{(i)}_n\|<\varepsilon/3$ for all $n\geq N_1$. As $(x^{(i)}_n)_n$ is $(\rho,\#)$-Cauchy, we can find $N_2\in\mathbb N$ such that $\|x^{(i)}_n - x^{(i)}_m\|_{\rho,\#} < \varepsilon/3$ for all $n,m\geq N_2$. It then follows for all $n,m\geq \max\{N_1,N_2\}$ that
\begin{equation}
\begin{aligned}
    \|x_n - x_m\|_{\rho,\#}
        &\leq \|x_n - x^{(i)}_n\| + \|x^{(i)}_n - x^{(i)}_m\|_{\rho,\#} + \|x^{(i)}_m - x_m\| \\
        &< \frac{\varepsilon}{3} + \frac{\varepsilon}{3} + \frac{\varepsilon}{3} = \varepsilon,
\end{aligned}
\end{equation}
where for the first inequality we used \eqref{eq:ske_normdominatesphinorm}. Hence, $(x_n)_n$ is indeed $(\rho,\#)$-Cauchy, so $S_{B,\rho}$ is closed. We conclude that $S_{B,\rho}$ is in fact a $C^*$-subalgebra of $B_\infty$. It contains $J_{B,\rho}$ because sequences which converge to zero in $(\rho,\#)$-norm, are definitely also $(\rho,\#)$-Cauchy. This finishes the proof.
\end{proof}

\begin{rem}
   If the state $\rho$ is not faithful,  $S_{B, \rho}$ is generally not closed under multiplication. For instance, if $B = M_2(\mathbb C)$ and $\rho$ is compression to the $(1,1)$-corner, then $(e_{1,2})_n$ and $((1+(-1)^n)e_{2,2})_n$ are in $S_{B, \rho}$ but their product $((1+(-1)^n)e_{1,2})_n$ is not in $S_{B, \rho}$. 

   This fact even holds for faithul states. Take the example $\rho$ on $B = \mathcal{C}([0,1], M_2)$ given in \eqref{eq:faithfulvsstrongly}. Let $(h_n)_n$ be a sequence of positive functions in $\mathcal{C}([0,1])$ converging pointwise to $\chi_C$, and consider $f_n := e_{2,2}\otimes (1+ (-1)^n) h_n \in B$ and $g_n := e_{1,2} \otimes 1_{[0,1]}$. Then
   \begin{equation}
       \rho(f_n) = (1+ (-1)^n) \rho(e_{2,2}\otimes h_n) \leq 2\rho(e_{2,2}\otimes h_n) \to 0
   \end{equation}
   so $(f_n)_n \in S_{B,\rho}$, and clearly also $g_n \in S_{B,\rho}$. However,
   \begin{equation}
       g_n f_n = e_{1,2}\otimes (1+ (-1)^n) h_n
   \end{equation}
   is not in $S_{B,\rho}$.
\end{rem}

\begin{cor} \label{cor:ske_maptoquotient}
    Let $B$ be a \ca\ with a strongly faithful state $\rho$. Then the map
    \begin{equation}
        q_B: S_{B,\rho} \to \pi_\rho(B)'': (x_n)_n \mapsto \slim \pi_\rho(x_n)
    \end{equation}
    is a well-defined, surjective \sh, of which the kernel is exactly $J_{B,\rho}$. In particular, $J_{B,\rho}$ is an ideal in $S_{B,\rho}$.\\
    The natural extension $\rho''$ (as defined in \eqref{eq:uvn_phiextenstiontoGNS}) of $\rho$ to $\pi_\rho(B)''$ induces the state $\rho''\circ q_B$ on $S_{B,\rho}$, which is given explicitly by
    \begin{equation}
        \rho''\circ q_B((x_n)_n) = \lim_{n\to \infty} \rho(x_n) \quad ((x_n)_n\in S_{B,\rho}).
    \end{equation}
    In particular, we have
    \begin{equation}
        \|q_B((x_n)_n)\|_{\rho'',\#} = \lim_{n\to \infty} \|x_n\|_{\rho,\#}
    \end{equation}
    for all $(x_n)_n\in S_{B,\rho}$.
\end{cor}
\begin{proof}
First note that $q_B$ is indeed a well-defined map. Indeed, the $*$-strong limit of sequences in $S_{B,\rho}$ exists by Proposition \ref{prop:ske_strongequivphinorm}, and it is independent of the choice of representative because sequences in $B$ which converge in norm to zero, also converge $*$-strongly to zero in $\pi_\rho(B)''$. It is a \sh\ because addition, the adjoint operation (by definition of the $*$-strong topology) and multiplication are $*$-strongly continuous on bounded sequences. From Proposition \ref{prop:ske_strongequivphinorm}, it is immediate that the kernel of $q_B$ is equal to $J_{B,\rho}$. To see that it is surjective, let $x\in \pi_\rho(B)''$ and use Kaplansky's density theorem to find a sequence $(x_n)_n$ in $B$ with norm bounded by $\|x\|$ and such that $\pi_\rho(x_n)$ converges $*$-strongly to $x$. Then $(x_n)_n$ is a $(\rho,\#)$-Cauchy sequence in $\ell^\infty(B)$, so it induces an element of $S_{B,\rho}$, which satisfies
\begin{equation}
    q_B(\pi_\infty(x_n)_n) = \slim \pi_\rho(x_n) = x,
\end{equation}
proving that $q_B$ is surjective. For the second part of the statement, we compute for arbitrary $(x_n)_n\in S_{B,\rho}$
\begin{equation}
\begin{aligned}
    \rho''\circ q_B((x_n)_n)
        &= \langle \left(\slim \pi_\rho(x_n)\right)\xi_\rho, \xi_\rho\rangle \\
        &= \langle \lim_{n\to\infty} \pi_\rho(x_n)\xi_\rho, \xi_\rho\rangle \\
        &= \lim_{n\to\infty} \langle \pi_\rho(x_n)\xi_\rho, \xi_\rho\rangle \\
        &= \lim_{n\to\infty} \rho(x_n).
\end{aligned}
\end{equation}
From this fact, it also follows that
\begin{equation}
\begin{aligned}
    2\|q_B((x_n)_n)\|_{\rho'',\#}
        &= \sqrt{\rho''(q_B((x_n^*x_n)_n))} + \sqrt{\rho''(q_B((x_nx_n^*)_n))} \\
        &= \lim_{n\to\infty} \sqrt{\rho(x_n^*x_n)} + \lim_{n\to\infty} \sqrt{\rho(x_nx_n^*)} \\
        &= 2 \lim_{n\to\infty} \|x_n\|_{\rho,\#},
\end{aligned}
\end{equation}
concluding the proof.
\end{proof}

\begin{defn}
    Let $B$ be a \ca\ with a strongly faithful state $\rho$. Using the previous corollary, we define the \emph{reduced state-kernel extension} $\mathfrak e_B$ as
    \begin{equation}
    \begin{tikzcd}
        0 \ar[r] & J_{B,\rho} \ar[r,"j_B"] & S_{B,\rho} \ar[r,"q_B"] & \pi_\rho(B)'' \ar[r] & 0,
    \end{tikzcd}
    \end{equation}
    where $j_B$ is the inclusion map.
\end{defn}

\begin{rem}
    One can observe that $S_{B,\rho}$ is not the largest possible $C^*$-subalgebra of $B_\infty$ in which $J_{B,\rho}$ is an ideal. That would be the normalizer
    \begin{equation}
        N(J_{B,\rho}) := \{ x\in B_\infty \mid xJ_{B,\rho}, J_{B,\rho}x \subseteq J_{B,\rho} \}
    \end{equation}
    of $J_{B,\rho}$. One could absolutely also work with the normalizer, but opting for $S_{B,\rho}$ comes with a number of advantages. The quotient $\pi_\rho(B)''$ is always a von Neumann algebra with a faithful normal state $\rho''$. In fact, one can choose $\rho$ appropriately such that $\pi_\rho(B)''$ is simply $B(\Hb)$, as we will show in Corollary \ref{cor:stronglyfaithfulstateexists}. If we were to work with the normalizer, then the quotient $N(J_{B,\rho})/J_{B,\rho}$ would no longer be a von Neumann algebra in the sequence algebra setting. It would be if we were working in $B_\omega$; the corresponding quotient is known as the Ocneanu ultrapower of $B$, and is a well-studied object in von Neumann theory (see for example \cite{OcneanuActionsOfDiscreteAmenableGroupsOnVonNeumannAlgebras} for the original construction by Ocneanu, and \cite{AndoHaagerup_UltraproductsOfVonNeumannAlgebras} for a more recent treatment). However, there are a number of reasons we want to work in the sequence algebra as much as possible. The main reason is to appeal to so-called intertwining by reparametrization (\cite[Theorem 4.3]{Gabe_ANewProofOfKirchbergsO2StableClassification} and our adaptation, Theorem \ref{thm:reparametrization}) in order to check that a given map into $B_\infty$ factors through $B$. This is a crucial step in going from a classification of maps into $S_{B,\rho}$ to a classification of maps into $B$ itself.
\end{rem}

\subsection{The reduced state-kernel extension in $B_\omega$}
As mentioned earlier, we would ideally like to work entirely in the sequence algebra setting. However, there is an important downside to doing so: the sequence algebra $B_\infty$ is not simple and purely infinite when $B$ is, while this does hold for $B_\omega$ (see \cite[Proposition 6.2.6]{Rordam_ClassificationofNuclearCalgebras}). This fact is crucial in obtaining our desired KK-uniqueness result (Theorem \ref{thm:KK-Uniqueness}), meaning we have to descend to the ultrapower whenever we need said result. Therefore, we will summarize the construction of the reduced state-kernel extension in $B_\omega$ in the following statement. All proofs from the previous section work with minor adaptations.

\begin{prop} \label{prop:skeinomega}
    Let $\rho$ be a state on a \ca\ $B$. The state-kernel in $B_\omega$
    \begin{equation}
        J_{B,\rho}^{(\omega)} := \{ (x_n)_n\in B_\infty \mid \lim_{n\to\omega} \|x_n\|_{\rho,\#} = 0\}
    \end{equation}
    is a well-defined, hereditary $C^*$-subalgebra of $B_\omega$. If $\rho$ is strongly faithful, then
    \begin{equation}
        S_{B,\rho}^{(\omega)} := \{ (x_n)_n\in B_\omega \mid (x_n)_n \text{ is $(\rho,\#)$-Cauchy along } \omega\}
    \end{equation}
    is a well-defined $C^*$-subalgebra of $B_\omega$ containing $J_{B,\rho}^{(\omega)}$, and
    \begin{equation}
        q_B^{(\omega)}: S_{B,\rho}^{(\omega)} \to \pi_\rho(B)'': (x_n)_n \mapsto \mathop{\operatorname{s^*-lim}}_{n\to\omega} \pi_\rho(x_n)
    \end{equation}
    is a well-defined, surjective \sh, of which the kernel is exactly $J_{B,\rho}^{(\omega)}$. In particular, $J_{B,\rho}^{(\omega)}$ is an ideal in $S_{B,\rho}^{(\omega)}$, and we can form the reduced state-kernel extension in $B_\omega$
    \begin{equation}
    \begin{tikzcd}
        \mathfrak e_B^{(\omega)}: 0 \ar[r] & J_{B,\rho}^{(\omega)} \ar[r,"j_B^{(\omega)}"] & S_{B,\rho}^{(\omega)} \ar[r,"q_B^{(\omega)}"] & \pi_\rho(B)'' \ar[r] & 0.
    \end{tikzcd}
    \end{equation}
    Moreover, we have
    \begin{equation}
        \rho''\circ q_B^{(\omega)}((x_n)_n) = \lim_{n\to \omega} \rho(x_n) \quad ((x_n)_n\in S_{B,\rho}^{(\omega)}).
    \end{equation}
    and
    \begin{equation}
        \|q_B((x_n)_n)\|_{\rho'',\#} = \lim_{n\to \omega} \|x_n\|_{\rho,\#}
    \end{equation}
    for all $(x_n)_n\in S_{B,\rho}^{(\omega)}$.
\end{prop}
\begin{rem}
    The two reduced state-kernel extensions can be connected by the commutative diagram
    \begin{equation} \label{eq:ske_inftyvsomegacommdiagram}
    \begin{tikzcd}
        0 \ar[r] & J_{B,\rho} \ar[r,"j_B"] \ar[d] & S_{B,\rho} \ar[r,"q_B"] \ar[d] & \pi_\rho(B)'' \ar[r] \ar[d, equals] & 0\\
        0 \ar[r] & J_{B,\rho}^{(\omega)} \ar[r,"j_B^{(\omega)}"] & S_{B,\rho}^{(\omega)} \ar[r,"q_B^{(\omega)}"] & \pi_\rho(B)'' \ar[r] & 0
    \end{tikzcd}
    \end{equation}
    where the downward arrows are the appropriate restrictions of the quotient map $\pi_{\infty\to\omega}: B_\infty\to B_\omega$.
\end{rem}

\subsection{Dependence on choice of state}
A very natural question would be which influence the choice of strongly faithful state has on the reduced state-kernel extension. In this section, we will show that $J_{B,\rho}$ only depends on the underlying representation. We will also show that strongly faithful states always exist for the class of \ca s of our interest. First, we recall the following standard result.

\begin{lemma} \label{lem:isoGNSrep}
    Let $M$ be a von Neumann algebra with a faithful normal state $\psi$. Suppose $B$ is a strongly dense $C^*$-subalgebra of $M$, and consider the GNS representation $\pi_\rho: B\to B(\Hb_\rho)$ associated to $\rho := \psi|_B$. There exists an isomorphism
    \begin{equation}
        \alpha: \pi_\rho(B)'' \to M
    \end{equation}
    such that $\rho'' = \psi \circ \alpha$, with $\rho''$ the natural extension of $\rho$ to $\pi_\rho(B)''$.
\end{lemma}

\begin{prop}
    Let $B$ be a separable \ca. Any faithful, non-degenerate representation $\pi: B\to B(\Hb)$ onto a separable Hilbert space $\Hb$ gives a strongly faithful state $\rho$ on $B$. Moreover, all faithful, normal states on $\pi(B)''$ produce the same state-kernel $J_{B,\rho}$.
\end{prop}
\begin{proof}
Fix a faithful, non-degenerate representation $\pi: B\to B(\Hb)$ of $B$ onto a separable Hilbert space $\Hb$. By faithfulness of $\pi$, we can identify $B$ with its image under $\pi$. Let $(\xi_n)_n$ be a countable orthonormal basis for $\Hb$, and define the map
\begin{equation}
    \psi: B(\Hb) \to \mathbb C: x\mapsto \sum_{n=1}^\infty 2^{-n} \langle x\xi_n, \xi_n\rangle.
\end{equation}
Note that it defines a faithful normal state on $B(\Hb)$, because $(\xi_n)_n$ is an orthonormal basis. Now consider $M := B''$. By definition, $B$ is a strongly dense $C^*$-subalgebra of $M$ and $\psi|_M$ is faithful normal state on $M$. By Lemma \ref{lem:isoGNSrep}, it then follows that $\rho := \psi|_B$ is a faithful state on $B$ such that
\begin{equation} \label{eq:ske_doublecommutantsiso}
    M \cong \pi_\rho(B)''
\end{equation}
where $\pi_\rho$ is the GNS representation associated to $\rho$, and such that $\psi|_M$ corresponds to the natural extension $\rho''$ of $\rho$ to $\pi_\rho(B)''$. As $\psi|_M$ is faithful on $M$, it follows that $\rho''$ is faithful on $\pi_\rho(B)''$. All faithful states on $\pi(B)''$ produce the same state-kernel $J_{B,\rho}$ by the formula for $J_{B,\rho}$ in Proposition \ref{prop:ske_strongequivphinorm}. This concludes the proof.
\end{proof}

\begin{cor}\label{cor:stronglyfaithfulstateexists}
    Let $B$ be a separable \ca\ which admits a faithful irreducible representation. Then there exists a strongly faithful state $\rho$ on $B$ such that $\pi_\rho(B)'' \cong B(\Hb)$ for some separable Hilbert space $\Hb$.
\end{cor}
\begin{proof}
Let $\pi: B\to B(\Hb)$ be a faithful irreducible representation of $B$ onto a separable Hilbert space $\Hb$, and consider the induced strongly faithful state $\rho$ on $B$ by the previous proposition. From \eqref{eq:ske_doublecommutantsiso} and irreducibility of $\pi$, it then follows that
\begin{equation}
    \pi_\rho(B)'' \cong \pi(B)'' = B(\Hb),
\end{equation}
as desired.
\end{proof}

\begin{rem}
It is important to point out that the above corollary does \emph{not} state that $\pi_\rho(B)''$ is equal to $B(\Hb_\rho)$, or in other words that $\rho$ is pure (which would contradict what was noted in Remark \ref{rem:faithfulvsstrongly}). It simply gives an isomorphism with the bounded operators on some separable Hilbert space.
\end{rem}

\section{$KK$-existence and -uniqueness} \label{sec:KK}
\subsection{The Cuntz pair picture of $KK$- and $KL$-theory}
\begin{defn}
    Let $\varphi,\psi:A\to E$ be be two \sh s between \ca s $A$ and $E$.
    \begin{itemize}
        \item Suppose $E$ contains a \ca\ $I$ as an ideal. The pair $(\varphi,\psi)$ is said to be a \emph{Cuntz pair} if $\varphi(a)-\psi(a)\in I$ for all $a\in A$. We will denote this as $(\varphi,\psi): A\rightrightarrows E \rhd I$.
        \item Suppose $E = \Mult I$ for some \ca\ $I$. Two Cuntz pairs $(\varphi_i,\psi_i): A\rightrightarrows \Mult I \rhd I$ ($i=0,1$) are said to be \emph{homotopic} if there exists a Cuntz pair
        \begin{equation}
            (\Phi,\Psi): A \rightrightarrows \mathcal{C}_\sigma([0,1], \Mult{I}) \rhd \mathcal{C}([0,1], I)
        \end{equation}
        such that evaluation at $i$ gives the Cuntz pair $(\varphi_i,\psi_i)$ for $i=0,1$.
        \item Suppose $E$ contains a unital copy of $\Otwo$, i.e.\ isometries $s_1, s_2\in E$ satisfying $s_1s_1^* + s_2s_2^* = 1_E$. The \emph{Cuntz sum} of $\varphi$ and $\psi$ is defined as
        \begin{equation}
            (\varphi\oplus_{s_1,s_2}\psi)(a) := s_1\varphi(a)s_1^* + s_1\varphi(a)s_1^* \quad (a\in A).
        \end{equation}
    \end{itemize}
\end{defn}
The Cuntz sum is independent of the choice of $\Otwo$-isometries up to unitary equivalence. We will often omit the explicit mention of $s_1$ and $s_2$ and simply write $\varphi\oplus\psi$ in situations where the specific choice does not matter.

The Cuntz pair picture of $KK$ is defined as follows. Given a separable \ca\ $A$ and a $\sigma$-unital \ca\ $I$, let $KK(A,I)$ be the set of Cuntz pairs $(\varphi,\psi): A\rightrightarrows \Mult {I\otimes \K} \rhd I\otimes \K$ up to Cuntz pair homotopy. We will denote class of such a Cuntz pair as $[\varphi,\psi]_{KK(A,I)}$. As $\Mult {I\otimes \K}$ contains a unital copy of $B(\Hb)$, we can define a binary operation on $KK(A,I)$ by the Cuntz sum: for two Cuntz pairs $(\varphi_i,\psi_i): A\rightrightarrows \Mult I \rhd I$ ($i=1,2$), we write
\begin{equation}
    [\varphi_1,\psi_1]_{KK(A,I)} + [\varphi_2,\psi_2]_{KK(A,I)} := [\varphi_1\oplus\varphi_2,\psi_1\oplus \psi_2]_{KK(A,I)}.
\end{equation}
This turns $KK(A,I)$ into an abelian group with neutral element $[\varphi,\varphi]_{KK(A,I)}$, where $\varphi: A\to \Mult{I\otimes \K}$ is any \sh. The inverse of an element $[\varphi,\psi]_{KK(A,I)}$ is $[\psi,\varphi]_{KK(A,I)}$.

For a \sh\ $\varphi: A\to I$, we will denote $[\varphi]_{KK(A,I)} := [\varphi,0]_{KK(A,I)}$. We can also define the $KK$-class $[\varphi,\psi]_{KK(A,I)}$ of an arbitrary Cuntz pair $(\varphi,\psi): A\rightrightarrows E \rhd I$ as the class of
\begin{equation}
    \big(\sigma\circ(\varphi\otimes e),\sigma(\psi\otimes e)\big): A\rightrightarrows \Mult {I\otimes \K} \rhd I\otimes \K,
\end{equation}
where $\sigma: E\otimes \K \to \Mult{I\otimes \K}$ is the canonical map induced by the inclusion $I\otimes \K\subseteq E\otimes \K$, and $e \in \K$ is a rank-one projection.

A key feature of $KK(\cdot, \cdot)$ is that it is a bifunctor. We can define $KK(A,I)$ when $I$ is not necessarily $\sigma$-unital, by the inductive limit
\begin{equation}
    KK(A,I) := \varinjlim_{I_0 \text{ sep.}} KK(A,I_0)
\end{equation}
ranging over all separable $C^*$-subalgebras $I_0$ of $I$ ordered by inclusion, with connecting maps $KK(A,\iota_{I_0\subseteq I_1})$ for $I_0\subseteq I_1$. A Cuntz pair $(\varphi,\psi): A\rightrightarrows E \rhd I$ induces a class in $KK(A,I)$ as the image in the inductive limit of the element $[\varphi|^{E_0},\psi|^{E_0}]_{KK(A,I_0)}$, where $E_0\subseteq E$ and $I_0\subseteq I$ are separable $C^*$-subalgebras such that $E_0$ contains $I_0$ as an ideal, contains the images of $\varphi$ and $\psi$ and $(\varphi|^{E_0},\psi|^{E_0})$ is a Cuntz pair. In \cite[Appendix B]{CarrionGabeSchafhauserTikuisisWhite_ClassifyingHomomorphismsIUnitalSimpleNuclearCalgebras}, it is shown that this construction is well-defined.

\begin{defn}
    Let $A$ and $I$ be $C^*$-algebras with $A$ separable. Let $\overline{\mathbb N} := \mathbb N \cup \{\infty\}$ be the one-point compactification of $\mathbb N$ and $\ev_n$ the evaluation map at a point $n\in\overline{\mathbb N}$. Define
    \begin{equation}
    KL(A,I)\coloneqq KK(A,I)/Z_{KK(A,I)},
    \end{equation}
    where
    \begin{equation}
    \begin{aligned}
        Z_{KK(A,I)} := \{ KK(A,\ev_\infty)(\kappa) \mid \kappa \in KK(A, \mathcal{C}(\overline{\mathbb N}, I)), KK(A,\ev_n)(\kappa) = 0\, \forall n \in \mathbb N \} 
    \end{aligned}
    \end{equation}
\end{defn}

Both $KK$ and $KL$ enjoy a number of nice properties. Perhaps the most important in general is the so-called Kasparov product, which is a map from $KK(A,I)\times KK(I,J)$ to $KK(A,J)$ generalizing composition of maps. We do not require its full strength in this paper; the following proposition highlights the main facts we need for our purposes.

\begin{prop} \label{prop:KKprops}
    The mapping $KK(\cdot, \cdot)$ is a bifunctor with the following properties. Suppose $A$ and $I$ are \ca s with $A$ separable, and $(\varphi,\psi): A\rightrightarrows E\rhd I$ is a Cuntz pair. Then
    \begin{enumerate}
        \item the map $KK(A,\iota^{(2)}_I): KK(A,I) \to KK(A,M_2(I))$ is an isomorphism, where $\iota^{(2)}_I: I\hookrightarrow M_2(I)$ is the top-left corner inclusion; \label{prop:KKprops.stable}
        \item for any unitary $u\in I^\dagger$, we have \label{prop:KKprops.unitary}
        \begin{equation}
            [\varphi,\psi]_{KK(A,I)} = [\Ad(u)\circ\varphi,\psi]_{KK(A,I)};
        \end{equation}
        \item functoriality can be described explicitly: given \sh s $\theta: C\to A$ and $\rho: I\to J$ such that $\rho$ extends to a \sh\ $\overline{\rho}: E \to F$ with $F \rhd J$, we have \label{prop:KKprops.functoriality}
        \begin{equation}
        \begin{aligned}
            KK(\theta,I)[\varphi,\psi]_{KK(A,I)} &= [\varphi\circ\theta, \psi\circ\theta]_{KK(C,I)}, \\
            KK(A,\rho)[\varphi,\psi]_{KK(A,I)} &= [\overline{\rho}\circ\varphi, \overline{\rho}\circ\psi]_{KK(A,J)}, \\
            KK(A,\iota_{I\subseteq E})[\varphi,\psi]_{KK(A,I)} &= [\varphi]_{KK(A,E)} - [\psi]_{KK(A,E)}.
        \end{aligned}
        \end{equation}
    \end{enumerate}
    These results descend naturally to $KL$.
\end{prop}

From now on, we will generally refer to the facts appearing in part (c) as ``functoriality''. For a proof that all the above statements hold for general $I$, see \cite[Appendix B]{CarrionGabeSchafhauserTikuisisWhite_ClassifyingHomomorphismsIUnitalSimpleNuclearCalgebras}.

We mention one more (crucial) fact about $KK$ and $KL$, namely their relation with (total) $K$-theory and the UCT. There exists a map
\begin{equation}
    \Gamma^{(A,I)}_*: KK(A,I) \to \Hom(K_*(A),K_*(I))
\end{equation}
which is natural in both variables. In general, it is defined using the Kasparov product; for an element $[\varphi]_{KK(A,I)}$ (with $\varphi: A\to I$ a \sh), it satisfies
\begin{equation}
    \Gamma^{(A,I)}_*[\varphi]_{KK(A,I)} = K_*(\varphi).
\end{equation}
There also exists a natural map
\begin{equation}
    \ker \Gamma^{(A,I)}_* \to \Ext(K_*(A), K_{1-*}(I)),
\end{equation}
but we will not need its explicit form in this paper.

\begin{defn}
    A separable \ca\ $A$ is said to \emph{satisfy the universal coefficient theorem (UCT)} if $\Gamma^{(A,I)}_*$ is surjective, and the map $\ker \Gamma^{(A,I)}_* \to \Ext(K_*(A), K_{1-*}(I))$ is bijective.
\end{defn}

When $A$ satisfies the UCT, we get a short exact sequence
\begin{equation}
\begin{tikzcd}
    0 \ar[r] & \Ext(K_*(A), K_{1-*}(I)) \ar[r] & \Gamma^{(A,I)}_* \ar[r, "\Gamma^{(A,I)}_*"] & \Hom(K_*(A),K_*(I)) \ar[r] & 0,
\end{tikzcd}
\end{equation}
allowing for the computation of $KK$-theory in terms of $K$-theory. We also mention \emph{universal multicoefficient theorem}, relating $KK$-theory to total $K$-theory.

\begin{thm}[The universal multicoefficient theorem] \label{thm:KKactiononK}
    Let $A$ and $I$ be \ca s with $A$ separable. There exists an action of $KK(A,I)$ on (total) $K$-theory, in the sense that there is a homomorphism
    \begin{equation}
        \Gamma^{(A,I)}_\Lambda: KK(A,I) \to \Hom_\Lambda(\totalK(A),\totalK(I))
    \end{equation}
    which is natural in both variables. Given a \sh\ $\varphi:A\to I$, it satisfies
    \begin{equation}
        \Gamma^{(A,I)}_\Lambda([\varphi]_{KK(A,I)}) = \totalK(\varphi).
    \end{equation}
    Moreover, it factors through $KL(A,I)$, inducing a map
    \begin{equation}
        \overline\Gamma^{(A,I)}_\Lambda: KL(A,I) \to \Hom_\Lambda(\totalK(A),\totalK(I))
    \end{equation}
    which is again natural in both variables, and an isomorphism if $A$ satisfies the UCT.
\end{thm}
\begin{proof}
See \cite[\S 8.2]{CarrionGabeSchafhauserTikuisisWhite_ClassifyingHomomorphismsIUnitalSimpleNuclearCalgebras}; in particular, Theorem 8.5.
\end{proof}

\subsection{Absorption \& pure largeness} \label{sec:absorption}
\begin{defn}
    Let $A$ and $I$ be \ca s with $A$ separable and $I$ $\sigma$-unital and stable, and let $\pi_I: \Mult I \rightarrow \Q I$ be the quotient map. A \sh\ $\varphi: A\rightarrow \Q I$ is said to be \emph{(unitally) absorbing} if for any (unital) \sh\ $\psi: A\to \Mult I$, there exists a unitary $u\in \Mult I$ such that
    \begin{equation}
        \varphi\oplus(\pi_I \circ \psi) = \Ad(\pi_I(u))\circ \varphi.
    \end{equation}
    A \sh\ $\varphi: A\rightarrow\Mult{I}$ is said to be \emph{(unitally) absorbing} if $\pi_I\circ \varphi: A\rightarrow \Q I$ is (unitally) absorbing as above.
\end{defn}

\begin{prop} \label{prop:kkeu_absorbingvsunitization}
    Let $A$ and $I$ be \ca s with $I$ $\sigma$-unital and stable. A \sh\ $\varphi: A\to \Q I$ is absorbing if and only if $\varphi^\dagger: A^\dagger \to \Q I$ is unitally absorbing.
\end{prop}
\begin{proof}
First suppose $\varphi$ is absorbing and let $\theta: A^\dagger\to \Q I$ be a unital \sh\ with a unital $*$-homomorphic lift to $\Mult I$. Then $\theta|_A$ lifts to a \sh\ to $\Mult I$, so by assumption there exists a unitary $u\in\Mult I$ such that $\varphi\oplus_{\pi_I(s_1),\pi_I(s_2)}\theta|_A = \Ad_{\pi_I(u)}\circ\varphi$, where $\pi_I: \Mult I \to \Q I$ is the quotient map and $s_1$ and $s_2$ are $\Otwo$-isometries. Then for all $a\in A$ and $\lambda \in \mathbb C$, we have
\begin{equation}
\begin{aligned}
    (\varphi^\dagger&\oplus_{\pi_I(s_1),\pi_I(s_2)}\theta)(a+\lambda 1_{A^\dagger}) \\
        &= \pi_I(s_1)(\varphi(a)+\lambda 1_{\Q I})\pi_I(s_1)^* + \pi_I(s_2)(\theta(a)+\lambda 1_{\Q I})\pi_I(s_2)^*\\
        &= \pi_I(s_1)\varphi(a)\pi_I(s_1)^* + \pi_I(s_1)\theta(a)\pi_I(s_1)^* + \lambda \pi_I(s_1s_1^* + s_2s_2^*)\\
        &= \Ad_{\pi_I(u)}\circ\varphi(a) + \lambda 1_{\Q I} = \Ad_{\pi_I(u)}\circ\varphi^\dagger(a+\lambda 1_{A^\dagger}),
\end{aligned}
\end{equation}
which shows that $\varphi^\dagger$ absorbs $\theta$.

Now suppose $\varphi^\dagger$ is unitally absorbing and $\theta: A\to \Q I$ is a \sh\ with a $*$-homomorphic lift $\gamma: A\to \Mult I$. We then claim that $\gamma^\dagger$ lifts $\theta^\dagger$. Indeed, for all $a\in A$ and $\lambda \in \mathbb C$, we have
\begin{equation}
\begin{aligned}
    q\circ \gamma^\dagger(a+\lambda 1_{A^\dagger})
        &= \pi_I(\gamma(a)) + \lambda \pi_I(1_{\Mult I}) = \theta(a) + \lambda 1_{\Q I}\\
        &= \theta^\dagger(a+\lambda 1_{A^\dagger}).
\end{aligned}
\end{equation}
Hence, $\theta^\dagger$ has a unital $*$-homomorphic splitting, so by assumption, $\varphi^\dagger$ absorbs $\theta^\dagger$. By restricting to $A$, it then follows that $\varphi$ absorbs $\theta$, concluding the proof.
\end{proof}

Pure largeness is a property of extensions originally introduced by Elliott and Kucerovsky in their landmark paper \cite{ElliottKucerovsky_AnAbstractVoiculescuBrownDouglasFillmoreAbsorptionTheorem}. In this paper, they showed that this property is equivalent to the associated Busby map being unitally absorbing under the right conditions. This result gives us a concrete criterion to verify in order to obtain absorption of the relevant maps. Recall that for two positive elements $a$ and $b$ in a \ca\ $A_+$, $a$ is said to be \emph{Cuntz subequivalent} to $b$, denoted $a\precsim b$, if for each $\varepsilon>0$, there exists a $c\in A$ such that $a \approx_{\varepsilon} b$. We write $a\sim b$ if $a\precsim b$ and $b\precsim a$.

\begin{defn}
    An extension 
    \begin{equation}
    \begin{tikzcd}
        \mathfrak e: 0 \arrow[r] & I \arrow[r,"j"] & E \arrow[r,"q"] & D \arrow[r] & 0
    \end{tikzcd}
    \end{equation}
    is said to be \emph{purely large} if for each $y\in I_+$ and $x\in E_+ \setminus I$, we have $y\precsim x$ in $E$.
\end{defn}

\begin{rem}
    Our definition of pure largeness is different from the one stated by Elliott and Kucerovsky in \cite{ElliottKucerovsky_AnAbstractVoiculescuBrownDouglasFillmoreAbsorptionTheorem}. As we will show in Proposition \ref{prop:elku_purelylargeequiv}, the two are equivalent when the ideal is $\sigma$-unital and stable (which is the setting in which their theorem is stated); in fact, stability of the ideal will be automatic from our definition in the $\sigma$-unital setting.
\end{rem}

We start by making the important observation that one can always control the norm of the element implementing the Cuntz subequivalence. This will be important later down the line, specifically in Proposition \ref{eq:sep_purelylargesepinh}.

\begin{prop} \label{prop:purelylargenormcontrol}
    Let 
    \begin{equation}
    \begin{tikzcd}
        \mathfrak e: 0 \arrow[r] & I \arrow[r,"j"] & E \arrow[r,"q"] & D \arrow[r] & 0
    \end{tikzcd}
    \end{equation}
    be a purely large extension. Then for any $x\in E_+\setminus I$ and $y\in I_+$, we can always arrange the element $c$ implementing the Cuntz subequivalence $y\precsim x$ for a given $\varepsilon>0$ to be in $I$ and satisfy
    \begin{equation}
        \|c\|^2 \leq \|y\|/\|q(x)\|.
    \end{equation}
\end{prop}
\begin{proof}
By multiplying $c$ from the right by an appropriate element of an approximate unit for $I$, we can always arrange $c\in I$. For the second part of the claim, take $x\in E_+\setminus I$, $y\in I_+$ and $0<\varepsilon < 1$. By dividing $x$ by $\|q(x)\|\neq 0$, we can assume that $\|q(x)\| = 1$ without loss of generality. The statement is trivial for $y=0$, so suppose $y\neq 0$. By dividing $y$ and $\varepsilon$ by $\|y\|$, we can further assume $\|y\|=1$.

First, we will additionally suppose that $\|x\|=1$. Let
\begin{equation}
    \lambda := 1 - \frac{\varepsilon}{2}>0 
\end{equation}
and define
\begin{equation}
    g:[0,1]\rightarrow [0,1]: t \mapsto \min\{t/\lambda, 1\}.
\end{equation}
Then $(x-\lambda)_+g(x) = (x-\lambda)_+$ and $g(x) \approx_{\varepsilon/2} x$. As $\|q(x)\|=1$, it follows that
\begin{equation}
    q((x-\lambda_+)) = (q(x)-\lambda)_+ \neq 0,
\end{equation}
so $(x-\lambda)_+ \notin I$. Hence, we can apply pure largeness to find a $c_0 \in I$ such that
\begin{equation}
    c_0^*(x-\lambda)_+c_0 \approx_{\varepsilon/4} (y-\varepsilon/4)_+ \approx_{\varepsilon/4} y.
\end{equation}
Letting $c := (x-\lambda)_+^{1/2} c_0 \in I$, we then find that
\begin{equation}
    \|c\|^2 = \|c_0^*(x-\lambda)_+c_0\| \leq \|(y-\varepsilon/4)_+\| + \frac{\varepsilon}{4} = 1
\end{equation}
as $\|y\|=1$, and
\begin{equation}
    c^*xc \approx_{\varepsilon/2} c^*g(x)c = c_0^*(x-\lambda)_+c_0 \approx_{\varepsilon/2} y,
\end{equation}
as desired.

Now for the general case, define the continuous function
\begin{equation}
    f:[0,\|x\|] \rightarrow [0,1]: t \mapsto \begin{cases}
    2t/\varepsilon & t<\varepsilon/2\\ 1 & \varepsilon/2 \leq t < 1 \\ 1/t & t\geq 1
\end{cases}
\end{equation}
and note that $f(x)^{1/2} x f(x)^{1/2} = x f(x) \approx_{\varepsilon/2} h(x)$, where
\begin{equation}
    h:[0,\|x\|] \rightarrow [0,1] :t \mapsto \min\{ t, 1\}.
\end{equation}
Remark that $\|h(x)\| = 1$, that
\begin{equation}
    \|q(h(x))\| = \|h(q(x))\| = 1
\end{equation}
as $\|q(x)\| = 1$, and subsequently that $h(x)\notin I$, so we can apply the specific case to find a $c_1\in I$ satisfying $\|c_1\|\leq 1$ and
\begin{equation}
    c_1^*h(x)c_1 \approx_{\varepsilon/2} y.
\end{equation}
Subsequently, the element $c := f(x)^{1/2} c_1 \in I$ satisfies
\begin{equation}
    \|c\| \leq \| f(x)^{1/2}\|\cdot \|c_1\| \leq 1
\end{equation}
and
\begin{equation}
    c^*xc \approx_{\varepsilon/2} c_1^*h(x) c_1 \approx_{\varepsilon/2} y
\end{equation}
as $f(x)^{1/2} x f(x)^{1/2} \approx_{\varepsilon/2} h(x)$, which is exactly what we needed to show.
\end{proof}

A second observation is that pure largeness is in fact closely related to pure infiniteness, specifically to the ideal being simple and purely infinite. This is illustrated by the following proposition.

\begin{prop} \label{prop:sep_purelyinfimpliespurelylarge}
    Let $A$ and $I$ be \ca s with $I$ non-zero, and
    \begin{equation}
    \begin{tikzcd}
        \mathfrak e: 0 \ar[r] & I \ar[r] & E \ar[r] & D \ar[r] & 0
    \end{tikzcd}
    \end{equation}
    an extension of $D$ by $I$. If $\mathfrak e$ is purely large, then $I$ is essential in $E$. The converse holds if $I$ is simple and purely infinite.
\end{prop}
\begin{proof}
First suppose $\mathfrak e$ is purely large and take $y\in E$ such that $yI = Iy = 0$. The same then holds for the positive element $z := y^*y$. We will argue by contradiction and suppose that $z\notin I$. Then we can apply pure largeness to find that $x\precsim z$ for all $x\in I_+$. Note that the elements implementing these subequivalences can be chosen to be in $I$, so for each $\varepsilon > 0$ and $x\in I_+$, there exists $c\in I$ such that
\begin{equation}
    \|c^*zc - x\| < \varepsilon.
\end{equation}
However, since $zc = 0$, it follows that $\|x\| < \varepsilon$, which is only possible if $x=0$. As the positive elements linearly span $I$, this implies that $I = 0$, which contradicts the assumption that $I$ is simple (and therefore non-zero). Therefore, we can only have $z\in I$, but then in particular $z^2=0$ so $z=0$, which implies that $y=0$. This shows that $I$ is essential in $E$.

Now suppose $I$ is simple, purely infinite and essential in $E$ and let $x\in I_+$ and $y\in E_+\setminus I$. As $I$ is essential, there exists some $c\in I$ such that $yc \neq 0$. Hence, the element $c^*y^2c \in I$ is positive and non-zero, so we can apply pure infiniteness of $I$ to find that $x \precsim c^*y^2c$. However, we have $c^*y^2c \precsim y$ by definition, so in fact $x\precsim y$, as desired.
\end{proof}

The main use case of pure largeness will not be for the reduced state-kernel extension itself, but rather for its pull-back by a given unital and full \sh.

\begin{prop} \label{prop:sep_pullbackpurelyarge}
Let $A$ be a \ca, and
\begin{equation}
    \begin{tikzcd}
    \mathfrak e: 0 \arrow[r] & I \arrow[r,"j"] & E \arrow[r,"q"] & D \arrow[r] & 0
    \end{tikzcd}
\end{equation}
a purely large extension of \ca s. Then the pull-back of $\mathfrak e$ by any injective \sh\ $\varphi: A\rightarrow D$ is purely large.
\end{prop}
\begin{proof}
Consider the pull-back $\mathfrak e \circ \varphi$ as given in \eqref{eq:sep_pullback}. Pick $(x,a)\in (E_\varphi)_+\setminus \iota(I)$ and $y\in I_+$. Then $a = \pi_A(x,a) \neq 0$, so also $q(x) = \varphi(a) \neq 0$ by injectivity of $\varphi$. Hence, $x\notin I$, so we can apply pure largeness of $\mathfrak e$ to find that $y\precsim x$ in $E$. Now, we claim that this implies that $(y,0)\precsim (x,a)$ in $E_\varphi$. Indeed, if $c \in E$ is an element implementing this subequivalence for a given $\varepsilon >0$, then we can arrange that $c\in I$ by Proposition \ref{prop:purelylargenormcontrol}. Subsequently, the element $(c,0)$ is in $E_\varphi$ and implements the subequivalence $(y,0)\precsim (x,a)$ in $E_\varphi$ for that $\varepsilon$.
\end{proof}

In order to prove Proposition \ref{prop:elku_purelylargeequiv}, we will need the following lemma.

\begin{lemma}
    Let $I$ be a \ca\ and $x\in \Mult I_+$. If for each $x_0\in \Mult I_+$ such that $x-x_0\in I$, the element $x_0$ Cuntz dominates every positive element in $I$, then $I$ is separably stable and $\overline{xIx}$ is separably stable and full in $I$.
\end{lemma}
\begin{proof}
Suppose $x\in \Mult I_+$ is as in the statement. To prove that $I$ is separably stable, pick $y\in I_+$. Fix $\varepsilon> 0$; we can assume $\varepsilon < \|y\|$ because the statement is trivial if $y=0$. If we let
\begin{equation}
    g:[0,\|b\|]\rightarrow [0,1]: t \mapsto \min\{t/\varepsilon, 1\},
\end{equation}
then the positive element $g(y)\in C^*(y)$ satisfies $g(y)(y-\varepsilon)_+ = (y-\varepsilon)_+$ by functional calculus. Now define the positive element
\begin{equation}
    x_0 := (1-g(y))x(1-g(y)) \in \Mult I,
\end{equation}
and note that
\begin{equation}
    x-x_0 = -g(y)x -xg(y) + g(y)xg(y) \in I
\end{equation}
because $g(y)\in C^*(y) \subseteq I$, so by assumption $y\precsim x_0$ in $\Mult I$. Rørdam's lemma \cite[Proposition 2.6]{KirchbergRordam_NonSimplePurelyInfiniteCalgebras} then gives a $\delta>0$ and $z\in \Mult I$ such that $z^*z = (y-\varepsilon)_+$ and $zz^*\in \overline{(x_0-\delta)_+ \Mult I (x_0-\delta)_+}$. In particular, as $(x_0-\delta)_+\in C^*(x_0) \subseteq \overline{x_0 I x_0}$, we have $\overline{(x_0-\delta)_+ I (x_0-\delta)_+}\subseteq \overline{x_0 I x_0}$, so also $zz^* \in \overline{x_0 I x_0}$. As $x_0 \perp (y-\varepsilon)_+$, it follows that
\begin{equation}
    zz^* \perp z^*z = (y-\varepsilon)_+ \approx_\varepsilon y,
\end{equation}
which proves that $I$ is separably stable.

To show that $\overline{xIx}_+$ is separably stable, apply the exact same argument for $y\in \overline{xIx}_+$. Remark that in this case, $x-x_0\in \overline{xIx}$ because $g(y)\in C^*(y)\subseteq \overline{xIx}$. Hence, we have
\begin{equation}
    x_0 = x-(x-x_0) \in \overline{x\Mult I x} + \overline{xIx} \subseteq \overline{x\Mult I x},
\end{equation}
and since $(y-\varepsilon)_+ \in \overline{xIx}$, it follows that we can arrange the element $z$ we obtain to lie in $\overline{xIx}$. This proves that $\overline{xIx}$ is separably stable.

To show that $\overline{x I x}$ is full in $I$, pick an arbitrary $y\in I_+$ and $\varepsilon> 0$. Applying the previous argument for $x_0=x$, we find a $z\in \Mult I$ such that $z^*z = (y-\varepsilon)_+$ and $zz^*\in \overline{x_0 I x_0} =  \overline{x I x}$. This implies that the ideal in $I$ generated by $\overline{x I x}$ contains $z^*z = (b-\varepsilon)_+$. As $(y-\varepsilon)_+ \approx_\varepsilon y$ and $\varepsilon$ was arbitrary, it follows that $y\in \overline{I\overline{x I x}I}$. Since each $y\in I$ is a linear combination of four positive elements in $I$, we conclude that $\overline{x I x}$ is indeed full.
\end{proof}

Now, we are ready to prove the relation between our notion of pure largeness and that of Elliott and Kucerovsky. This result is also to appear in the follow-up paper to \cite{CarrionGabeSchafhauserTikuisisWhite_ClassifyingHomomorphismsIUnitalSimpleNuclearCalgebras}; we kindly thank the authors for allowing us to include it and its proof in this paper.

\begin{prop}
\label{prop:elku_purelylargeequiv}
    Let
    \begin{equation}
    \begin{tikzcd}
        \mathfrak e: 0 \arrow[r] & I \arrow[r,"j"] & E \arrow[r,"q"] & D \arrow[r] & 0
    \end{tikzcd}
    \end{equation}
    be an extension of \ca s. The following are equivalent:
    \begin{enumerate}
        \item $\mathfrak e$ is purely large;
        \item $I$ is separably stable and for each $x\in E\setminus I$, the $C^*$-subalgebra $\overline{x^*Ix}$ is stable and full in $I$;
        \item $I$ is separably stable and $\mathfrak e$ is purely large in the sense of Elliott and Kucerovsky:\\ for each $x\in E\setminus I$, there exists a $\sigma$-unital stable $C^*$-subalgebra $D\subseteq \overline{x^*Ix}$ which is full in $I$.
    \end{enumerate}
    In particular, if $I$ is $\sigma$-unital and stable, then our definition of pure largeness coincides with Elliott and Kucerovsky's.
\end{prop}
\begin{proof}
$(a)\implies (b)$: consider for arbitrary $x\in E\setminus I$ the hereditary $C^*$-subalgebra $\overline{x^*Ix}$. First note that we can assume $x$ is positive without loss of generality. Indeed, we have $\overline{x^*Ix} = \overline{x^*xIx^*x}$, and $x^*x$ cannot be in $I$, because otherwise polar decomposition in \ca s would imply that
\begin{equation}
    x = v(x^*x)^{1/4} \in I
\end{equation}
for some $v\in E$. Now, for each $x_0\in \Mult I_+$ such that $x-x_0 \in I$, it follows that $x_0 \notin I$ (because otherwise $x\in I$), so by assumption we have in particular that $y\precsim x_0$ for all $y\in I$. Hence, the previous lemma tells us that $I$ is separably stable and $\overline{xIx}$ is separably stable and full in $I$. As $\overline{xIx}$ is $\sigma$-unital, it is in fact stable by \cite{HjelmborgRordam_OnStabilityofCalgebras}. This is exactly what we needed to show.

$(b)\implies (c)$: immediate.

$(c)\implies (a)$: let $x\in E\setminus I$ be a positive element. By assumption, there exists a $\sigma$-unital, stable $C^*$-subalgebra $D$ of $\overline{xIx}$ which is full in $I$. As $D$ is $\sigma$-unital, it contains a strictly positive element $h$, which satisfies $D =\overline{hDh}$. Since also $h\in \overline{xIx}$, it follows that $h\precsim x$. Now $\overline{hDh}$ is stable, which implies that $h$ is properly infinite (in $D$, so in particular also in $I$) by \cite[Proposition 3.7]{KirchbergRordam_NonSimplePurelyInfiniteCalgebras}. As $\overline{IhI} = I$ by fullness of $\overline{hDh}$, it then follows from \cite[Proposition 3.5]{KirchbergRordam_NonSimplePurelyInfiniteCalgebras} that $y\precsim h \precsim x$ for each $y\in \overline{IhI}_+ = I_+$, as desired.

Finally, the particular consequence follows after noting that stability and separable stability are equivalent in the $\sigma$-unital setting by \cite[Theorem 2.1]{HjelmborgRordam_OnStabilityofCalgebras}.
\end{proof}

We now explicitly state the Elliott--Kucerovsky theorem. For a proof, see \cite{ElliottKucerovsky_AnAbstractVoiculescuBrownDouglasFillmoreAbsorptionTheorem}. Note that the statement below holds with our definition of pure largeness, by the previous proposition.

\begin{thm}[Elliott--Kucerovsky] \label{thm:elliottkucerovsky}
Let $A$ be a separable unital \ca, $I$ a $\sigma$-unital and stable \ca\ and $\mathfrak e$ a unital extension of $A$ by $I$ with associated Busby map $\tau$. Then $\tau$ is unitally nuclearly absorbing if and only if $\mathfrak{e}$ is purely large. In particular, if $A$ is nuclear, then $\tau$ is unitally absorbing if and only if it is purely large.
\end{thm}

We end this subsection by proving that the reduced state-kernel extension is indeed purely large.

\begin{prop} \label{prop:sep_skepurelylarge}
    Let $B$ be a simple and purely infinite \ca\ with a strongly faithful state $\rho$. Then the reduced state-kernel extensions $\mathfrak e_B$ and $\mathfrak e_B^{(\omega)}$ are purely large.
\end{prop}
\begin{proof}
The proof in the ultrapower setting is completely analogous, so let us prove that $\mathfrak e_B$ is purely large. Let $y\in (J_{B,\rho})_+$ and $x\in (S_{B,\rho})_+ \setminus J_{B,\rho}$. Fixing positive lifts $(x_n)_n, (y_n)_n \in \ell^\infty(B)$ of $x$ resp.\ $y$, it follows from Corollary \ref{cor:ske_maptoquotient} that
\begin{equation}
    0 < \|q_B(x)\|_{\rho'',\#} = \lim_{n\to\infty} \|x_n\|_{\rho,\#},
\end{equation}
as $q_B(x)\neq 0$ and $\rho''$ is faithful. Hence, there exist $\delta>0$ and $N\in\mathbb N$ such that $\|x_n\|_{\rho,\#}>\delta$ for all $n\geq N$. By replacing $x_n$ for $n<N$ by an appropriate element of an approximate unit for $B$, we may assume $\|x_n\|_{\rho,\#}>\delta$ for all $n\in\mathbb N$. Note that the resulting sequence will still represent $x$, and that this also implies that $\|x_n\| > \delta$ for all $n\in\mathbb N$ by \eqref{eq:ske_normdominatesphinorm}.

Fix an $n\in\mathbb N$, and define $B_n$ as the hereditary $C^*$-subalgebra of $B$ generated by the positive element $(x_n-\delta/2)_+$. Then $B_n$ is non-zero as $(x_n-\delta/2)_+ \neq 0$, and simple and purely infinite since $B$ is. Hence, it contains a positive norm-$1$ element $b_n$ with infinite spectrum. Considering the commutative $C^*$-subalgebra $C^*(b_n)$, note that the restriction $\rho|_{C^*(b_n)}$ of $\rho$ is induced by some Radon measure $\mu_n$ on $\sigma(b_n)$, in the sense that
\begin{equation}
    \rho(g(b_n)) = \int_{\sigma(b_n)} g(t) \dee \mu_n(t) \quad \big(g\in \mathcal{C}_0(\sigma(b_n)\setminus \{0\})\big).
\end{equation}
Pick $n$ distinct points $\lambda_1, ...,\lambda_n \in \sigma(b_n)$ and let $X_1, ..., X_n$ be mutually disjoint open neighborhoods of the $\lambda_j$. As $\mu_n$ is Radon, it follows that
\begin{equation}
    \sum_{j=1}^n \mu_n(X_j) = \mu_n\left( \bigsqcup_{j=1}^n X_j \right) \leq \mu_n(\sigma(b_n)) = \rho(b_n) \leq \|b_n\| = 1,
\end{equation}
so there exists some $j\in\{1,...,n\}$ such that $\mu_n(X_j) \leq 1/n$. Now let $h_n: \sigma(b_n) \to [0,1]$ be a continuous function supported on this $X_j$ such that $h_n(\lambda_j) = 1$, and note that
\begin{equation}
    \rho(h_n(b_n)) = \int_{\sigma(b_n)} h_n(t) \dee \mu_n(t) \leq \mu_n(X_j) \leq \frac{1}{n}.
\end{equation}
Hence, $z_n := h_n(b_n)^{1/2}$ is a positive norm-$1$ element in $B_n$ with $\|z_n\|_{\rho,\#} = \rho(z_n^2) \leq 1/n$. Moreover, if we let
\begin{equation}
    f: [0,1] \to [0,1]: t\mapsto \max\{2t/\delta,1\},
\end{equation}
then $f(x_n)z_n = z_n$, as $f(x_n)$ acts as the identity on $(x_n-\delta/2)_+$.

The resulting sequence $(z_n)_n$ is then norm-bounded and converges to zero in $(\rho,\#)$-norm, so $(z_n)_n\in J_{B,\rho}$.  Now, use the fact that $B$ is simple and purely infinite to find for each $n\in\mathbb N$ a $c_n\in B$ such that
\begin{equation}
    c_n^*z_n^2c_n = y_n^{1/2}.
\end{equation}
By \cite[Lemma 4.1.7]{Rordam_ClassificationofNuclearCalgebras}, we can arrange that
\begin{equation}
    \|c_n\| \leq \frac{\|y_n^{1/2}\|^{1/2}}{\|z_n^2\|^{1/2}}+1 = \frac{\|y_n\|^{1/4}}{\|z_n\|} + 1 \leq \|y\|^{1/4}+1,
\end{equation}
so the $c_n$ form a bounded sequence in $B$. The sequence $d_n := z_n c_n y_n^{1/4}$ ($n\in\mathbb N$) is then also bounded, and using \eqref{eq:ske_phinormofproduct}, it satisfies
\begin{equation}
    2\|d_n\|_{\rho,\#} \leq \|z_n c_n\| \|y_n^{1/4}\|_\rho + \|y_n^{1/4}c_n^*\| \|z_n\|_\rho \xrightarrow{n\to\infty} 0
\end{equation}
as both $y_n^{1/4}$ and $z_n$ are in $J_{B,\rho}$. Hence, $(d_n)_n$ is a sequence in $J_{B,\rho}$, and
\begin{equation}
    d_n^*f(x_n)d_n = d_n^*d_n = y_n
\end{equation}
for all $n\in\mathbb N$ as $f(x_n)$ acts as the identity on $z_n$, so $y\precsim f(x)$. As $f(x) \sim x$ by \cite[Lemma 2.2]{KirchbergRordam_NonSimplePurelyInfiniteCalgebras}, we can conclude that $y\precsim x$. This finishes the proof.
\end{proof}

\subsection{$KK$-existence and -uniqueness}
\begin{thm}[$KK$-existence]\label{thm:KK-Existence}
    Let $A$ be a separable \ca\ and $I$ a $\sigma$-unital, stable \ca.
    \begin{enumerate}
        \item If $\psi: A \to \Mult I$ is an absorbing \sh\ and $\kappa \in KK(A, I)$, then there is an absorbing \sh\ $\varphi: A \to \Mult I$ such that $(\varphi, \psi)$ is an $(A, I)$-Cuntz pair and $[\varphi,\psi]_{KK(A,I)} = \kappa$. \label{thm:KK-Existence.C1}
        \item Suppose $A$ is also nuclear. If $\theta: A \to \Q I$ is an absorbing \sh, then there is a (necessarily) absorbing \sh\ $\rho: A \to \Mult I$ lifting $\theta$ if and only if $[\theta]_{KK(A,\Q I)}=0$. \label{thm:KK-Existence.C2}
    \end{enumerate}
\end{thm}
\begin{proof}
See \cite[Theorem 5.14]{CarrionGabeSchafhauserTikuisisWhite_ClassifyingHomomorphismsIUnitalSimpleNuclearCalgebras}.
\end{proof}

The proof of $KK$- and $KL$-uniqueness follows the same strategy as in \cite[\S 5.4]{CarrionGabeSchafhauserTikuisisWhite_ClassifyingHomomorphismsIUnitalSimpleNuclearCalgebras}. The main difference is that we aim to obtain uniqueness on the nose, without having to stabilize by $\Z$. In principle, one could just use the $\Z$-stable $KK$-uniqueness theorem, as $\Z$-stability is automatic in our setting. As stated in the introduction, however, we want to avoid doing so, meaning we will have to take a different approach. For simple and purely infinite \ca s, one can appeal to a result by Loreaux and Ng \cite[Theorem 2.5]{LoreauxNg_RemarksOnEssentialCodimension} to obtain $K_1$-injectivity of $\Q I \cap q(\varphi(A))'$ on the nose, without having to stabilize by $\Z$. However, this requires $J_{B,\rho}$ (which, up to separabilization, plays the role of $I$ in this context) to be simple and purely infinite. This is not the case in the sequence algebra setting, as $B_\infty$ is not simple and purely infinite when $B$ is, but it does hold for the ultrapower $B_\omega$ (see \cite[Proposition 6.2.6]{Rordam_ClassificationofNuclearCalgebras}).

\begin{lemma} \label{lem:AsympEquiv}
    Let $A$ be a unital, separable, nuclear $C^*$-algebra, and let $I$ be a separable, simple, purely infinite and stable \ca. Suppose $\varphi, \psi: A \to \Mult I$ are unital, unitally absorbing \sh s such that $\pi_I \circ \varphi = \pi_I \circ\psi$, where $\pi_I:\Mult I \to \Q I$ is the quotient map. Suppose further that there is a continuous path $(u_t)_{t \geq 0}$ of unitaries in $\Mult I$ such that
    \begin{equation}\label{eq:asymp-equiv-lemma}
        \big( t \mapsto u_t \varphi(a) u_t^* - \psi(a) \big) \in C_0([0, \infty), I), \quad a\in A.
    \end{equation}
    If $(u_t)_{t\geq 0}$ can be chosen with $[\pi_I(u_0)]_1 = 0$ in $K_1(\Q I \cap \pi_I(\varphi(A))')$, then we can arrange that $u_t \in I^\dag \subseteq \Mult I$ for all $t \geq 0$.
\end{lemma}
\begin{proof}
First note that indeed $\pi_I(u_0) \in \Q I \cap \pi_I(\varphi(A))'$, as
\begin{equation}
    \Ad_{\pi_I(u_0)} \circ \pi_I\circ \varphi = \pi_I \circ \Ad_{u_0} \circ \varphi = \pi_I\circ \psi = \pi_I\circ\varphi.
\end{equation}
By \cite[Lemma 5.16.i]{CarrionGabeSchafhauserTikuisisWhite_ClassifyingHomomorphismsIUnitalSimpleNuclearCalgebras}, it suffices to show that $\Q I \cap \pi_I(\varphi(A))'$ is $K_1$-injective. This follows from \cite[Theorem 2.5]{LoreauxNg_RemarksOnEssentialCodimension}.
\end{proof}

\begin{thm}[$KK$- and $KL$-uniqueness] \label{thm:KK-Uniqueness}
    Let $A$ be a unital separable, nuclear $C^*$-algebra, and let $I$ be a separable, simple, purely infinite and stable \ca. Let $(\varphi,\psi): A\rightrightarrows \Mult I \rhd I$ be a Cuntz pair of absorbing \sh s.
    \begin{enumerate}
        \item If $[\varphi, \psi]_{KK(A,I)} = 0$, then there is a norm-continuous path $(u_t)_{t \geq 0}$ of unitaries in $I^\dag$ such that
        \begin{equation}
            \| u_t \varphi(a) u_t^* - \psi(a) \| \rightarrow 0, \quad a \in A.
        \end{equation}
        \item If $[\varphi, \psi]_{KL(A,I)} = 0$, then there is a sequence $(u_n)_{n=1}^\infty$ of unitaries in $I^\dag$ such that
        \begin{equation}
            \| u_n \varphi(a) u_n^* - \psi(a) \| \rightarrow 0, \quad a \in A.
        \end{equation}
    \end{enumerate}
\end{thm}
\begin{proof}
By using Lemma \ref{lem:AsympEquiv}, the exact same proof as that of \cite[Theorem 5.15]{CarrionGabeSchafhauserTikuisisWhite_ClassifyingHomomorphismsIUnitalSimpleNuclearCalgebras} works here. The only difference lies in removing the tensorial factor $\Z$ in the computations in the second half of the proof.
\end{proof}

\section{Separabilization of the reduced state-kernel extension} \label{sec:separabilization}
The goal of this section is to set up the technical machinery of separabilization, based on the concept of separably inheritable properties. 

\subsection{Separably inheritable properties}
The concept of separably inheritable properties was introduced by Blackadar; see \cite[\S II.8.5]{Blackadar_OperatorAlgebrasTheoryofCalgebrasandvonNeumannAlgebras}.
\begin{defn}
    A property $(P)$ of \ca s is said to be \emph{separably inheritable} if
    \begin{enumerate}
        \item whenever a \ca\ $A$ satisfies $(P)$ and $A_0$ is a separable $C^*$-subalgebra of $A$, then there exists a separable $C^*$-subalgebra $A_1$ of $A$ which satisfies $(P)$ and contains $A_0$;
        \item whenever $(A_n,\rho_n)_n$ is an inductive system of separable \ca s such that each $A_n$ satisfies $(P)$ and each $\rho_n$ is injective, then $\varinjlim A_n$ also satisfies $(P)$.
    \end{enumerate}
\end{defn}

\begin{defn}
    Let $(P)$ be a property of \ca s. We say a \ca\ $A$ \emph{separably satisfies} $(P)$ if for every separable $C^*$-subalgebra $A_0$ of $A$, there exists a separable $C^*$-subalgebra $A_1$ of $A$ containing $A_0$ and satisfying $(P)$.
\end{defn}

Morally speaking, any property of \ca s which can be formulated as an approximate statement with a control on the norms of relevant elements, is separably inheritable. This fact is exemplified by the following two propositions.

\begin{prop}
    The notion of separable stability is defined unambiguously: an arbitrary \ca\ $A$ satisfies the Hjelmborg-R\o rdam criterion if and only if each separable $C^*$-subalgebra of $A$ is contained in a separable, stable $C^*$-subalgebra.
\end{prop}
\begin{proof}
See \cite[Proposition 6.10]{CarrionGabeSchafhauserTikuisisWhite_ClassifyingHomomorphismsIUnitalSimpleNuclearCalgebras}.
\end{proof}

\begin{prop} \label{prop:sep_simplepurelyinfsepinh}
    The property of being simple and purely infinite is separably inheritable.
\end{prop}
\begin{proof}
Suppose $B$ is simple and purely infinite, and $I_0\subseteq B$ is a separable $C^*$-subalgebra. We inductively construct an increasing sequence of separable $C^*$-subalgebras $B_k$ ($k\in\mathbb N$) containing $B_0$ and countable dense subsets $\{x_{k,n}\}_n$ of $(B_k)_+$ such that for all $m,n\in\mathbb N$, there exists a $c\in B_{k+1}$ such that 
\begin{equation} \label{eq:sep_simplepurelyinfsepinh}
    c^*x_{k,n}c = x_{k,m}, \quad \|c\|^2 \leq \frac{\|x_{k,m}\|}{\|x_{k,n}\|} + 1.
\end{equation}
To this end, suppose $B_k$ has been obtained for some $k\geq 0$. Let $\{x_{k,n}\}_n$ be a countable dense subset of $(B_k)_+$. For all $m,n\in\mathbb N$ such that $x_{k,n}\neq 0 \neq x_{k,m}$, use the fact that $B$ is simple and purely infinite and \cite[Lemma 4.1.7]{Rordam_ClassificationofNuclearCalgebras} to find $c_{m,n}\in B$ satisfying \eqref{eq:sep_simplepurelyinfsepinh}. Let $C_k$ be the set consisting of all these $c_{m,n}$, and define $B_{k+1} := C^*(I_k\cup C_k)$. Then $B_{k+1}$ is a separable $C^*$-subalgebra of $B$ satisfying the required property.

Now let $\widetilde B := \overline{\bigcup_{k\in\mathbb N} B_k}$. Then $\widetilde B$ is a separable $C^*$-subalgebra of $B$, and we claim that it is simple and purely infinite. To see this, note that it suffices to show $y\precsim x$ for arbitrary non-zero $x,y\in \widetilde B_+$ by \cite[Lemma 2.4.ii]{KirchbergRordam_InfiniteNonSimpleCalgebrasAbsorbingTheCuntzAlgebrasOinfty}. Thus, we pick non-zero $x,y \in \widetilde B_+$ and $\varepsilon >0$, and take $0<\delta<\|x\|$ such that
\begin{equation}
    \left(\frac{\|y\|+\delta}{\|x\|-\delta} + 2\right)\delta < \varepsilon.
\end{equation}
Take $k,m,n\in\mathbb N$ such that $\|x - x_{k,n}\|, \|y - x_{k,m}\|< \delta$. Letting $c \in B_{k+1}$ be the element as in \eqref{eq:sep_simplepurelyinfsepinh} corresponding to these values of $k$, $m$ and $n$, we then compute
\begin{equation}
\begin{aligned}
\|c^*xc - y\|
    &\leq \|c\|^2\|x-x_{k,n}\| + \|c^* x_{k,n} c - y_{k,m} \| + \|y_{k,m}-y\|\\
    &< \left(\frac{\|x_{k,m}\|}{\|x_{k,n}\|} + 1\right)\delta + 0 + \delta\\
    &< \frac{(\|y\|+\delta)\delta}{\|x\|-\delta} + 2\delta < \varepsilon,
\end{aligned}
\end{equation}
showing that $y\precsim x$. This concludes the proof.
\end{proof}

The final technical ingredient we need is the following lemma. The idea is to mirror the previous result for pure largeness; the proof is essentially also a souped-up version of the proof above. Condition e) tells us that in addition, we can arrange both properties simultaneously.\footnote{In the follow-up paper to \cite{CarrionGabeSchafhauserTikuisisWhite_ClassifyingHomomorphismsIUnitalSimpleNuclearCalgebras}, the notion of separably inheritable properties of extensions will be introduced. In this language, the first four conditions of the lemma imply that pure largeness is separably inheritable, while condition e) implies that the ideal being simple and purely infinite is separably inheritable.}

\begin{lemma} \label{lem:sep_purelylargesepinh}
    Suppose
    \begin{equation}
    \begin{tikzcd}
        \mathfrak e: 0 \arrow[r] & I \arrow[r,"\iota"] & E \arrow[r,"\pi"] & D \arrow[r] & 0
    \end{tikzcd}
    \end{equation}
    is a purely large extension. Then for arbitrary separable $C^*$-subalgebras $I_0\subseteq I$, $E_0\subseteq E$ and $D_0\subseteq D$, there exists a separable subextension
    \begin{equation}
    \begin{tikzcd}
        \mathfrak e_s: 0 \arrow[r] & I_s \arrow[r,"\iota|_{I_s}"] & E_s \arrow[r,"\pi|_{E_s}"] & D_s \arrow[r] & 0
    \end{tikzcd}
    \end{equation}
    of $\mathfrak e$ such that
    \begin{enumerate}
        \item $E_0 \subseteq E_s$,
        \item $I_0 \subseteq I_s = E_s\cap I$,
        \item $D_0 \subseteq D_s = C^*(\pi(E_0)\cup D_0)$,
        \item $\mathfrak e_s$ is purely large,
        \item if additionally $I$ is simple and purely infinite, then we can arrange that the same holds for $I_s$.
    \end{enumerate}
    In particular, if $\pi(E_0) = D_0$, then $D_s = D_0$.
\end{lemma}
\begin{proof}
We first arrange conditions a)-d), by inductively constructing a sequence of separable subextensions
\begin{equation}
\begin{tikzcd}
    \mathfrak e_k: 0 \arrow[r] & I_k \arrow[r,"\iota|_{I_k}"] & E_k \arrow[r,"\pi|_{E_k}"] & D_k \arrow[r] & 0
\end{tikzcd}
\end{equation}
of $\mathfrak e$ and countable dense subsets $\{y_{k,m}\}_m$ of $(I_k)_+$ and $\{x_{k,n}\}_n$ of $(E_k)_+$ for $k\in\mathbb N$ such that for all $k\in \mathbb N$, we have $\mathfrak e_k \subseteq \mathfrak e_{k+1}$ and for each $m,n\in\mathbb N$ such that $x_{k,n}\notin I_k$, there exists a $c\in I_{k+1}$ such that
\begin{equation} \label{eq:sep_purelylargesepinh}
    \|c^*x_{k,n}c - y_{k,m}\| < \frac{1}{k+1}, \quad \|c\|^2 \leq \frac{\|y_{k,m}\|}{\|\pi(x_{k,n})\|}.
\end{equation}
Suppose $I_k$, $E_k$ and $D_k$ have been found for some $k\geq 0$. In the case $k=0$, note that we can always enlarge $E_0$ with $I_0$ and a countable set $S$ of elements in $E$ such that $\pi(S)$ is dense in $D_0$, and enlarge $I_0$ by $E_0\cap I$, in order to arrange $I_0 = E_0 \cap I$ and $D_0 \subseteq \pi(E_0)$. These are the only additional assumptions we will use to construct $\mathfrak e_{k+1}$.

Take countable dense subsets $\{y_{k,m}\}_m$ in $(I_k)_+$ and $\{x_{k,n}\}_n$ in $(E_k)_+$, and let $G\subseteq \mathbb N$ be the set of all $n\in \mathbb N$ such that $x_{k,n} \notin I_k$, or equivalently $x_n\notin I$ (as $I_k = E_k\cap I$). Then by Proposition \ref{prop:purelylargenormcontrol}, we know for each $m\in \mathbb N$ and $n\in G$ that there exists an element $c_{m,n}\in I$ such that
\begin{equation}
     \|c_{m,n}^*x_{k,n}c_{m,n} - y_{k,m}\| < \frac{1}{k+1}, \quad \|c_{m,n}\|^2 \leq \frac{\|y_{k,m}\|}{\|\pi(x_{k,n})\|}.
\end{equation}
Then define
\begin{itemize}
    \item $E_{k+1} := C^*\big( E_k \cup \{c_{m,n}\}_{m\in\mathbb N,n\in N}\big) \supseteq E_k$,
    \item $I_{k+1} := E_{k+1}\cap I = E_k \cap I \supseteq I_k$,
    \item $D_{k+1} := \pi(E_{k+1}) = \pi(E_k) \supseteq D_k$.
\end{itemize}
Note that $E_{k+1}$, $I_{k+1}$ and $D_{k+1}$ are also separable, and $\pi(E_{k+1}) = \pi(E_k)$ because we only enlarged $E_k$ by elements of $I$ to obtain $E_{k+1}$. As
\begin{equation}
    \ker \pi|_{E_{k+1}} = \ker \pi \cap E_{k+1} = I \cap E_{k+1} = I_{k+1},
\end{equation}
it follows immediately that we get a short exact sequence
\begin{equation}
    \begin{tikzcd}
    \mathfrak e_{k+1}: 0 \arrow[r] & I_{k+1} \arrow[r,"\iota|_{I_{k+1}}"] & E_k \arrow[r,"\pi|_{E_{k+1}}"] & D_{k+1} \arrow[r] & 0,
    \end{tikzcd}
\end{equation}
which satisfies \eqref{eq:sep_purelylargesepinh} by construction.

Given this sequence of separable subextensions, we now let
\begin{equation}
    \begin{tikzcd}
    \mathfrak e_s: 0 \arrow[r] & I_s \arrow[r,"\iota|_{I_s}"] & E_s \arrow[r,"\pi|_{E_s}"] & D_s \arrow[r] & 0
    \end{tikzcd}
\end{equation}
be the inductive limit $\varinjlim \mathfrak e_n$. This is again a separable subextension of $\mathfrak e$ such that (a), (b) and the inclusion in (c) are satisfied. To show that $D_s = C^*(\pi(E_0)\cup D_0)$, note that in each step of the inductive process aside from $k=0$, we only enlarge $E_k$ with elements of $I$, so $\pi(E_k) = \pi(E_1)$. In the step $k=0$, we also enlarged $E_0$ by a subset $S$ satisfying $\overline{\pi(S)}= D_0$, so indeed
\begin{equation}
    D_s = \overline{\bigcup_{k\in\mathbb N} \pi(E_k)} = \pi(E_1) = \pi\big(C^*(E_0\cup S)\big) = C^*(\pi(E_0)\cup D_0).
\end{equation}
To show that (d) holds, take $x\in (E_s)_+\setminus I_s$, $y\in (I_s)_+$ and $\varepsilon>0$. Note that $\pi(x)\neq 0$ because $x\notin I_s = I\cap E_s$, so there exists $0<\delta < \|\pi(x)\|$ such that
\begin{equation}
    \left(\frac{\|y\|+\delta}{\|\pi(x)\|-\delta} + 2\right)\delta < \varepsilon.
\end{equation}
As $x\notin I_s$ and $I_s$ is closed, there exists $k\in \mathbb N$, $\tilde x \in (E_k)_+\setminus I_s$ and $\tilde y\in I_k$ such that
\begin{equation}
    \frac{1}{k+1}<\delta,\quad \|\tilde x - x\| < \frac{\delta}{2}, \quad \|\tilde y - y\| < \frac{\delta}{2}.
\end{equation}
In particular, we have $\tilde x \notin I_k$, so similarly there exist $n\in N$ and $m\in \mathbb N$ such that
\begin{equation}
    \|\tilde x - x_{k,n}\| < \frac{\delta}{2}, \quad \|\tilde y - y_{k,m}\| < \frac{\delta}{2}.
\end{equation}
By the last property of the sequence of subextensions we constructed, it then follows that there exists a $c\in I_{k+1} \subseteq E_s$ such that
\begin{equation}
\begin{aligned}
\|c^*xc - y\|
    &\leq \|c\|^2\|x-x_{k,n}\| + \|c^* x_{k,n} c - y_{k,m} \| + \|y_{k,m}-y\|\\
    &< \frac{\|y_{k,m}\|}{\|\pi(x_{k,n})\|}\delta + \frac{1}{k+1} + \delta\\
    &< \frac{(\|y\|+\delta)\delta}{\|\pi(x)\|-\delta} + 2\delta < \varepsilon,
\end{aligned}
\end{equation}
as desired.

Now suppose additionally that $I$ is simple and purely infinite. Use the previous part and Proposition \ref{eq:sep_simplepurelyinfsepinh} to inductively construct an increasing sequence of separable subextensions
\begin{equation}
\begin{tikzcd}
    \widetilde{\mathfrak e}_k: 0 \arrow[r] & \widetilde I_k \arrow[r,"\iota|_{\widetilde I_k}"] & \widetilde E_k \arrow[r,"\pi|_{\widetilde E_k}"] & \widetilde D_k \arrow[r] & 0
\end{tikzcd}
\end{equation}
of $\mathfrak e$ such that $\widetilde{\mathfrak e}_k$ is purely large for all odd $k$ and $\widetilde I_k$ is simple and purely infinite for all even $k$. The latter we can arrange as follows. Given $\widetilde{\mathfrak e}_k$, use Proposition \ref{eq:sep_simplepurelyinfsepinh} to find a separable, simple and purely $C^*$-subalgebra $\widetilde I_{k+1}$ of $I$ containing $\widetilde I_k$, and define $\widetilde E_{k+1} := C^*(\widetilde E_k \cup \widetilde I_{k+1})$ and $\widetilde D_{k+1} := \pi(\widetilde E_{k+1}) = \pi(\widetilde E_k)$. We then claim that the resulting inductive limit extension
\begin{equation}
\begin{tikzcd}
    \widetilde{\mathfrak e}_s: 0 \arrow[r] & \widetilde I_s \arrow[r] & \widetilde E_s \arrow[r] & \widetilde D_s \arrow[r] & 0
\end{tikzcd}
\end{equation}
is purely large, and $\widetilde I_s$ is simple and purely infinite. The former follows by applying the same argument from the previous part of the proof to show that $\varinjlim \mathfrak e_n$ is purely large, as all odd-numbered $\widetilde{\mathfrak e}_k$ in particular satisfy the same conditions the $\mathfrak e_k$ do. The latter follows from the fact that
\begin{equation}
    \widetilde I_s = \overline{\bigcup_{k\in\mathbb N} \widetilde I_k} = \overline{\bigcup_{k \text{ even}} \widetilde I_k}
\end{equation}
is purely infinite by Proposition \ref{eq:sep_simplepurelyinfsepinh}. This concludes the proof.
\end{proof}

\subsection{Separabilization of the reduced state-kernel extension}
We are now ready to separabilize the reduced state-kernel extension. The proof follows the same strategy as \cite[Lemma 7.8]{CarrionGabeSchafhauserTikuisisWhite_ClassifyingHomomorphismsIUnitalSimpleNuclearCalgebras}; the main difference lies in how absorption of $\theta$ is obtained. In the reference, one appeals to the so-called corona factorization property to obtain equivalence between absorption and being unitizably full. In this paper, we instead go through the notion of pure largeness, directly making use of the Elliott-Kucerovsky theorem. A major upside of this approach is the fact that we do not need to arrange any $\Z$-stability of the extension algebra, which is required to obtain the corona factorization property. We also obtain stability of the ideal for free from Proposition \ref{prop:elku_purelylargeequiv}.

\begin{lemma}
    Let $A$ be a separable \ca,
    \begin{equation}
    \begin{tikzcd}
        \mathfrak e_S: 0 \arrow[r] & J \arrow[r,"j"] & S \arrow[r,"q"] & Q \arrow[r] & 0
    \end{tikzcd}
    \end{equation}
    a unital extension of \ca s and $\theta: A\to Q$ a nuclear, unitizably full map. Then any unital separable subextension $\mathfrak e_0$ of $\mathfrak e_S$ can be enlarged to a separable unital subextension
    \begin{equation}
    \begin{tikzcd}
        \mathfrak e: 0 \arrow[r] & I \arrow[r,"j_e"] & E \arrow[r,"q_e"] & D \arrow[r] & 0
    \end{tikzcd}
    \end{equation}
    of $\mathfrak e_S$ satisfying all of the following properties.
    \begin{enumerate}
        \item $\theta$ corestricts to a nuclear, unitizably full map $\theta|^D\colon A \to D$. \label{lem:sep_separabilize.A}
        \item Given $\kappa \in KK(A,S)$ such that $KK(A,q)(\kappa) = [\theta]_{KK(A,Q)}$, $\mathfrak{e}$ can be chosen so that $\kappa$ factorizes as $KK(A, \iota_{E\subseteq S})(\kappa')$ with $\kappa' \in KK(A,E)$ and $KK(A, q)(\kappa')=[\theta|^D]_{KK(A, D)}$. \label{lem:sep_separabilize.D}
        \item Given $\psi: A\to S$ and $\kappa \in KK(A,J)$, $\mathfrak{e}$ can be chosen so that $\psi(A)\subseteq E$ and $\kappa$ factorizes as $KK(A, \iota_{I \subseteq J})(\kappa')$ for some $\kappa' \in KK(A,I)$. \label{lem:sep_separabilize.E}
        \item Given two maps $\varphi_1,\varphi_2: A \to S$, $\mathfrak{e}$ can be chosen so that $\varphi_1(A)\cup\varphi_2(A)\subseteq E$. Moreover, if $\varphi_1,\varphi_2: A\rightrightarrows S\rhd J$ is a Cuntz pair such that $[\varphi_1,\varphi_2]_{KL(A,J)}=0$, then $\mathfrak{e}$ can be chosen so that $\big[\varphi_1|^E,\varphi_2|^E\big]_{KL(A, I)} = 0$. \label{lem:sep_separabilize.F}
  \end{enumerate}
  Moreover, any larger subextension of $\mathfrak e$ also satisfies these properties.
\end{lemma}
\begin{proof}
Fix a unital separable subextension
\begin{equation}
\begin{tikzcd}
    \mathfrak e_0: 0 \arrow[r] & I_0 \arrow[r,"j_0"] & E_0 \arrow[r,"q_0"] & D_0 \arrow[r] & 0
\end{tikzcd}
\end{equation}
of $\mathfrak e_S$. Remark that if a separable unital subextension satisfies the above properties, then indeed the same holds for any larger subextension of $\mathfrak e$. Specifically for \ref{lem:sep_separabilize.A}, this follows from the fact that the composition of a full and nuclear \sh\ with a unital \sh\ is again full and nuclear. Hence, it suffices to arrange each of these properties separately\footnote{For \ref{lem:sep_separabilize.D}, we do need the additional property that $\theta$ maps into $D_0$, but \ref{lem:sep_separabilize.A} shows in particular that we can always enlarge $\mathfrak e_0$ to arrange this.}.

\ref{lem:sep_separabilize.A} As $A^\dagger$ is separable and the unitization $\theta^\dagger$ is nuclear by \cite[Proposition 2.2.4]{BrownOzawa_CalgebrasandFiniteDimensionalApproximations}, it follows from \cite[Proposition 1.9]{Schafhauser_SubalgebrasOfSimpleAFAlgebras} that there exists a separable unital $C^*$-subalgebra $D_\theta$ of $Q$ such that $\theta^\dagger(A^\dagger)\subseteq D_\theta$ and $\theta^\dagger|^{D_\theta}$ is full and nuclear. Enlarging $E_0$ by a countable set $N\subseteq E$ such that the resulting \ca\ $\widetilde E := C^*(E_0\cup N)$ satisfies $D_0 \cup D_\theta \subseteq \widetilde D := q(\widetilde E)$, we get a unital separable subextension
\begin{equation}
\begin{tikzcd}
    \widetilde{\mathfrak e}: 0 \arrow[r] & \widetilde I \arrow[r,"\tilde \jmath"] & \widetilde E \arrow[r,"\tilde q"] & \widetilde D \arrow[r] & 0,
\end{tikzcd}
\end{equation}
where $\widetilde I := J\cap \widetilde E$ and $\tilde \jmath$ and $\tilde q$ are the appropriate restrictions of $j$ and $q$. This extension then satisfies condition \ref{lem:sep_separabilize.A}.

For conditions \ref{lem:sep_separabilize.D}, \ref{lem:sep_separabilize.E}, and \ref{lem:sep_separabilize.F}, we use the characterizations of $KK(A, C)$ and $KL(A,C)$ as inductive limits of $KK(A, C_0)$ and $KL(A,C_0)$ over separable subalgebras $C_0$ of a \ca\ $C$. The approach for each of these conditions is similar; we only describe it in detail for the first.

\ref{lem:sep_separabilize.D} Let $\kappa$ be as given. By \ref{lem:sep_separabilize.A}, we can assume without loss of generality that $\theta(A) \subseteq D_0$. By the characterization of $KK(A,S)$ as an inductive limit, we have $\kappa = KK(A, \iota_{C\subseteq S})(\kappa_C)$ for some separable $C^*$-subalgebra $C$ of $S$ and $\kappa_C \in KK(A,C)$. Defining the $C^*$-subalgebras $E_1 := C^*(E_0\cup C) \subseteq S$ and $D_1 := q(E_1)\subseteq Q$, we now have constructed the first three rows in the commutative diagram (dashed arrows commute on the level of $KK$)
\begin{equation}
\begin{tikzcd}
    &&& A \ar[d,"\theta|^{D_0}"] \ar[dl,dashed,"\kappa_C"] \ar[ddddl,dashed,bend right = 70,"\kappa"'] & \\
    && C \ar[d,hook] & D_0 \ar[d,hook] & \\
    && E_1 \ar[d,hook] & D_1 \ar[d,hook] & \\
    && \widetilde E \ar[r,"\tilde q"] \ar[d,hook] & \widetilde D \ar[d,hook] & \\
    0 \arrow[r] & J \arrow[r,"j"] & S \arrow[r,"q"] & Q \arrow[r] & 0
\end{tikzcd}
\end{equation}
where the hooked arrows are the inclusion maps. To complete the diagram, note that
\begin{equation}
\begin{aligned}
    KK(A,\iota_{D_1 \subseteq Q})\circ KK(A,q_1\circ \iota_{C\subseteq E_1})(\kappa_C)
        &= KK(A,q)(\kappa) \\
        &= KK(A,\iota_{D_1\subseteq Q})([\theta|^{D_1}]_{KK(A,D_1)})
\end{aligned}
\end{equation}
in $KK(A,Q)$, where $q_1 := q|_{E_1}$. Using the inductive limit picture of $KK(A,Q)$ (and basic properties of inductive limits), this implies that there exists some separable $C^*$-subalgebra $D_2$ of $Q$ containing $D_1$ such that
\begin{equation}
    KK(A,\iota_{D_1\subseteq D_2})\circ KK(A,q_1\circ \iota_{C\subseteq E_1})(\kappa_C) = KK(A,\iota_{D_1\subseteq D_2})([\theta|^{D_1}]_{KK(A,D_1)}).
\end{equation}
Enlarge $E_1$ by a countable set such that its image under $q$ contains $D_2$ to obtain a separable unital \ca\ $\widetilde E \subseteq S$. Defining $\widetilde D := q(E_2)$ and $\widetilde I := J \cap \widetilde E$, we get a separable unital subextension
\begin{equation}
\begin{tikzcd}
    \widetilde{\mathfrak e}: 0 \arrow[r] & \widetilde I \arrow[r,"\tilde \jmath"] & \widetilde E \arrow[r,"\tilde q"] & \widetilde D \arrow[r] & 0,
\end{tikzcd}
\end{equation}
with $\tilde \jmath$ and $\tilde q$ the appropriate restrictions of $j$ and $q$. Moreover, the element $\tilde \kappa := KK(A,\iota_{C\subseteq \widetilde E})(\kappa_C) \in KK(A,\widetilde E)$ satisfies
\begin{equation}
    KK(A,\iota_{\widetilde E \subseteq S})(\tilde \kappa) = \kappa
\end{equation}
and
\begin{equation}
\begin{aligned}
    KK(A,\tilde q)(\tilde \kappa)
        &= KK(A,\iota_{D_1\subseteq \widetilde D}\circ q_1\circ \iota_{C\subseteq E_1})(\kappa_C) \\
        &= KK(A,\iota_{D_1\subseteq \widetilde D})([\theta|^{D_1}]_{KK(A,D_1)}) \\
        &= [\theta|^{\widetilde D}]_{KK(A,\widetilde D)},
\end{aligned}
\end{equation}
as $\tilde q \circ \iota_{E_1\subseteq\widetilde E} = \iota_{D_1\subseteq \widetilde D} \circ q_1$. Thus, $\widetilde{\mathfrak e}$ and $\tilde \kappa$ satisfy the desired condition.

\ref{lem:sep_separabilize.E} Given $\psi$ and $\kappa$ as in this condition, there exist a separable $C^*$-subalgebra $C\subseteq J$ and $\kappa_C\in KK(A, C)$ such that $\kappa = KK(A, \iota_{C,J})(\kappa_C)$. Hence, we can enlarge $I_0$ by $C$ and $E_0$ by $C\cup\psi(A)$ to obtain a separable unital extension $\widetilde{\mathfrak e}$ satisfying the desired condition, after defining $\tilde \kappa := KK(A, \iota_{C\subseteq \widetilde E})(\kappa_C)$.

\ref{lem:sep_separabilize.F} Given $\varphi_1$ and $\varphi_2$ as in this condition, we can enlarge $E_0$ by $\varphi_1(A)\cup\varphi_2(A)$ (and hence also $I_0$ and $D_0$) to obtain a separable unital extension $\widetilde{\mathfrak e}$. If additionally $\varphi_1$ and $\varphi_2$ lift $\theta$ and
\begin{equation}
    0 = [\varphi_1,\varphi_2]_{KL(A,J)} = KL(A,\iota_{\widetilde E\subseteq J})[\varphi_1|^{\widetilde E},\varphi_2|^{\widetilde E}]_{KL(A,\widetilde E)}.
\end{equation}
From the inductive limit picture of $KK(A,Q)$ (and basic properties of inductive limits), it follows that we can enlarge $\widetilde E$ separably such that $[\varphi_1|^{\widetilde E},\varphi_2|^{\widetilde E}]_{KL(A,\widetilde E)}=0$. The resulting extension $\widetilde{\mathfrak e}$ then satisfies the desired condition.
\end{proof}

\begin{prop}[Separabilization of the reduced state-kernel extension] \label{prop:sepstatekernel}
    Let $A$ be a separable \ca, $B$ a unital \ca\ with a strongly faithful state $\rho$ and $\theta: A\to \pi_\rho(B)''$ a nuclear, unitizably full map.
    \begin{enumerate}
        \item There exists a separable unital subextension
        \begin{equation}
        \begin{tikzcd}
            \mathfrak e: 0 \arrow[r] & I \arrow[r,"j_e"] & E \arrow[r,"q_e"] & D \arrow[r] & 0
        \end{tikzcd}
        \end{equation}
        of the reduced state-kernel extension $\mathfrak e_B$ satisfying properties \ref{lem:sep_separabilize.A} - \ref{lem:sep_separabilize.F} from the previous lemma, such that $I$ is stable and $\tau_e\circ \theta|^D$ is absorbing, with $\tau_e$ the Busby map associated to $\mathfrak e$. \label{prop:sepstatekernel.infty}
        \item The same holds true for the reduced state-kernel extension $\mathfrak e_B^{(\omega)}$ in the ultrapower, and we can additionally arrange that $I$ is simple and purely infinite. \label{prop:sepstatekernel.omega}
    \end{enumerate}
\end{prop}
\begin{proof}
\ref{prop:sepstatekernel.infty} Apply the previous lemma to obtain a separable unital subextension $\mathfrak e_0$
of $\mathfrak e_B$ satisfying the properties in this lemma. Note that $\mathfrak e_B$ is purely large by Proposition \ref{prop:sep_skepurelylarge}, so we can use Lemma \ref{lem:sep_purelylargesepinh} to enlarge $\mathfrak e_0$ to a separable, unital and purely large subextension
\begin{equation}
\begin{tikzcd}
    \mathfrak e: 0 \arrow[r] & I \arrow[r,"j_e"] & E \arrow[r,"q_e"] & D \arrow[r] & 0
\end{tikzcd}
\end{equation}
of $\mathfrak e_B$, which also satisfies conditions \ref{lem:sep_separabilize.A} - \ref{lem:sep_separabilize.F} from the previous lemma. Proposition \ref{prop:sep_pullbackpurelyarge} implies that the pull-back $\mathfrak e \circ \theta^\dagger|^D$ of $\mathfrak e$ by the injective \sh\ $\theta^\dagger$ is purely large. By Proposition \ref{prop:elku_purelylargeequiv}, it follows that the separable \ca\ $I$ is separably stable, so it is in fact stable by \cite[Theorem 2.1]{HjelmborgRordam_OnStabilityofCalgebras}. Now, $A^\dagger$ is also separable, so the Elliott-Kucerovsky theorem (Theorem \ref{thm:elliottkucerovsky}) implies that $\tau_e\circ \theta^\dagger|^D = (\tau_e\circ\theta|^D)^\dagger$ is unitally absorbing. By Proposition \ref{prop:kkeu_absorbingvsunitization}, we can then conclude that $\tau_e\circ\theta|^D$ is absorbing, so $\mathfrak e$ satisfies all the required conditions.

\ref{prop:sepstatekernel.omega} We aim to apply the exact same argument as above, but additionally make use of property e) from Lemma \ref{lem:sep_purelylargesepinh}. To this end, we need to show that $J_{B,\rho}^{(\omega)}$ is simple and purely infinite. This follows from the fact that it is a a hereditary $C^*$-subalgebra of $B_\omega$ by Proposition \ref{prop:ske_statekernelhereditary}, which itself is simple and purely infinite by \cite[Proposition 6.2.6]{Rordam_ClassificationofNuclearCalgebras}. This finishes the proof.
\end{proof}

\section{Classification}
At this point, we have all the main technical ingredients to make the classification machinery work. The classification consists of three major steps. First, we classify lifts of unitizably full maps into the quotient $\pi_\rho(B)''$. The proof is very similar to its stably finite counterpart \cite[Theorem 7.1]{CarrionGabeSchafhauserTikuisisWhite_ClassifyingHomomorphismsIUnitalSimpleNuclearCalgebras}; the main difference lies in the fact that our $KK$-uniqueness result is on the level of the ultrapower, so we will only obtain uniqueness on that level. The second step consists of using the first to obtain a classification result for unital and full maps into $S_{B,\rho}$. Unlike in the tracial setting, obtaining a classification of maps into the middle algebra from the classification of lifts is basically immediate, as the quotient is trivial in $KK$ and does not carry any important (tracial) data that needs to be kept track of. From this viewpoint, the classification of lifts is in fact a stronger result. The main content of the proof is to turn a result in terms of unitizably full maps into a result in terms of unital and full maps (which is done separately in \cite{CarrionGabeSchafhauserTikuisisWhite_ClassifyingHomomorphismsIUnitalSimpleNuclearCalgebras}). This is achieved with the same ``de-unitization'' trick (\cite[Proposition 7.11]{CarrionGabeSchafhauserTikuisisWhite_ClassifyingHomomorphismsIUnitalSimpleNuclearCalgebras}); only the details of the proof are different. In the final step, we prove the desired classification of embeddings $A\hookrightarrow B$ by $\widetilde{KL}(A,B)$, a quotient of $KL$. It is straightforward to see that $\widetilde{KL}$ is equal to $KL$ when $A$ satisfies the UCT, as will be shown in the proof of Theorem \ref{thm:ClassifyingEmbeddingsbyK}. Without this assumption, one can still compare our result with the original Kirchberg-Phillips theorems, and observe that $KL$ and $\widetilde{KL}$ classify the same class of \sh s up to the same notion of equivalence. Hence, they will coincide a posteriori.

\begin{rem}
    One of the core differences between the trace-kernel extension and the reduced state-kernel extension lies in the fact that $S_{B,\rho}$ and $\pi_\rho(B)''$ do not have the sequence algebra properties $B_\infty$ and $B^\infty$ have, in the sense that one cannot perform reindexing arguments in them like \cite[Lemma 3.15, Proposition 9.2]{CarrionGabeSchafhauserTikuisisWhite_ClassifyingHomomorphismsIUnitalSimpleNuclearCalgebras}. In particular, approximate unitary equivalences cannot be reindexed into unitary equivalences, as exemplified by the uniqueness result in Theorem \ref{thm:ClassifyingLifts}. It also contributes to the issue brought up in Remark \ref{rem:whyskipclassifyingunitallifts}. 
\end{rem}

\subsection{Classification of lifts}
The proofs in this section follow the same strategy as those of \cite[Proposition 7.9 \& Theorem 7.1]{CarrionGabeSchafhauserTikuisisWhite_ClassifyingHomomorphismsIUnitalSimpleNuclearCalgebras}. We do not need the analogue of \cite[Theorem 7.1.ii]{CarrionGabeSchafhauserTikuisisWhite_ClassifyingHomomorphismsIUnitalSimpleNuclearCalgebras}, as it is only used to fix the $\Kalg$-class of the \sh\ being constructed in the proof of existence in \cite[Theorem 9.1]{CarrionGabeSchafhauserTikuisisWhite_ClassifyingHomomorphismsIUnitalSimpleNuclearCalgebras}. The only reason we also prove Proposition \ref{prop:KKexistence_forske} is because it is used in the proof of existence in Theorem \ref{thm:ClassifyingLifts}.

\begin{prop}\label{prop:KKexistence_forske}
    Let $A$ be a separable, nuclear \ca, $B$ a unital, simple and purely infinite \ca\ and $\rho$ a strongly faithful state on $B$. Suppose $\theta: A \to \pi_\rho(B)''$ is a unitizably full \sh. If $\psi: A \to S_{B,\rho}$ is a lift of $\theta$ and $\kappa\in KK(A,J_{B,\rho})$, then there exists a \sh\ $\varphi: A \to S_{B,\rho}$ that also lifts $\theta$ and such that $[\varphi,\psi]_{KK(A, J_{B,\rho})}=\kappa$.
\end{prop}
\begin{proof}
Use Proposition \ref{prop:sepstatekernel}.\ref{prop:sepstatekernel.infty} to find a unital separable subextension $\mathfrak{e}$ of the reduced state-kernel extension such that $I$ is stable, $\psi(A)\subseteq E$, $\theta|^D$ is unitizably full, $\tau_e \circ \theta|^D$ is absorbing and $\kappa=KK(A,\iota_{I\subseteq J_{B,\rho}})(\kappa')$ for some $\kappa'\in KK(A,I)$. Hence, the lift $\sigma_e\circ \psi|^E$ of the absorbing \sh\ $\tau_e\circ \theta|^D$, with $\sigma_e: E\rightarrow \Mult I$ the canonical map, is absorbing. Therefore, by $KK$-existence (Theorem \ref{thm:KK-Existence}), $\kappa'$ can be realized as $[\overline{\varphi}, \sigma_e \circ \psi|^E]_{KK(A,I)}$ for some absorbing \sh\ $\overline{\varphi}: A\to \Mult I$ such that $(\overline{\varphi}, \sigma_e \circ \psi|^E)$ forms a Cuntz pair. As such, we have the following commutative diagram:
\begin{equation}
\begin{tikzcd}
    &&&A\arrow[d,"\theta|^D"] \arrow[dl,"\psi|^E"'] \arrow[dd,"\pi_I \circ \overline{\varphi}", bend left=70]\\
    0 \arrow[r] & I \arrow[r,"j_e"]\arrow[d,equal] & E \arrow[r,"q_e", two heads] \arrow[d,"\sigma_e"] & D \arrow[d,"\tau_e"] \\
    0 \arrow[r] & I \arrow[r,"\iota_I"] & \Mult I \arrow[r,"\pi_I", two heads] & \mathcal{Q}(I) 
\end{tikzcd}
\end{equation}
In particular, we have $\pi_I \circ \overline{\varphi} = \tau_e\circ \theta|^D$, so by the universal property of the pull-back (which $E$ satisfies), there exists a \sh\ $\varphi: A\rightarrow E$ satisfying $\overline{\varphi}= \sigma_e\circ \varphi$ and $\theta|^D = q|_E \circ \varphi$. Then, by definition of the $KK$-class of Cuntz pairs not mapping into $\Mult I$, we have
\begin{equation}
  \begin{aligned}
    [\varphi, \psi]_{KK(A,I)} &=
    [\sigma_e \circ \varphi, \sigma_e \circ \psi|^E]_{KK(A,I)} \\
    &= [\overline{\varphi},\sigma_e \circ \psi^E]_{KK(A,I)} \\ &= \kappa'.
  \end{aligned}
\end{equation}
Therefore, viewing $\varphi$ as a map into $S_{B,\rho}$, we can conclude that $[\varphi,\psi]_{KK(A,J_{B,\rho})}=\kappa$.
\end{proof}

\begin{thm}[Classification of lifts of unitizably full maps]\label{thm:ClassifyingLifts}
    Let $A$ be a separable, nuclear \ca, $B$ a unital, simple and purely infinite \ca\ and $\rho$ a strongly faithful state on $B$. Suppose $\theta: A \to \pi_\rho(B)''$ is a unitizably full \sh.
    \begin{enumerate}
        \item Existence: for any $\kappa \in KK(A,S_{B,\rho})$ satisfying $KK(A,q_B)(\kappa) = [\theta]_{KK(A,\pi_\rho(B)'')}$, there exists a \sh\ $\varphi: A \to S_{B,\rho}$ that lifts $\theta$ and such that $[\varphi]_{KK(A,S_{B,\rho})}=\kappa$;
        \label{thm:ClassifyingLifts.C1}
        \item Uniqueness: if $\psi_1, \psi_2: A\to S_{B,\rho}^{(\omega)}$ are two \sh s lifting $\theta$ such that $[\psi_1,\psi_2]_{KL(A,J_{B,\rho}^{(\omega)})} = 0$, then they are approximately unitarily equivalent.\label{thm:ClassifyingLifts.C3}
\end{enumerate}
\end{thm}
\begin{proof}
Existence: use Proposition \ref{prop:sepstatekernel}.\ref{prop:sepstatekernel.infty} to produce a unital separable subextension $\mathfrak{e}\subseteq \mathfrak e_B$ with Busby map $\tau_e$ such that $I$ is stable, $\theta(A)\subseteq D$ and its corestriction to $D$ is unitizably full, $\tau_e\circ \theta|^D$ is absorbing, and $\kappa$ factorizes as $\kappa=KK(A,\iota_{E\subseteq S_{B,\rho}})(\kappa')$ for some $\kappa'\in KK(A,E)$ with $KK(A,q_\mathfrak{e})(\kappa') = [\theta|^D]_{KK(A,D)}$. By use of the commutative diagram
\begin{equation}
\begin{tikzcd}
    &&&A\arrow[d,"\theta|^D"] \arrow[dl,dashed,"\kappa'"'] &\\
    0 \arrow[r] & I \arrow[r,"j_e"]\arrow[d,equal] & E \arrow[r,"q_e"] \arrow[d,"\sigma_e"] & D \arrow[d,"\tau_e"]\arrow[r] &0 \\
    0 \arrow[r] & I \arrow[r,"\iota_I"] & \Mult I \arrow[r,"\pi_I"] & \mathcal{Q}(I) \arrow[r] &0
\end{tikzcd}
\end{equation}
with $\sigma_e$ the canonical map and functoriality of $KK$, we find that
\begin{equation}
    [\tau_e \circ \theta|^D]_{KK(A,\mathcal Q(I))} = KK(A,\pi_I\circ \sigma_e)(\kappa'),
\end{equation}
so $[\tau_e \circ \theta|^D]_{KK(A,\mathcal Q(I))}$ factors through $\Mult{I}$. However, as $I$ is stable, we have $KK( A, \Mult I ) = 0$ by \cite[Proposition 4.1]{Dadarlat_ApproximateUnitaryEquivalenceAndTheTopologyOfExt}, so $[\tau_e \circ \theta|^D]_{KK(A,\mathcal Q(I))}=0$. As $\tau_e \circ \theta|^D: A\rightarrow \mathcal{Q}(I)$ is absorbing, it follows from the $KK$-existence theorem (Theorem \ref{thm:KK-Existence}.\ref{thm:KK-Existence.C2}) that it lifts to a \sh\ $\overline{\psi} \colon A \to \Mult I$. By definition of the pull-back, the maps $\bar\psi$ and $\theta|^D$ then induce a \sh\ $\psi\colon A \to E$. Regarding $\psi$ as taking values in $S_{B,\rho}$, it follows that $\psi$ lifts $\theta$. There is however no reason to expect that the $KK$-class of $\psi$ is equal to $\kappa$, so this must be corrected. By construction, we have in $KK(A, \pi_\rho(B)'')$
\begin{equation}
\begin{aligned}
     KK(A,q_B)(\kappa)
        &= [\theta]_{KK(A, \pi_\rho(B)'')} \\
        &= KK(A,q_{B})([\psi]_{KK(A,S_{B,\rho})}),
\end{aligned}
\end{equation}
using functoriality of $KK$. This implies that
\begin{equation}
    \kappa- [\psi]_{KK(A,S_{B,\rho})} \in \ker KK(A, q_B) = \im KK(A,j_B)
\end{equation}
by half-exactness of $KK(A, \,\cdot\,)$ when $A$ is nuclear \cite[Proposition 5.4.v]{CarrionGabeSchafhauserTikuisisWhite_ClassifyingHomomorphismsIUnitalSimpleNuclearCalgebras}. In other words, there exists a $\kappa'\in KK(A,J_{B,\rho})$ with
\begin{equation}
    KK(A, j_B)(\kappa')=\kappa-[\psi]_{KK(A,S_{B,\rho})}.
\end{equation}
Then by Proposition \ref{prop:KKexistence_forske}, there exists $\varphi: A \to S_{B,\rho}$ also lifting $\theta$ such that $[\varphi,\psi]_{KK(A,J_{B,\rho})}=\kappa'$. It follows that
\begin{equation}
      [\varphi]_{KK(A,S_{B,\rho})} - [\psi]_{KK(A,S_{B,\rho})} = KK(A,j_B)(\kappa') = \kappa - [\psi]_{KK(A,S_{B,\rho})},
\end{equation}
so indeed $[\varphi]_{KK(A,S_{B,\rho})} = \kappa$.

Uniqueness: use Proposition \ref{prop:sepstatekernel}.\ref{prop:sepstatekernel.omega} to obtain a unital separable subextension $\mathfrak{e}$ of $\mathfrak e_B^{(\omega)}$ such that $I$ is stable, simple and purely infinite, $\theta(A)\subseteq D$, $\theta|^D$ is unitizably full and its composition with the Busby map $\tau_e$ of $\mathfrak e$ is absorbing, $\psi_1(A)\cup\psi_2(A) \subseteq E$ and $\big[\psi_1|^E,\psi_2|^E\big]_{KL(A, I)} =0$. Then the lifts $\sigma_e \circ \psi_i|^E$ ($i\in\{1,2\}$) of $\tau_e\circ \theta|^D$ are again absorbing, where $\sigma_e: E\to \Mult I$ is the canonical map. Moreover, we have by definition of $KL$-classes of Cuntz pairs with general codomains
\begin{equation}
    [\sigma_e \circ \psi_1|^E, \sigma_e \circ \psi_2|^E]_{KL(A,I)} = [\psi_1|^E,\psi_2|^E]_{KL(A, I)} = 0.
\end{equation}
As $I$ is simple and purely infinite, the $KL$-uniqueness theorem (Theorem \ref{thm:KK-Uniqueness}) gives a sequence $(u_n)_{n=1}^\infty$ of unitaries in $I^\dagger$ such that
\begin{equation}
    \|\Ad(u_n)\circ\sigma_e \circ\psi_1(a) - \sigma_e \circ\psi_2(a)\| \rightarrow 0
\end{equation}
for all $a\in A$. Now $E$, being unital, contains a copy of $I^\dagger$, and by uniqueness of the map $\sigma_a$, it restricts to the identity map on $I^\dagger$. Hence, $\Ad(u_n)$ and $\sigma_e$ commute, so we can rewrite the previous statement as
\begin{equation}
    \|\sigma_e \circ \big( \Ad(u_n)\circ \psi_1(a) - \psi_2(a)\big)\| \rightarrow 0.
\end{equation}
As $\sigma_e$ is injective by Proposition \ref{prop:sep_purelyinfimpliespurelylarge}, it then follows that
\begin{equation}
    \|\Ad(u_n)\circ\psi_1(a) - \psi_2(a)\| \rightarrow 0
\end{equation}
for all $a\in A$, i.e.\ $\psi_1$ and $\psi_2$ are approximately unitarily equivalent as maps into $E\subseteq S_{B,\rho}^{(\omega)}$.
\end{proof}

\subsection{Classification of maps into $S_{B,\rho}$}
The goal of this section is to mold the classification of lifts we just obtained into a classification of unital maps into $S_{B,\rho}$:

\begin{thm}[Classification of unital full maps into $S_{B,\rho}$] \label{thm:ClassifyingApproximateMaps}
    Let $A$ be a unital, separable, nuclear \ca, $B$ a unital, simple and purely infinite \ca\ and $\rho$ a strongly faithful state on $B$ such that $\pi_\rho(B)'' \cong B(\Hb)$ for some separable Hilbert space $\Hb$.
    \begin{enumerate}
        \item Existence: for any $\kappa \in KK(A,S_{B,\rho})$ such that $\Gamma_0^{(A,S_{B,\rho})}(\kappa)[1_A]_0 = [1_{S_{B,\rho}}]_0$, there exists a unital, full \sh\ $\varphi: A\to S_{B,\rho}$ inducing $\kappa$,
        \item Uniqueness: if two unital, full \sh s $\varphi_1,\varphi_2: A\to S_{B,\rho}$ satisfy $[\pi_{\infty\to\omega}\circ\varphi_1]_{KL(A,S_{B,\rho}^{(\omega)})} = [\pi_{\infty\to\omega}\circ\varphi_2]_{KL(A,S_{B,\rho}^{(\omega)})}$, then they are unitarily equivalent as maps into $B_\omega$.
    \end{enumerate}
\end{thm}

Note that we do not obtain a complete classification of maps into $S_{B,\rho}$ by $KL(A, S_{B,\rho})$, unlike in the stably finite classification; we only get uniqueness up to unitary equivalence by a unitary in $B_\omega$.

\begin{rem}
    Up until Theorem \ref{thm:ClassifyingApproximateMaps}, we had not yet imposed any conditions on $\pi_\rho(B)''$. We now need to do so in order to gain access to a uniqueness theorem for maps into $\pi_\rho(B)''$. By Corollary \ref{cor:stronglyfaithfulstateexists}, we know that we can always arrange $\pi_\rho(B)''\cong B(\Hb)$ in our setting, for which Voiculescu's theorem gives uniqueness of unital and full \sh s up to approximate unitary equivalence. In other words, we do not need the very deep and technical classification of type-III factors to classify their $C^*$-algebraic analogue; the uniqueness result for maps into type-I factors suffices.
\end{rem}

\begin{rem} \label{rem:whyskipclassifyingunitallifts}
    Theorem \ref{thm:ClassifyingApproximateMaps} combines \cite[Theorem 7.10]{CarrionGabeSchafhauserTikuisisWhite_ClassifyingHomomorphismsIUnitalSimpleNuclearCalgebras} and \cite[Theorem 9.1]{CarrionGabeSchafhauserTikuisisWhite_ClassifyingHomomorphismsIUnitalSimpleNuclearCalgebras} into a single step. We do so because our setting comes with some additional subtleties which complicate matters. Our uniqueness result for maps into the quotient is Voiculescu's theorem, i.e.\ only up to approximate unitary equivalence instead of unitary equivalence (as in \cite[Theorem 6.5]{CarrionGabeSchafhauserTikuisisWhite_ClassifyingHomomorphismsIUnitalSimpleNuclearCalgebras}). As one cannot reindex in $S_{B,\rho}$ or $B(\Hb)$, this entails that in the proof of existence, we can no longer conjugate the map $\varphi:A\to S_{B,\rho}$ realizing a given $\kappa\in KK(A,S_{B,\rho})$ by a single unitary to make it lift a given $\theta: A\to \pi_\rho(B)''$. However, we can circumvent that issue by going directly to a classification of maps into $S_{B,\rho}$, because there is no tracial data in the quotient which could obstruct a given $\kappa\in KK(A,S_{B,\rho})$ getting realized by a \sh . Therefore, we can simply pick a map $\theta: A\to \pi_\rho(B)''$ to apply the classification of lifts to and disregard the fact that the resulting $\varphi$ no longer lifts $\theta$.
\end{rem}

In order to prove Theorem \ref{thm:ClassifyingApproximateMaps}, we will need a couple of additional ingredients. First, we need a way to apply our classification of lifts in the context of unital and full maps, even though such maps will never be unitizably full. This is done by use of the following ``de-unitization'' trick from \cite[\S 7.4]{CarrionGabeSchafhauserTikuisisWhite_ClassifyingHomomorphismsIUnitalSimpleNuclearCalgebras}, as mentioned in the beginning of this section. The proof of the ``only if''-statement is the same as \cite[Proposition 7.10]{CarrionGabeSchafhauserTikuisisWhite_ClassifyingHomomorphismsIUnitalSimpleNuclearCalgebras}; we will need the other direction later on.

\begin{lemma} \label{lem:deunitizationtrick}
    Let $\varphi: A\to D$ be a unital \sh\ between unital \ca s $A$ and $D$. Then $\varphi$ is full if and only if the map $\iota^{(k)}_D\circ \varphi$ is unitizably full for some (or all) $2\leq k\in\mathbb N$, where $\iota^{(k)}_D: D\hookrightarrow M_k(D)$ is the top-left corner embedding.
\end{lemma}
\begin{proof}
First suppose $\varphi$ is full, and pick $k\geq 2$. Any $x\in A^\dagger$ can be written as $x = a + \lambda(1_{A^\dagger} - 1_A)$ for some $a\in A$ and $\lambda \in \mathbb C$, and as such
\begin{equation}
    \big(\iota^{(k)}_D\circ\varphi\big)^\dagger(x) = \theta(a)\oplus \lambda 1_{M_{k-1}(D)}.
\end{equation}
Whenever $a\neq 0$ or $\lambda\neq 0$, the above element generates $M_k(D)$ as an ideal since $\varphi$ is full.

Now suppose $\iota^{(k)}_D\circ \varphi$ is unitizably full for some $k\geq 2$. Applying this statement for an arbitrary non-zero $a\in A \subseteq A^\dagger$, we find
\begin{equation}
    \iota^{(k)}_D(1_D) \in \overline{M_k(D) \iota^{(k)}_D(\varphi(a)) M_k(D)},
\end{equation}
or, after multiplying from both sides by $\iota^{(k)}_D(1_D)$,
\begin{equation}
    \iota^{(k)}_D(1_D) \in \overline{\iota^{(k)}_D(1_D) M_k(D) \iota^{(k)}_D(\varphi(a)) M_k(D)\iota^{(k)}_D(1_D)} = \overline{\iota^{(k)}_D(D\varphi(a)D)} = \iota^{(k)}_D\left(\overline{D\varphi(a)D}\right),
\end{equation}
so $1_D \in \overline{D\varphi(a)D}$.
\end{proof}

\begin{rem} \label{rem:matrixamplifiedske}
    For a \ca\ $D$ and $k\in\mathbb N$, let $\iota_D^{(k)}: D \to M_k(D)$ be the top-left corner embedding. and let $\tr$ be the canonical normalized trace on $M_k(\mathbb C)$. Given $k\in\mathbb N$ and a unital \ca\ $B$ with a strongly faithful state $\rho$, note that $\tr\circ\rho$ (where $\rho$ is applied pointwise) is a faithful state on $M_k(B)$, as for any $x := (x_{ij})_{i,j}\in M_k(B)$, we have
    \begin{equation}
        \tr\circ\rho(x^*x) = \sum_{i,j} \rho(x_{ij}^*x_{ij}).
    \end{equation}
    In fact, we claim that it is strongly faithful. To see this, note that $M_k(\pi_\rho(B)'')$ is a von Neumann algebra with a faithful normal state $\tr\circ\rho''$, containing $M_k(B)$ as a strongly dense $C^*$-subalgebra and $(\tr\circ\rho'')|_{M_k(B)} = \tr\circ\rho$. Therefore, \ref{lem:isoGNSrep} implies that there exists an isomorphism
    \begin{equation}
        \alpha: \pi_{\tr\circ\rho}(M_k(B))'' \to M_k(\pi_\rho(B)'')
    \end{equation}
    satisfying $(\tr\circ\rho)'' = \tr\circ\rho''\circ\alpha$. In particular, $(\tr\circ\rho)''$ is faithful, proving our claim. Hence, we can form the reduced state-kernel extension $\mathfrak e_{M_k(B)}$ with respect to this state, and obtain a commutative diagram
    \begin{equation}\label{eq:matrixamplifiedske}
    \begin{tikzcd}[row sep = large]
        0 \ar[r] & J_{B,\rho} \ar[r, "{j_B}"] \ar[d, "{\iota^{(k)}_{J_{B,\rho}}}"] & S_{B,\rho} \ar[d, "{\iota^{(k)}_{S_{B,\rho}}}"] \ar[r, "{q_B}"] & \pi_\rho(B)'' \ar[d, "{\pi_\rho(B)''}"] \ar[r] & 0\\
        0 \ar[r] & M_k(J_{B,\rho}) \ar[r] \ar[d, "{\cong}"] & M_k(S_{B,\rho}) \ar[d, "\cong"] \ar[r] & M_k(\pi_\rho(B)'') \ar[d, "\alpha^{-1}"] \ar[r] & 0\\
        0 \ar[r] & J_{M_k(B),\tr\circ\rho} \ar[r, "{j_{M_k(B)}}"] & S_{M_k(B),\tr\circ\rho} \ar[r, "{q_{M_k(B)}}"] & \pi_{\tr\circ\rho}(M_k(B))'' \ar[r] & 0
    \end{tikzcd}
    \end{equation}
    where the isomorphisms are the obvious maps.
\end{rem}

The final two lemmas are needed for the proof of uniqueness. The first tells us that one can conjugate any two unital full maps into $S_{B,\rho}$ by a unitary in $B_\infty$ such that they form a Cuntz pair. The second will imply that if these maps have the same $KL$-class, then the corresponding class in $KL(A,J_{B,\rho}^{(\omega)})$ will be zero. This means that (after the de-unitization trick) one can immediately appeal to Theorem \ref{thm:ClassifyingLifts}.\ref{thm:ClassifyingLifts.C3}.

\begin{lemma} \label{lem:fullcuntzpairuptounitary}
    Let $A$ be a separable unital \ca\ and $B$ be a unital \ca\ with a strongly faithful state $\rho$ such that $\pi_\rho(B)'' \cong B(\Hb)$. For any two u.c.p.\ maps $\psi_1,\psi_2: A\to S_{B,\rho}$ such that $q_B\circ \psi_1$ and $q_B\circ \psi_2$ are full \sh s, there exists a unitary $v\in B_\infty$ such that $\Ad(v)\circ\psi_2 - \psi_1 \in J_{B,\rho}$ and $\Ad(v)\circ\psi_2$ maps into $S_{B,\rho}$.
\end{lemma}
\begin{proof}
First note that we can construct the following commutative diagram:
\begin{equation}
\begin{tikzcd}[row sep=large, column sep=large]
    & \mathcal{S}_{B,\rho} \ar[r,hook] \ar[d,"\pi_\infty"] & \ell^\infty(B) \ar[d,"\pi_\infty"]\\
    A \ar[ru,"(\psi_{i,k})_k",dashed] \ar[r,"\psi_i"] \ar[rd,"q_B\circ\psi_i"'] & S_{B,\rho} \ar[r,hook] \ar[d,"q_B"] & B_\infty\\
    & \pi_\rho(B)'' &
\end{tikzcd}
\end{equation}
for $i\in\{1,2\}$. Here $\mathcal{S}_{B,\rho}$ is the analogue of $S_{B,\rho}$ in $\ell^\infty(B)$, the hooked arrows are inclusion maps, and the $(\psi_{i,k})_k$ are set-theoretic lifts of the $\psi_i$ to $\mathcal{S}_{B,\rho}$.

As the maps $q_B \circ \psi_1$ and $q_B\circ \psi_2$ are full \sh s and $\pi_\rho(B)''\cong B(\Hb)$, it follows from Voiculescu's theorem (see e.g.\ \cite[Ch.1, Theorem 7.6]{BrownOzawa_CalgebrasandFiniteDimensionalApproximations}) that $q_B \circ \psi_1$ and $q_B\circ \psi_2$ are approximately unitarily equivalent. Using \eqref{eq:ske_normdominatesphinorm}, this implies in particular that there exists a sequence of unitaries $(v^{(n)})_n$ in $\pi_\rho(B)''$ such that
\begin{equation}
    \|\Ad(v^{(n)})\circ q_B\circ\psi_2(a) - q_B \circ \psi_1(a)\|_{\rho'',\#} \xrightarrow{n\rightarrow\infty} 0
\end{equation}
for all $a\in A$. Since $B(\Hb)$ is a von Neumann algebra, all its unitaries are exponentials, so in particular we can lift each $v^{(n)}$ to a unitary sequence $(v^{(n)}_k)_k$ in $\mathcal{S}_{B,\rho}$. By Corollary \ref{cor:ske_maptoquotient}, it follows that
\begin{equation}
\begin{aligned}
    \|\Ad(v^{(n)})&\circ q_B\circ\psi_2(a) - q_B \circ \psi_1(a)\|_{\rho'',\#} \\
        &= \|q_B\circ\pi_\infty\big(\Ad(v^{(n)}_k)\circ \psi_{2,k}(a) - \psi_{1,k}(a)\big)_k\|_{\rho'',\#} \\
        &= \lim_{k\to \infty} \|\Ad(v^{(n)}_k)\circ \psi_{2,k}(a) - \psi_{1,k}(a)\|_{\rho,\#}
\end{aligned}
\end{equation}
for all $a\in A$. Since this expression tends to zero for $n\to\infty$ and $A$ is separable, we can then apply the $\varepsilon$-test to obtain a unitary sequence $(v_k)_k \in \ell^\infty(B)$ such that
\begin{equation}
    \lim_{k\rightarrow\infty} \|\Ad(v_k)\circ \psi_{2,k}(a_j) - \psi_{1,k}(a_j)\|_{\rho,\#} = 0
\end{equation}
for all $j\in \mathbb N$, with $\{a_j\}_{j\in\mathbb N}$ a dense sequence in $A$. By noting that
\begin{equation}
    \|\Ad(v_k)\circ \psi_2(a-a_j)\|_{\rho,\#} \leq \|\Ad(v_k)\circ \psi_2(a-a_j)\| = \|a-a_j\|
\end{equation}
and similarly $\|\psi_1(a-a_j)\|_{\rho,\#} \leq \|a-a_j\|$
for all $a\in A$ and $j,k\in \mathbb N$, it follows that
\begin{equation}
    \lim_{k\rightarrow\infty} \|\Ad(v_k)\circ \psi_{2,k}(a) - \psi_{1,k}(a)\|_{\rho,\#} = 0
\end{equation}
for all $a\in A$. Defining $v := \pi_\infty((v_k)_k) \in B_\infty$, we then obtain for all $a\in A$
\begin{equation}
    \Ad(v)\circ \psi_2(a) - \psi_1(a) = \pi_\infty\big(\Ad(v_k)\circ \psi_{2,k}(a) - \psi_{1,k}(a)\big)_k \in J_{B,\rho},
\end{equation}
and also
\begin{equation}
    \Ad(v)\circ \psi_2(a) = \psi_1(a) + (\Ad(v)\circ \psi_2(a) - \psi_1(a)) \in S_{B,\rho} + J_{B,\rho} \subseteq S_{B,\rho}.
\end{equation}
This concludes the proof.
\end{proof}

\begin{rem}
    Observe that we do not actually need the full strength of Voiculescu's theorem to make the above argument. We only need $\pi_\rho(B)''$ to be a von Neumann algebra with a uniqueness theorem for full maps up to $(\rho,\#)$-approximate unitary equivalence, or equivalently (by Proposition \ref{prop:ske_strongequivphinorm}), approximate unitary equivalence in the strong-$*$ topology.
\end{rem}

\begin{lemma} \label{lem:hereditaryisoinKK}
    Let $A$ and $D$ be \ca s with $A$ separable. For any full, hereditary $C^*$-subalgebra $I$ of $D$, the inclusion map $\iota: I\hookrightarrow D$ induces an isomorphism  $KK(A,\iota): KK(A,I)\to KK(A,D)$, and therefore also an isomorphism on the level of $KL$. In particular, this holds for the inclusion $J_{B,\rho}^{(\omega)} \subseteq B_\omega$ when $B$ is a simple and purely infinite \ca\ and $\rho$ a state on $B$.
\end{lemma}
\begin{proof}
The first part of the statement is \cite[Proposition 12.24]{Gabe_ClassificationofOinftystableCalgebras} for $X=\{*\}$, in the case where $I$ and $D$ are $\sigma$-unital. The general case follows by using the direct limit definition of $KK$ in the non-separable setting. The particular consequence follows as $J_{B,\rho}^{(\omega)}$ is hereditary in $B_\omega$ by Proposition \ref{prop:skeinomega} and $B_\omega$ is simple by \cite[Proposition 6.2.6]{Rordam_ClassificationofNuclearCalgebras}.
\end{proof}

We are now ready to prove Theorem \ref{thm:ClassifyingApproximateMaps}.

\begin{proof}[Proof of Theorem \ref{thm:ClassifyingApproximateMaps}]
The general strategy of the proof is to apply classification of lifts to $M_2(B)$ and $\tr\circ\rho$ instead of $B$ and $\rho$. Note that the state $\tr\circ\rho$ on $M_2(B)$ is strongly faithful by Remark \ref{rem:matrixamplifiedske}, and that
\begin{equation} \label{eq:classifyingapproximatemaps}
\begin{aligned}
    KK(A,\pi_{\tr\circ\rho}(M_2(B))'')
        &\cong KK(A,M_2(\pi_\rho(B)'')) \\
        &\cong KK(A,\pi_\rho(B)'') \\
        &\cong KK(A,B(\Hb)) = 0,
\end{aligned}
\end{equation}
by the rightmost isomorphism in \eqref{eq:matrixamplifiedske}, Proposition \ref{prop:KKprops}.\ref{prop:KKprops.stable} and \cite[Proposition 4.1]{Dadarlat_ApproximateUnitaryEquivalenceAndTheTopologyOfExt}. Hence, $M_2(B)$ (which is simple and purely infinite by \cite[Proposition 4.1.8.i]{Rordam_ClassificationofNuclearCalgebras}) and $\tr\circ\rho$ do indeed satisfy the required conditions to apply classification of lifts to. We will also use commutativity of \eqref{eq:matrixamplifiedske} throughout the proof.

Existence: pick $\kappa$ as in the statement, and fix\footnote{Remark that such a $\theta$ always exists. Indeed, letting $\tilde\theta: A\to B(\Hb)$ be a unital and full \sh, the map $\theta := \iota^{(2)}_{B(\Hb)}\circ\tilde\theta: A\to M_2(B(\Hb)) \cong M_2(\pi_\rho(B)'')$ is unitizably full by Lemma \ref{lem:deunitizationtrick}.} a unitizably full \sh\ $\theta: A\to M_2(\pi_{\rho}(B)'')$. Note that $KK(A,\iota^{(2)}_{S_{B,\rho}})(\kappa)$ is automatically a $KK$-lift of $\theta$ by \eqref{eq:classifyingapproximatemaps}, so we can apply classification of lifts to this element to find a unitizably full map $\varphi^{(2)}: A\to S_{M_2(B),\tr\circ\rho}\cong M_2(S_{B,\rho})$ lifting $\theta$ such that
\begin{equation}
    [\varphi^{(2)}]_{KK(A,M_2(S_{B,\rho}))} = KK(A,\iota^{(2)}_{S_{B,\rho}})(\kappa).
\end{equation}
By naturality of $\Gamma_0^{(A,\cdot)}$, it then follows that
\begin{equation}
    [\varphi^{(2)}(1_A)]_0 = [\iota^{(2)}_{S_{B,\rho}}(1_{S_{B,\rho}})]_0.
\end{equation}
As $q_{M_2(B)}\circ\varphi^{(2)}(1_A) = \iota^{(2)}_{B(\Hb)}(1_{B(\Hb)}) = q_{M_2(B)}\circ\iota^{(2)}_{S_{B,\rho}}(1_{S_{B,\rho}})$ is full and properly infinite, it follows from Lemma \ref{lem:propertiesliftingfromquotient} that both $\varphi^{(2)}(1_A)$ and $\iota^{(2)}_{S_{B,\rho}}(1_{S_{B,\rho}})$ are full and properly infinite. By \cite[Proposition 4.1.4]{Rordam_ClassificationofNuclearCalgebras}, we then have $\varphi^{(2)}(1_A)\sim\iota^{(2)}_{S_{B,\rho}}(1_{S_{B,\rho}})$, so $\varphi^{(2)}(1_A)\oplus 0 \sim_u\iota^{(4)}_{S_{B,\rho}}(1_{S_{B,\rho}})$. Hence, $\varphi^{(2)}\oplus 0: A\to M_4(S_{B,\rho})$ is unitarily equivalent to $\iota^{(4)}\circ\tilde\varphi$ for some unital \sh\ $\tilde \varphi:A\to S_{B,\rho}$. We now claim that $\varphi$ is the desired \sh. To see this, first note that it is full by Lemma \ref{lem:deunitizationtrick}, as $\iota^{(4)}\circ\varphi\sim_u \varphi^{(2)}\oplus 0$ is unitizably full. To show that $\varphi$ realizes the class of $\kappa$ in $KK$, we compute
\begin{equation}
\begin{aligned}
    KK(A,\iota^{(4)}_{S_{B,\rho}})[\varphi]_{KK(A,S_{B,\rho})}
        &= KK(A,\iota^{(2)}_{M_2(S_{B,\rho})})[\varphi^{(2)}]_{KK(A,M_2(S_{B,\rho}))} \\
        &= KK(A,\iota^{(4)}_{S_{B,\rho}})(\kappa),
\end{aligned}
\end{equation}
using functoriality and Proposition \ref{prop:KKprops}.\ref{prop:KKprops.unitary}. As $KK(A,\iota^{(4)}_{S_{B,\rho}})$ is an isomorphism by Proposition \ref{prop:KKprops}.\ref{prop:KKprops.stable}, this implies that $[\varphi]_{KK(A,S_{B,\rho})} = \kappa$, as desired.

Uniqueness: suppose we are given $\varphi_1$ and $\varphi_2$ as in the statement. By Lemma \ref{lem:fullcuntzpairuptounitary}, there exists a unitary $v\in B_\infty$ such that $\varphi_1, \Ad(v)\circ \varphi_2: A\rightrightarrows S_{B,\rho}\rhd J_{B,\rho}$ forms a Cuntz pair. It then follows that
\begin{equation}
\begin{aligned}
    KL(A,&\iota_{J_{B,\rho}^{(\omega)}\subseteq B_\omega})[\pi_{\infty\to\omega}\circ\varphi_1, \pi_{\infty\to\omega}\circ\Ad(v)\circ\varphi_2]_{KL(A,J_{B,\rho}^{(\omega)})} \\
        &= [\pi_{\infty\to\omega}\circ\varphi_1]_{KL(A,B_\omega)} - [\Ad(\pi_{\infty\to\omega}(v))\circ \pi_{\infty\to\omega}\circ\varphi_2]_{KL(A,B_\omega)} \\
        &= KL(A,\iota_{S_{B,\rho}^{(\omega)}\subseteq B_\omega})\left( [\pi_{\infty\to\omega}\circ\varphi_1]_{KL(A,S_{B,\rho}^{(\omega)})} - [\pi_{\infty\to\omega}\circ\varphi_2]_{KL(A,S_{B,\rho}^{(\omega)})} \right) \\
        &= 0,
\end{aligned}
\end{equation}
by functoriality and Proposition \ref{prop:KKprops}.\ref{prop:KKprops.unitary}. As $KL(A,\iota_{J_{B,\rho}^{(\omega)}\subseteq B_\omega})$ is in particular injective by Lemma \ref{lem:hereditaryisoinKK}, this implies that $[\pi_{\infty\to\omega}\circ\varphi_1, \pi_{\infty\to\omega}\circ\Ad(v)\circ\varphi_2]_{KL(A,J_{B,\rho}^{(\omega)})} = 0$. Hence, the \sh s
\begin{equation}
\begin{aligned}
    \psi_1 &:= \iota^{(2)}_{S_{B,\rho}^{(\omega)}}\circ \pi_{\infty\to\omega}\circ \varphi_1, \\
    \psi_2 &:= \iota^{(2)}_{S_{B,\rho}^{(\omega)}}\circ \pi_{\infty\to\omega}\circ \Ad(v)\circ\varphi_2
\end{aligned}
\end{equation}
form a Cuntz pair and satisfy $[\psi_1, \psi_2]_{KL(A,J_{M_2(B),\tr\circ\rho}^{(\omega)})} = 0$. Classification of lifts (with $\theta := q_{M_2(B)}^{(\omega)} \circ \psi_1 = q_{M_2(B)}^{(\omega)}\circ \psi_2$) then implies that $\psi_1$ and $\psi_2$ are approximately unitarily equivalent. Viewing them as maps into $(M_2(B))_\omega \cong M_2(B_\omega)$ and recalling that $A$ is separable, we can apply the $\varepsilon$-test to find that they are in fact unitarily equivalent by a unitary $u\in M_2(B_\omega)$. In other words, we have
\begin{equation}
    u \begin{pmatrix} \pi_{\infty\to\omega}\circ\varphi_1(a) &0\\ 0&0 \end{pmatrix} u^* = \begin{pmatrix} \pi_{\infty\to\omega}\circ\Ad(v)\circ\varphi_2(a) &0\\ 0&0 \end{pmatrix}.
\end{equation}
By unitality of $\varphi_1$ and $\varphi_2$, this unitary must be of the form $u = u_1\oplus u_2$ for some unitaries $u_1,u_2\in B_\omega$, so $\pi_{\infty\to\omega}\circ\varphi_1 \sim_u \pi_{\infty\to\omega}\circ\varphi_1$ in $B_\omega$, as desired.
\end{proof}

\subsection{Classification of embeddings}
The following lemma and subsequent theorem are an adaptation of \cite[Lemma 4.2 \& Theorem 4.3]{Gabe_ANewProofOfKirchbergsO2StableClassification} to the ultrapower setting, with very similar proofs.
\begin{lemma}
    Let $A$ be a separable and $B$ a unital \ca. Suppose $\rho: A\to B_\infty$ is a \sh\ with the following property: for any map $\eta: \mathbb N \to \mathbb N$ for which $\lim_{n\to\infty} \eta(n) = \infty$, the maps $\pi_{\infty\to\omega}\circ \rho$ and $\pi_{\infty\to\omega}\circ \eta^* \circ \rho$ are approximately unitarily equivalent (in $B_\omega$). Then for any set-theoretic lift $(\rho_n)_n: A\to \ell^\infty(B)$, finite $F\subseteq A$, $\varepsilon> 0$ and $m\in\mathbb N$, there exists $k\geq m$ such that for all $n\geq k$, there exists a unitary $u\in B$ satisfying
    \begin{equation}
        \|u\rho_n(a)u^* - \rho_k(a)\| < \varepsilon
    \end{equation}
    for all $a\in F$.
\end{lemma}
\begin{proof}
We will argue by contradiction and suppose the statement does not hold. Then there exist a lift $(\rho_n)_n$ of $\rho$, a finite set $F\subseteq A$, $\varepsilon > 0$ and $m\in \mathbb N$ such that for all $k\geq m$, there exists $n_k\geq k$ such that
\begin{equation} \label{eq:onesidedintertwining}
    \max_{a\in F} \|v\rho_{n_k}(a)v^* - \rho_k(a)\| \geq \varepsilon
\end{equation}
for any unitary $v\in B$. Define the map $\eta:\mathbb N \to \mathbb N$ by $\eta(k) = 1$ for $k < m$ and $\eta(k) = n_k$ for $k\geq m$. As $n_k \geq k$ for all $k\in\mathbb N$, it follows that $\lim_{k\to\infty} \eta(k) = \infty$. By assumption, the maps $\pi_{\infty\to\omega}\circ \rho$ and $\pi_{\infty\to\omega}\circ \eta^* \circ \rho$ are then approximately unitarily equivalent, so there exists a unitary $u\in B_\omega$ such that
\begin{equation}
    \|\Ad(u)\circ \pi_{\infty\to\omega}\circ \eta^* \circ \rho(a) - \pi_{\infty\to\omega}\circ \rho(a)\| < \varepsilon
\end{equation}
for all $a\in F$. Moreover, we have
\begin{equation}
    \pi_\infty(\rho_{n_k}(a))_k = \eta^*\circ \rho(a) \quad (a\in A)
\end{equation}
by definition of $\eta$. Fixing a unitary lift $(u_k)_k \in \ell^\infty(B)$ of $u$, we now note that
\begin{equation}
\begin{aligned}
    \lim_{k\to \omega} \| u_k \rho_{n_k}(a) u_k^* - \rho_k(a)\|
        &= \|\pi_\omega\big(u_k \rho_{n_k}(a) u_k^* - \rho_k(a)\big)_k\| \\
        &= \| u (\pi_{\infty\to\omega}\circ \eta^* \circ \rho(a)) u^* - \pi_{\infty\to\omega}\circ\rho(a)\| < \varepsilon
\end{aligned}
\end{equation}
for all $a\in F$. This implies that there exists a $k\geq m$ such that
\begin{equation}
    \| u_k \rho_{n_k}(a) u_k^* - \rho_k(a)\| < \varepsilon
\end{equation}
for all $a\in F$, contradicting \eqref{eq:onesidedintertwining}. Hence, the statement must be true.
\end{proof}

\begin{thm}[Intertwining by reparametrization] \label{thm:reparametrization}
    Let $A$ be a separable and $B$ a unital \ca, and $\rho: A\to B_\infty$ a \sh. Then $\pi_{\infty\to\omega}\circ\rho$ is unitarily equivalent to a \sh\ factoring through $B$ if for any map $\eta: \mathbb N \to \mathbb N$ for which $\lim_{n\to\infty} \eta(n) = \infty$, the maps $\pi_{\infty\to\omega}\circ \rho$ and $\pi_{\infty\to\omega}\circ \eta^* \circ \rho$ are approximately unitarily equivalent.
\end{thm}
\begin{proof}
Let $(F_n)_n$ be an increasing sequence of finite subsets of $A$ such that their union is dense. Fix an arbitrary lift $(\rho_n)_n: A\to \ell^\infty(B)$ of $\rho$, and recursively apply the previous lemma for $F=F_n$, $\varepsilon = 2^{-n}$ and $m= k_{n-1} + 1$ (defining $k_0 := 1$) to obtain a sequence $1= k_0 < k_1 < k_2 < ...$ and a sequence $(u_n)_n$ of unitaries in $B$ such that
\begin{equation}
    \|u_n \rho_{k_{n+1}}(x) u_n^* - \rho_{k_n}(x) \| < 2^{-n}
\end{equation}
for all $n\in\mathbb N$ and $x\in F_n$. Defining $v_n := u_1 \cdots u_n$, we now claim that
\begin{equation}
    \psi := \lim_{n\to\infty} \Ad(v_n)\circ \rho_{k_{n+1}}
\end{equation}
is a well-defined \sh\ from $A$ to $B$ which, when viewed as a map into $B_\infty$, is approximately unitarily equivalent to $\rho$. To see that $\psi$ is well-defined, it suffices to check $(v_n\rho_{k_{n+1}}(a)v_n^*)_n$ is a Cauchy sequence for any $a\in A$. To this end, pick $\varepsilon>0$, pick $N_0\in\mathbb N$ and $b\in F_{N_0}$ such that $\|a-b\| < \varepsilon/3$. By construction, we then find for all $m\geq n\geq N_0$ that
\begin{equation}
    v_m\rho_{k_{m+1}}(b)v_m^* \approx_{2^{-m}} v_{m-1}\rho_{k_m}(b)v_{m-1}^* \approx_{2^{-m+1}} ... \approx_{2^{-n-1}} v_n\rho_{k_{n+1}}(b)v_n^*,
\end{equation}
so
\begin{equation}
    \| v_m\rho_{k_{m+1}}(b)v_m^* - v_n\rho_{k_{n+1}}(b)v_n^* \| \leq \sum_{j=n+1}^m 2^{-j} \leq \sum_{j=n+1}^\infty 2^{-j} = 2^{-n-1}.
\end{equation}
Hence, there exists $N_1 \geq N_0$ such that the above is bounded by $\varepsilon/3$ for all $m\geq n \geq N_1$. Moreover, we have
\begin{equation}
    \limsup_{k\to\infty} \|\rho_k(a)-\rho_k(b)\| = \|\rho(a)-\rho(b)\| < \frac{\varepsilon}{3},
\end{equation}
so since $(k_n)_n$ is a strictly increasing sequence, there exists $N_2 \in\mathbb N$ such that $\|\rho_{k_n}(a)-\rho_{k_n}(b)\|<\varepsilon/3$ for all $n\geq N_2$. Combining all these estimates, we obtain for all $m\geq n \geq \max\{N_1,N_2\}$
\begin{equation}
    \| v_m\rho_{k_{m+1}}(a)v_m^* - v_n\rho_{k_{n+1}}(a)v_n^* \| < \varepsilon,
\end{equation}
proving that $(v_n\rho_{k_{n+1}}(a)v_n^*)_n$ is Cauchy. Hence, $\psi$ is indeed well-defined.

To show that it is approximately unitarily equivalent to $\rho$ in $B_\infty$, let $v \in B_\infty$ be the unitary induced by $(v_n)_n \in \ell^\infty(B)$, and let
\begin{equation}
    \eta:\mathbb N\to \mathbb N: n\mapsto k_{n+1}.
\end{equation}
By construction of $\psi$, it then follows that
\begin{equation}
    \iota_B\circ\psi = \Ad(v)\circ\eta^*\circ\rho,
\end{equation}
so by assumption, $\pi_{\infty\to\omega}\circ\iota_B\circ\psi$ and $\pi_{\infty\to\omega}\circ\rho$ are approximately unitarily equivalent. Since we are working in $B_\omega$, this means that they are unitarily equivalent. Finally, we observe that the composition of $\psi$ with the embedding $\iota_B$ is a \sh, so the same must hold for $\psi$ itself.
\end{proof}

\begin{lemma} \label{lem:jbisoinKK}
    Let $A$ be a separable, nuclear \ca\ and $B$ a \ca\ with a strongly faithful state $\rho$ such that $\pi_\rho(B)'' \cong B(\Hb)$ for some separable Hilbert space $\Hb$. Then the inclusion $j_B^{(\omega)}: J_{B,\rho}^{(\omega)}\hookrightarrow S_{B,\rho}^{(\omega)}$ induces an isomorphism $KK(A,j_B^{(\omega)}): KK(A,J_{B,\rho}^{(\omega)})\to KK(A,S_{B,\rho}^{(\omega)})$. Subsequently, the same holds for $KL(A,j_B^{(\omega)})$.
\end{lemma}
\begin{proof}
As $A$ is nuclear, the reduced state-kernel extension induces a six-term exact sequence
\begin{equation}
\begin{tikzcd}[column sep=large]
    KK(A,J_{B,\rho}^{(\omega)}) \ar[r,"{KK(A,j_B^{(\omega)})}"] & KK(A,S_{B,\rho}^{(\omega)}) \ar[r] & KK(A,B(\Hb)) \ar[d]\\
    KK^1(A,B(\Hb)) \ar[u] & KK^1(A,S_{B,\rho}^{(\omega)}) \ar[l] & KK^1(A,J_{B,\rho}^{(\omega)}) \ar[l]
\end{tikzcd}
\end{equation}
in $KK$ (see \cite[Theorem 19.5.7]{Blackadar_KTheoryForOperatorAlgebras}). By \cite[Proposition 4.1]{Dadarlat_ApproximateUnitaryEquivalenceAndTheTopologyOfExt}, it follows that both $KK(A,B(\Hb))$ and $KK^1(A,B(\Hb)) \cong KK(SA,B(\Hb))$ are trivial, which results in $KK(A,j_B^{(\omega)})$ being an isomorphism. The same holds for $KL(A,j_B^{(\omega)})$ by naturality of $KL$.
\end{proof}

\begin{defn}
    Let $A$ and $B$ be \ca s with $A$ separable. We define
    \begin{equation}
        \widetilde{KL}(A,B) := KL(A,B)/\ker KL(A,\iota_B^{(\omega)}),
    \end{equation}
    where $\iota_B^{(\omega)}: B\to B_\omega$ is the canonical inclusion. Given a \sh\ $\varphi:A\to B$, we will denote $[\varphi]_{\widetilde{KL}(A,B)}$ for the element in $\widetilde{KL}(A,B)$ induced by $[\varphi]_{KL(A,B)}$. We will say an element $\tilde\lambda\in \widetilde{KL}(A,B)$ is \emph{unital} if there exists a lift $\lambda \in KL(A,B)$ of $\tilde\lambda$ such that $\overline\Gamma_0^{(A,B)}(\lambda)[1_A]_0 = [1_B]_0$.
\end{defn}

\begin{thm}[Classification of unital embeddings] \label{thm:ClassifyingEmbeddings}
    Let $A$ be a unital, separable, nuclear \ca, and $B$ a unital, simple, separable and purely infinite \ca. For any unital $\tilde \lambda \in \widetilde{KL}(A,B)$, there exists a unital, injective \sh\ $\varphi: A\to B$ inducing $\tilde \lambda$, and such a $\varphi$ is unique up to approximate unitary equivalence.
\end{thm}
\begin{proof}
As $B$ is simple, it has a faithful irreducible representation. From Corollary \ref{cor:stronglyfaithfulstateexists}, we then obtain a strongly faithful state $\rho$ on $B$ such that $\pi_\rho(B)''\cong B(\Hb)$ for some separable Hilbert space $\Hb$. Hence, we can apply classification of maps into $S_{B,\rho}$ for this state.

Existence: pick $\tilde\lambda$ as in the statement, and let $\lambda$ be a lift of $\tilde\lambda$ which preserves the $K_0$-class of the unit. Fix a lift $\kappa \in KK(A,B)$ of this $\lambda$ and let $\iota: B\hookrightarrow S_{B,\rho}$ be the corestriction of the canonical inclusion of $B$ into $B_\infty$ to $S_{B,\rho}$. As $\Gamma_\Lambda^{(A,B)}$ factors through $KL(A,B)$ by Theorem \ref{thm:KKactiononK}, we also have $\Gamma_0^{(A,B)}(\kappa)[1_A]_0 = [1_B]_0$, so by naturality of $\Gamma_\Lambda^{(A,\cdot)}$, we obtain
\begin{equation}
    \Gamma_0^{(A,S_{B,\rho})}\big(KK(A,\iota)(\kappa)\big)[1_A]_0 = [1_{S_{B,\rho}}]_0.
\end{equation}
By the existence part of Theorem \ref{thm:ClassifyingApproximateMaps}, there exists a unital, full \sh\ $\varphi_\infty: A\to S_{B,\rho}$ inducing the class $KK(A,\iota)(\kappa)$ in $KK$. Descending to $KL$, it then satisfies
\begin{equation} \label{eq:classifyingembeddings}
    [\varphi_\infty]_{KL(A,S_{B,\rho})} = KL(A,\iota)(\lambda).
\end{equation}
Now, we aim to apply intertwining by reparametrization (Theorem \ref{thm:reparametrization}). To this end, pick $\eta$ as in the statement of Theorem \ref{thm:reparametrization}, and note that $\eta^*\circ\varphi_\infty$ still maps into $S_{B,\rho}$. Moreover, we have
\begin{equation}
\begin{aligned}
    [\pi_{\infty\to\omega}\circ \eta^*\circ\varphi_\infty]_{KK(A,S_{B,\rho}^{(\omega)})}
        &= KK(A, \pi_{\infty\to\omega}\circ\eta^*\circ \iota)(\kappa)\\
        &= KK(A, \pi_{\infty\to\omega}\circ \iota)(\kappa)\\
        &=  [\pi_{\infty\to\omega}\circ\varphi_\infty]_{KK(A,S_{B,\rho}^{(\omega)})},
\end{aligned}
\end{equation}
by functoriality and $\eta^*\circ \iota = \iota$. Hence, the classification of maps into $S_{B,\rho}$ implies that $\pi_{\infty\to\omega}\circ \eta^*\circ\varphi_\infty$ and $\pi_{\infty\to\omega}\circ\varphi_\infty$ are unitarily equivalent as maps into $B_\omega$. By Theorem \ref{thm:reparametrization}, we then obtain a \sh\ $\varphi: A\to B$ such that $\pi_{\infty\to\omega}\circ\varphi_\infty \sim_u \iota_B^{(\omega)}\circ \varphi$ in $B_\omega$. We now claim that $\varphi$ is the desired map. To see this, note that it is unital because $\pi_{\infty\to\omega}\circ\varphi_\infty$ and $\iota_B^{(\omega)}$ are, and injective because $\pi_{\infty\to\omega}\circ\varphi_\infty$ is full. We also have
\begin{equation}
\begin{aligned}
    KL(A, &\iota_B^{(\omega)})([\varphi]_{KL(A,B)} -\lambda) \\
        &= [\pi_{\infty\to\omega}\circ\varphi_\infty]_{KL(A,B_\omega)} - KL(A,\pi_{\infty\to\omega}|_{S_{B,\rho}}\circ\iota)(\lambda) \\
        &= KL(A,\pi_{\infty\to\omega}|_{S_{B,\rho}})\big([\varphi_\infty]_{KL(A,S_{B,\rho})} - KL(A,\iota)(\lambda)\big) = 0
\end{aligned}
\end{equation}
by functoriality, \eqref{eq:classifyingembeddings} and $\iota_B^{(\omega)} = \pi_{\infty\to\omega}|_{S_{B,\rho}}\circ\iota$ (which follows from \eqref{eq:ske_inftyvsomegacommdiagram}). This tells us exactly that $[\varphi]_{\widetilde{KL}(A,B)} = \tilde\lambda$ in $\widetilde{KL}(A,B)$, as desired.

Uniqueness: suppose two unital, injective (and therefore full, by simplicity of $B$) \sh s $\varphi_1,\varphi_2: A \to B$ satisfy $[\varphi_1]_{\widetilde{KL}(A,B)} = [\varphi_2]_{\widetilde{KL}(A,B)}$. Observe that $KL(A,\iota_B^{(\omega)})$ naturally induces a monomorphism
\begin{equation}
    \widetilde{KL}(A,B) \to KL(A,B_\omega): [\lambda] \mapsto KL(A,\iota_B^{(\omega)})(\lambda)
\end{equation}
by the first isomorphism theorem for groups, so in fact
\begin{equation}
    [\iota_B^{(\omega)}\circ\varphi_1]_{KL(A,B_\omega)} = [\iota_B^{(\omega)}\circ\varphi_2]_{KL(A,B_\omega)}.
\end{equation}
Decomposing $\iota_B^{(\omega)}$ as $\iota_{S_{B,\rho}^{(\omega)}\subseteq B_\omega}\circ \pi_{\infty\to\omega}\circ\iota$ (using \eqref{eq:ske_inftyvsomegacommdiagram}), we can use functoriality to rewrite this as
\begin{equation} \label{eq:classifyingembeddings2}
    KL(A,\iota_{S_{B,\rho}^{(\omega)}\subseteq B_\omega}) \left( [\pi_{\infty\to\omega}\circ\iota\circ\varphi_1]_{KL(A,S_{B,\rho}^{(\omega)})} -[\pi_{\infty\to\omega}\circ\iota\circ\varphi_2]_{KL(A,S_{B,\rho}^{(\omega)})} \right) = 0.
\end{equation}
Now note that the map $KL(A,\iota_{S_{B,\rho}^{(\omega)}\subseteq B_\omega})$ is an isomorphism, as both $KL(A,j_B^{(\omega)})$ and
\begin{equation}
    KL(A, \iota_{J_B^{(\omega)}\subseteq B_\omega}) = KL(A,\iota_{S_{B,\rho}^{(\omega)}\subseteq B_\omega}) \circ KL(A,j_B^{(\omega)})
\end{equation}
are isomorphisms by Lemmas \ref{lem:jbisoinKK} and \ref{lem:hereditaryisoinKK}. Hence, \eqref{eq:classifyingembeddings2} implies that
\begin{equation}
    [\pi_{\infty\to\omega}\circ\iota\circ\varphi_1]_{KL(A,S_{B,\rho}^{(\omega)})} = [\pi_{\infty\to\omega}\circ\iota\circ\varphi_2]_{KL(A,S_{B,\rho}^{(\omega)})},
\end{equation}
so the maps $\iota \circ \varphi_1$ and $\iota \circ \varphi_2$ satisfy the uniqueness criterion in classification of maps into $S_{B,\rho}$. Therefore, $\pi_{\infty\to\omega}\circ\iota \circ \varphi_1$ and $\pi_{\infty\to\omega}\circ\iota \circ \varphi_2$ are unitarily equivalent as maps into $B_\omega$. In other words, we have $\iota_B^{(\omega)}\circ\varphi_1 \sim_u \iota_B^{(\omega)}\circ\varphi_2$, from which we can conclude that $\varphi_1$ and $\varphi_2$ are approximately unitarily equivalent by \cite[Lemma 6.2.5]{Rordam_ClassificationofNuclearCalgebras}.
\end{proof}

\subsection{Corollaries}
\begin{thm} \label{thm:ClassifyingEmbeddingsbyK}
    Let $A$ be a unital, separable, nuclear \ca\ satisfying the UCT, and $B$ a unital, simple, separable and purely infinite \ca. Then each element of $\Hom_\Lambda(\totalK(A), \totalK(B))$ preserving the $K_0$-class of the unit is induced by a unital, injective \sh\ from $A$ to $B$, and uniquely so up to approximate unitary equivalence.
\end{thm}
\begin{proof}
If $A$ satisfies the UCT, then Theorem \ref{thm:KKactiononK} tells us that the natural map $\overline{\Gamma}_\Lambda^{(A,B)}: KL(A,B)\to \Hom_\Lambda(\totalK(A), \totalK(B))$ is an isomorphism. Note that the map $KL(A,\iota_B^{(\omega)}): KL(A,B)\to KL(A,B_\omega)$ then corresponds to
\begin{equation}
    \Hom_\Lambda(\totalK(A), \totalK(B)) \to \Hom_\Lambda(\totalK(A), \totalK(B_\omega)): \alpha \mapsto \totalK(\iota_B^{\omega)})\circ\alpha.
\end{equation}
By using the same argument as in \cite[Lemma 3.16]{CarrionGabeSchafhauserTikuisisWhite_ClassifyingHomomorphismsIUnitalSimpleNuclearCalgebras} for $\iota_B^{(\omega)}$, it follows that $\totalK(\iota_B^{(\omega)})$ is injective, so the same must hold for the corresponding map $KL(A,\iota_B^{(\omega)})$. Therefore, we have $\widetilde{KL}(A,B) = KL(A,B)$, which is naturally isomorphic with $\Hom_\Lambda(\totalK(A), \totalK(B))$ as mentioned earlier. The result now follows from Theorem \ref{thm:ClassifyingEmbeddings}.
\end{proof}

\begin{thm}\label{thm:KPKtheory}
    Let $A$ and $B$ be two unital Kirchberg algebras satisfying the UCT. Then $A\cong B$ if and only if $(K_*(A),[1_A]_0) \cong (K_*(B),[1_B]_0)$. Moreover, for each isomorphism $\alpha_*: (K_*(A),[1_A]_0) \to (K_*(B),[1_B]_0)$, there exists an isomorphism $\varphi: A\to B$ with $K_*(\varphi) = \alpha_*$.
\end{thm}
\begin{proof}
The ``only if''-statement is immediate. For the ``if''-statement, suppose $\alpha_*: (K_*(A),[1_A]_0) \to (K_*(B),[1_B]_0)$ is an isomorphism. By \cite[Proposition 2.15]{CarrionGabeSchafhauserTikuisisWhite_ClassifyingHomomorphismsIUnitalSimpleNuclearCalgebras}, there exists an isomorphism $\underline{\alpha}: \totalK(A)\to \totalK(B)$ extending $\alpha_*$. The existence part of previous corollary then implies that there exist unital \sh s $\varphi_0: A\to B$ and $\psi_0: B\to A$ such that $\totalK(\varphi_0) = \underline{\alpha}$ and $\totalK(\psi_0) = \underline{\alpha}^{-1}$. Subsequently, we have $\totalK(\psi_0\circ\varphi_0) = \id_{\totalK(A)} = \totalK(\id_A)$ and similarly $\totalK(\varphi_0\circ\psi_0) = \totalK(\id_B)$. The uniqueness part then tells us that $\psi_0\circ\varphi_0\approx_u\id_A$ and $\varphi_0\circ\psi_0\approx_u\id_B$, so by two-sided intertwining (\cite[Corollary 2.3.4]{Rordam_ClassificationofNuclearCalgebras}), there exists an isomorphism $\varphi: A\to B$ with $\varphi\approx_u \varphi_0$ and $\varphi^{-1}\approx_u \psi_0$. As approximately unitarily equivalent \sh s agree on the level of $K_*$, it follows that
\begin{equation}
    K_*(\varphi) = K_*(\varphi_0) = \alpha_*,
\end{equation}
as desired.
\end{proof}

As an immediate consequence we get the following nuclear version of Kirchberg's $\mathcal O_2$-embedding theorem.

\begin{cor} \label{cor:Otwoembedding}
    Any unital, separable and nuclear \ca\ embeds unitally into $\Otwo$.
\end{cor}
\begin{proof}
Note that $0\in KL(A,\Otwo)$ satisfies
\begin{equation}
    \overline\Gamma_0^{(A,\Otwo)}(0)[1_A]_0 = 0 = [1_{\Otwo}]_0
\end{equation}
as $K_0(\Otwo)=0$ (\cite[Theorem 3.7]{Cuntz_KTheoryforCertainCalgebras}). Hence, Theorem \ref{thm:ClassifyingEmbeddings} implies that there exists a unital, injective \sh\ $\iota: A\to \Otwo$ inducing $0$.
\end{proof}

\begin{cor} \label{cor:Otwoabsorption}
Let $A$ be a unital, separable, unital, simple $C^\ast$-algebra. Then $\mathcal O_2 \cong A \otimes \mathcal O_2$. Moreover, any such isomorphism is approximately unitarily equivalent to the first factor inclusion $\mathrm{id}_{\mathcal O_2} \otimes 1_A \colon \mathcal O_2 \to \mathcal O_2 \otimes A$. In particular, $\mathcal O_2$ is strongly self-absorbing. 
\end{cor}
\begin{proof}
    Since both $\mathcal O_2$ and $\mathcal O_2 \otimes A$ are unital Kirchberg algebras which are $KK$-equivalent to $0$, this follows from Theorems \ref{thm:KPKtheory} and \ref{thm:ClassifyingEmbeddingsbyK}. The final part follows by applying the result to $A= \mathcal O_2$.
\end{proof}

The classical proof of the Kirchberg--Phillips theorem implies that if two unital Kirchberg algebras are unitally $KK$-equivalent, then they are isomorphic. While the methods we have used does not imply this (since we do not know that the Kasparov product preserves $\widetilde{KL}$) we draw inspiration from Schafhauser's recent theorem \cite{Schafhauser_KKRigidityOfSimpleNuclearCalgebras} for classification of stably finite $C^\ast$-algebras without assuming the UCT, and prove this result provided there exists a $\ast$-homomorphism inducing the $KK$-equivalence. As proved above, we do not need this additional assumption under UCT assumptions.

\begin{prop}\label{p:class}
    Let $A$ and $B$ be unital Kirchberg algebras and suppose $\phi \colon A \to B$ is a unital $\ast$-homomorphism which is a $KK$-equivalence. Then $\phi$ is approximately unitarily equivalent to an isomorphism.
\end{prop}
\begin{proof}
By Theorem \ref{thm:ClassifyingEmbeddings} there is a unital $\ast$-homomorphism $\psi \colon B \to A$ such that $(\iota_A^{(\omega)})_\ast([\psi]_{KL}) = (\iota_A^{(\omega)})_\ast([\phi]^{-1}_{KL})$. By \cite[Proposition 18.7.1]{Blackadar_KTheoryForOperatorAlgebras}, we get that
\begin{equation}
\begin{aligned}
    (\iota_A^{(\omega)})_\ast([\psi \circ \phi]_{KL})
        &= (\iota_A^{(\omega)})_\ast([\psi]_{KL}) \circ [\phi]_{KL} = (\iota_A^{(\omega)})_\ast([\phi]^{-1}_{KL}) \circ [\phi]_{KL}\\
        &= (\iota_A^{(\omega)})_\ast([\phi]_{KL}^{-1} \circ [\phi]_{KL}) = (\iota_A^{(\omega)})_\ast([\mathrm{id}_A]_{KL}). 
\end{aligned}
\end{equation}
Hence $\psi \circ \phi \approx_u \mathrm{id}_A$ by Theorem \ref{thm:ClassifyingEmbeddings}. Moreover,\footnote{This is where we use that $\phi$ is a $\ast$-homomorphism.} since $\iota_B^{(\omega)} \circ \phi = \phi_\omega \circ \iota_A^{(\omega)}$, we get that
\begin{equation}
\begin{aligned}
    (\iota_B^{(\omega)})_\ast([\phi \circ \psi]_{KL})
        &= (\phi_\omega)_\ast((\iota_A^{(\omega)})_\ast([\psi]_{KL})) = (\phi_\omega)_\ast((\iota_A^{(\omega)})_\ast([\phi]^{-1}_{KL})) \\
        &= (\iota_B^{(\omega)})_\ast([\phi]_{KL} \circ [\phi]_{KL}^{-1}) = (\iota_B^{(\omega)})_\ast([\mathrm{id}_B), 
\end{aligned}
\end{equation}
so that $\phi \circ \psi \approx_u \mathrm{id}_B$. By a two-sided intertwining argument, $\phi$ is approximately unitarily equivalent to an isomorphism. 
\end{proof}

As an immediate consequence we get the following.

\begin{cor}
    Any two unital Kirchberg algebras which are homotopic are isomorphic.
\end{cor}

Finally, we obtain Kirchberg's Geneva Theorem stating that Kirchberg algebras are $\mathcal O_\infty$-stable. At the same time, we get that $\mathcal O_\infty$ is strongly self-absorbing.

\begin{cor} \label{cor:Oinfabsorption}
    Any unital Kirchberg algebra $A$ is $\Oinf$-stable, i.e.\ $A\cong A\otimes \Oinf$. Moreover, the isomorphism can be chosen such that it is approximately unitarily equivalent with the first-factor embedding $\id_A\otimes 1_{\Oinf}: A\to A\otimes \Oinf$. In particular, $\Oinf$ is strongly self-absorbing.
\end{cor}
\begin{proof}
By \cite[Corollary 3.11]{Cuntz_KTheoryforCertainCalgebras}, the canonical unital map $\iota: \mathbb C \to \Oinf$ induces a $KK$-equivalence. Hence, $\id_A\otimes 1_{\Oinf}$ induces a $KK$-equivalence, as
\begin{equation}
    \id_A\otimes 1_{\Oinf} = (\id_A\otimes \iota) \circ (\id_A\otimes 1_{\mathbb C})
\end{equation}
and $\id_A\otimes 1_{\mathbb C}: A\to A\otimes \mathbb C$ is an isomorphism. The result then follows from Proposition \ref{p:class}. The particular consequence is obtained by applying the result to $A=\Oinf$.
\end{proof}


\begin{thebibliography}{10}

\bibitem{AnantharamanDelaroche_PurelyInfiniteCalgebrasArisingFromDynamicalSystems}
C.~Anantharaman-Delaroche.
\newblock Purely infinite {$C^*$}-algebras arising from dynamical systems.
\newblock {\em Bull. Soc. Math. France}, 125(2):199--225, 1997.

\bibitem{AndoHaagerup_UltraproductsOfVonNeumannAlgebras}
H.~Ando and U.~Haagerup.
\newblock Ultraproducts of von {N}eumann algebras.
\newblock {\em J. Funct. Anal.}, 266(12):6842--6913, 2014.

\bibitem{Blackadar_KTheoryForOperatorAlgebras}
B.~Blackadar.
\newblock {\em {$K$}-theory for operator algebras}, volume~5 of {\em Mathematical Sciences Research Institute Publications}.
\newblock Cambridge University Press, Cambridge, second edition, 1998.

\bibitem{Blackadar_OperatorAlgebrasTheoryofCalgebrasandvonNeumannAlgebras}
B.~Blackadar.
\newblock {\em Operator Algebras: Theory of $C^*$-Algebras and von Neumann Algebras}, volume 122 of {\em Encyclopaedia of Mathematical Sciences}.
\newblock Springer Berlin Heidelberg, Berlin, 2006.

\bibitem{BrownOzawa_CalgebrasandFiniteDimensionalApproximations}
N.~P. Brown and N.~Ozawa.
\newblock {\em {$C^*$}-algebras and finite-dimensional approximations}, volume~88 of {\em Graduate Studies in Mathematics}.
\newblock American Mathematical Society, Providence, RI, 2008.

\bibitem{CarrionGabeSchafhauserTikuisisWhite_ClassifyingHomomorphismsIUnitalSimpleNuclearCalgebras}
J.~R. Carrión, J.~Gabe, C.~Schafhauser, A.~Tikuisis, and S.~White.
\newblock Classifying $^*$-homomorphisms {I}: Unital simple nuclear {$C^*$}-algebras. arXiv:2307.06480v3.

\bibitem{CastillejosEvingtonTikuisisWhite_ClassifyingMapsIntoUniformTracialSequenceAlgebras}
J.~Castillejos, S.~Evington, A.~Tikuisis, and S.~White.
\newblock Classifying maps into uniform tracial sequence algebras.
\newblock {\em M\"unster J. Math.}, 14(2):265--281, 2021.

\bibitem{ChoiEffros_TheCompletelyPositiveLiftingProblemForCalgebras}
M.~D. Choi and E.~G. Effros.
\newblock The completely positive lifting problem for {$C\sp*$}-algebras.
\newblock {\em Ann. of Math. (2)}, 104(3):585--609, 1976.

\bibitem{Cuntz_SimpleCalgebrasGeneratedbyIsometries}
J.~Cuntz.
\newblock Simple {$C^*$}-algebras generated by isometries.
\newblock {\em Comm. Math. Phys.}, 57(2):173--185, 1977.

\bibitem{Cuntz_KTheoryforCertainCalgebras}
J.~Cuntz.
\newblock {$K$}-theory for certain {$C\sp{\ast} $}-algebras.
\newblock {\em Ann. of Math. (2)}, 113(1):181--197, 1981.

\bibitem{Dadarlat_ApproximateUnitaryEquivalenceAndTheTopologyOfExt}
M.~Dadarlat.
\newblock Approximate unitary equivalence and the topology of {${\rm Ext}(A,B)$}.
\newblock In {\em {$C^*$}-algebras ({M}\"{u}nster, 1999)}, pages 42--60. Springer, Berlin, 2000.

\bibitem{Elliott_OnTheClassificationOfCalgebrasOfRealRankZero}
G.~A. Elliott.
\newblock On the classification of {$C^*$}-algebras of real rank zero.
\newblock {\em J. Reine Angew. Math.}, 443:179--219, 1993.

\bibitem{ElliottGongLinNiu_OnTheClassificationOfSimpleAmenableCalgebrasWithFiniteDecompositionRank}
G.~A. Elliott, G.~Gong, H.~Lin, and Z.~Niu.
\newblock On the classification of simple amenable {$C^*$}-algebras with finite decomposition rank, {II}. arXiv:1507.03437v2.

\bibitem{ElliottKucerovsky_AnAbstractVoiculescuBrownDouglasFillmoreAbsorptionTheorem}
G.~A. Elliott and D.~Kucerovsky.
\newblock An abstract {V}oiculescu-{B}rown-{D}ouglas-{F}illmore absorption theorem.
\newblock {\em Pacific J. Math.}, 198(2):385--409, 2001.

\bibitem{Gabe_ANewProofOfKirchbergsO2StableClassification}
J.~Gabe.
\newblock A new proof of {K}irchberg's {$\mathcal O_2$}-stable classification.
\newblock {\em J. Reine Angew. Math.}, 761:247--289, 2020.

\bibitem{Gabe_ClassificationofOinftystableCalgebras}
J.~Gabe.
\newblock Classification of {$\Oinf$}-stable {$C^*$}-algebras.
\newblock {\em Mem. Amer. Math. Soc.}, 293(1461):v+115, 2024.

\bibitem{GabeRuiz_TheUnitalExtGroupsAndClassificationOfCalgebras}
J.~Gabe and E.~Ruiz.
\newblock The unital {E}xt-groups and classification of {$C^*$}-algebras.
\newblock {\em Glasg. Math. J.}, 62(1):201--231, 2020.

\bibitem{GongLinNiu_AClassificationOfFiniteSimpleAmenableZstableCalgebrasI}
G.~Gong, H.~Lin, and Z.~Niu.
\newblock A classification of finite simple amenable {$\Z$}-stable {$C^\ast$}-algebras, {I}: {$C^\ast$}-algebras with generalized tracial rank one.
\newblock {\em C. R. Math. Acad. Sci. Soc. R. Can.}, 42(3):63--450, 2020.

\bibitem{GongLinNiu_AClassificationOfFiniteSimpleAmenableZstableCalgebrasII}
G.~Gong, H.~Lin, and Z.~Niu.
\newblock A classification of finite simple amenable {$\Z$}-stable {$C^*$}-algebras, {II}: {$C^*$}-algebras with rational generalized tracial rank one.
\newblock {\em C. R. Math. Acad. Sci. Soc. R. Can.}, 42(4):451--539, 2020.

\bibitem{HjelmborgRordam_OnStabilityofCalgebras}
J.~v.~B. Hjelmborg and M.~R{\o}rdam.
\newblock On stability of {$C^*$}-algebras.
\newblock {\em J. Funct. Anal.}, 155(1):153--170, 1998.

\bibitem{Kirchberg_TheClassificationofPurelyInfiniteCalgebrasUsingKasparovsTheory}
E.~Kirchberg.
\newblock The classification of purely infinite {$C^*$}-algebras using {K}asparov's theory.
\newblock {\em Unpublished preprint}, 1994.

\bibitem{Kirchberg_ExactCalgebrasTensorProductsAndTheClassificationOfPurelyInfiniteAlgebras}
E.~Kirchberg.
\newblock Exact {$C^*$}-algebras, tensor products, and the classification of purely infinite algebras.
\newblock In {\em Proceedings of the {I}nternational {C}ongress of {M}athematicians, {V}ol.\ 1, 2 ({Z}\"urich, 1994)}, pages 943--954. Birkh\"auser, Basel, 1995.

\bibitem{Kirchberg_EmbeddingOfExactCalgebrasInTheCuntzAlgebraOtwo}
E.~Kirchberg and N.~C. Phillips.
\newblock Embedding of exact {$C^*$}-algebras in the {C}untz algebra {$\Otwo$}.
\newblock {\em J. Reine Angew. Math.}, 525:17--53, 2000.

\bibitem{KirchbergRordam_NonSimplePurelyInfiniteCalgebras}
E.~Kirchberg and M.~R{\o}rdam.
\newblock Non-simple purely infinite {$C^\ast$}-algebras.
\newblock {\em Amer. J. Math.}, 122(3):637--666, 2000.

\bibitem{KirchbergRordam_InfiniteNonSimpleCalgebrasAbsorbingTheCuntzAlgebrasOinfty}
E.~Kirchberg and M.~R{\o}rdam.
\newblock Infinite non-simple {$C^*$}-algebras: absorbing the {C}untz algebras {$\mathcal O_\infty$}.
\newblock {\em Adv. Math.}, 167(2):195--264, 2002.

\bibitem{LoreauxNg_RemarksOnEssentialCodimension}
J.~Loreaux and P.~W. Ng.
\newblock Remarks on essential codimension.
\newblock {\em Integral Equations Operator Theory}, 92(1):Paper No. 4, 35, 2020.

\bibitem{OcneanuActionsOfDiscreteAmenableGroupsOnVonNeumannAlgebras}
A.~Ocneanu.
\newblock {\em Actions of discrete amenable groups on von {N}eumann algebras}, volume 1138 of {\em Lecture Notes in Mathematics}.
\newblock Springer-Verlag, Berlin, 1985.

\bibitem{Phillips_AClassificationTheoremforNuclearPurelyInfiniteSimpleCalgebras}
N.~C. Phillips.
\newblock A classification theorem for nuclear purely infinite simple {$C^*$}-algebras.
\newblock {\em Doc. Math.}, 5:49--114, 2000.

\bibitem{Rordam_ClassificationofCertainInfiniteSimpleCAlgebras}
M.~R{\o}rdam.
\newblock Classification of certain infinite simple {$C^*$}-algebras.
\newblock {\em J. Funct. Anal.}, 131(2):415--458, 1995.

\bibitem{Rordam_ClassificationofNuclearCalgebras}
M.~R{\o}rdam.
\newblock Classification of nuclear, simple {$C^*$}-algebras.
\newblock In {\em Classification of nuclear {$C^*$}-algebras. {E}ntropy in operator algebras}, volume 126 of {\em Encyclopaedia Math. Sci.}, pages 1--145. Springer, Berlin, 2002.

\bibitem{Rordam_ASimpleCalgebraWithAFiniteAndAnInfiniteProjection}
M.~R{\o}rdam.
\newblock A simple {$C^*$}-algebra with a finite and an infinite projection.
\newblock {\em Acta Math.}, 191(1):109--142, 2003.

\bibitem{Schafhauser_SubalgebrasOfSimpleAFAlgebras}
C.~Schafhauser.
\newblock Subalgebras of simple {AF}-algebras.
\newblock {\em Ann. of Math. (2)}, 192(2):309--352, 2020.

\bibitem{Schafhauser_KKRigidityOfSimpleNuclearCalgebras}
C.~Schafhauser.
\newblock {KK}-rigidity of simple nuclear {$C^*$}-algebras. arXiv:2408.02745v2.

\bibitem{TikuisisWhiteWinter_QuasidiagonalityOfNuclearCalgebras}
A.~Tikuisis, S.~White, and W.~Winter.
\newblock Quasidiagonality of nuclear {$C^*$}-algebras.
\newblock {\em Ann. of Math. (2)}, 185(1):229--284, 2017.

\end{thebibliography}
\end{document}